\documentclass[runningheads]{llncs}

\usepackage{graphicx}
\usepackage{mathtools}
\DeclarePairedDelimiter{\group}{(}{)}
\DeclarePairedDelimiter{\sqgroup}{[}{]}
\DeclarePairedDelimiter{\set}{\{}{\}}
\DeclarePairedDelimiter{\sett}{\{}{\}}

\DeclarePairedDelimiter{\abs}{\vert}{\vert}

\newcommand{\xqed}[1]{%
  \leavevmode\unskip\penalty9999 \hbox{}\nobreak\hfill
  \quad\hbox{#1}}

\newenvironment{graytext}{\color{gray}}{\ignorespacesafterend}

\newcommand{\propositionref}[1]{Proposition~\ref{#1}}
\newcommand{\lemmaref}[1]{Lemma~\ref{#1}}
\newcommand{\theoremref}[1]{Theorem~\ref{#1}}
\newcommand{\equationref}[1]{Equation~\eqref{#1}}
\newcommand{\definitionref}[1]{Definition~\ref{#1}}
\newcommand{\corollaryref}[1]{Corollary~\ref{#1}}
\newcommand{\figureref}[1]{Figure~\ref{#1}}
\newcommand{\exampleref}[1]{Example~\ref{#1}}
\newcommand{\propertyref}[1]{Property~\ref{#1}}

\newcommand{\pths}{\Omega}
\newcommand{\posspace}{\mathcal{X}}

\newcommand{\intervalpq}{\sqgroup{p,q}}

\newcommand{\random}{\textnormal{R}}
\newcommand{\randomI}{I_\textnormal{R}(\pth)}
\newcommand{\frcstI}{I_\frcstsystem(\pth)}
\newcommand{\ch}{\textnormal{CH}}
\newcommand{\wch}{\textnormal{wCH}}
\newcommand{\co}{\textnormal{C}}
\newcommand{\s}{\textnormal{S}}
\newcommand{\ml}{\textnormal{ML}}
\newcommand{\wml}{\textnormal{wML}}
\newcommand{\churchI}{I_\textnormal{CH}(\pth)} 
\newcommand{\weakchurchI}{I_\wch(\pth)}
\newcommand{\compI}{I_\co(\pth)}
\newcommand{\mlI}{I_\ml(\pth)}
\newcommand{\weakmlI}{I_\wml(\pth)}
\newcommand{\schnorrI}{I_\s(\pth)}

\newcommand{\intervals}{\mathcal{I}}
\newcommand{\randomintervals}{\mathcal{I}_\random(\pth)}
\newcommand{\mlintervals}{\mathcal{I}_\ml(\pth)}
\newcommand{\weakmlintervals}{\mathcal{I}_\wml(\pth)}
\newcommand{\compintervals}{\mathcal{I}_\co(\pth)}
\newcommand{\churchintervals}{\mathcal{I}_\ch(\pth)}
\newcommand{\weakchurchintervals}{\mathcal{I}_\wch(\pth)}
\newcommand{\schnorrintervals}{\mathcal{I}_\s(\pth)}
\newcommand{\randompths}{\pths_\random(\frcstsystem)}
\newcommand{\mlpths}{\pths_\ml(\frcstsystem)}
\newcommand{\weakmlpths}{\pths_\wml(\frcstsystem)}
\newcommand{\comppths}{\pths_\co(\frcstsystem)}
\newcommand{\schnorrpths}{\pths_\s(\frcstsystem)}
\newcommand{\churchpths}{\pths_\ch(\frcstsystem)}
\newcommand{\weakchurchpths}{\pths_\wch(\frcstsystem)}

\newcommand{\sits}{\mathbb{S}}
\newcommand{\map}{\lambda}
\newcommand{\outcomes}{\sett{0,1}}

\newcommand{\frcstsystem}{\varphi}

\newcommand{\lfrcstsystem}{\underline{\frcstsystem}}
\newcommand{\ufrcstsystem}{\overline{\frcstsystem}}
\newcommand{\frcstsystems}{\Phi}

\newcommand{\process}{F}

\newcommand{\selection}{S}

\newcommand{\selectionsum}{\sum_{k=0}^{n-1}\selection(\pthtok)}

\newcommand{\selectionsdense}{\mathcal{S}_\textnormal{R}(\pth)}
\newcommand{\selectionsfdense}{\Sigma_\textnormal{R}(\pth)}

\newcommand{\multprocess}{D}
\newcommand{\mint}[1][\multprocess]{#1^{\raisebox{0.5pt}{$\scriptstyle \circledcirc$}}}

\newcommand{\supermartin}{M}
\newcommand{\realgrowth}{\tau}

\newcommand{\allowables}{\mathbb{A}}
\newcommand{\callowables}{\mathbb{A^+_{\textrm{C}}}}
\newcommand{\mlallowables}{\mathbb{A^+_{\textrm{ML}}}}

\newcommand{\tests}[1]{\overline{\mathbb{T}}(#1)}

\newcommand{\mltests}[1]{\overline{\mathbb{T}}_\ml(#1)}
\newcommand{\weakmltests}[1]{\overline{\mathbb{T}}_\wml(#1)}
\newcommand{\comptests}[1]{\overline{\mathbb{T}}_\co(#1)}
\newcommand{\schnorrtests}[1]{\overline{\mathbb{T}}_\s(#1)}
\newcommand{\testsupermartins}[1]{\overline{\mathbb{T}}(#1)}
\newcommand{\testscomp}[1]{\overline{\mathbb{T}}_{\textnormal{C}}(#1)}
\newcommand{\testsml}[1]{\overline{\mathbb{T}}_\textnormal{ML}(#1)}
\newcommand{\testswml}[1]{\overline{\mathbb{T}}_\textnormal{wML}(#1)}
\newcommand{\testsschnorr}[1]{\overline{\mathbb{T}}_\textnormal{S}(#1)}
\newcommand{\test}{T}

\newcommand{\callowabletests}[1]{\overline{\mathbb{M}}_{\mathrm{C}}(#1)}

\newcommand{\naturals}{\mathbb{N}}
\newcommand{\naturalswithzero}{{\mathbb{N}_0}}

\newcommand{\reals}{\mathbb{R}}
\newcommand{\posreals}{\mathbb{R}_{>0}}
\newcommand{\nonnegreals}{\mathbb{R}_{\geq0}}

\newcommand{\rationals}{\mathbb{Q}}
\newcommand{\posrationals}{\mathbb{Q}_{>0}}
\newcommand{\nonnegrationals}{\mathbb{Q}_{\geq0}}

\newcommand{\ex}{E}
\newcommand{\lex}{\underline{\ex}}

\newcommand{\uex}{\overline{\ex}}

\newcommand{\gambleinterval}[1][]{\sqgroup{\underline{\gamble},\overline{\gamble}}}

\newcommand{\init}{\square}
\newcommand{\pth}{\omega}

\newcommand{\pthat}[1]{\pth_{#1}}
\newcommand{\pthatn}{\pth_{n}}
\newcommand{\pthatnplus}{\pth_{n+1}}
\newcommand{\pthatk}{\pth_{k}}
\newcommand{\pthatkplus}{\pth_{k+1}}
\newcommand{\pthto}[1]{\pth_{1:#1}}
\newcommand{\pthton}{\pth_{1:n}}
\newcommand{\pthtonplus}{\pth_{1:n+1}}

\newcommand{\pthtok}{\pth_{1:k}}

\newcommand{\sittok}{\sit_{1:k}}
\newcommand{\sitton}{\sit_{1:n}}
\newcommand{\sitatkplus}{\sit_{k+1}}

\newcommand{\sit}{s}

\newcommand{\precedes}{\sqsubseteq}
\newcommand{\sprecedes}{\sqsubset}

\newcommand{\xval}[1][]{x_{#1}}
\newcommand{\xvaltolong}[1][n]{\xval[1],\dots,\xval[#1]}
\newcommand{\xvalto}[1]{\xval[1:#1]}

\newcommand{\xvaltok}{\xval[1:k]}

\newcommand{\xvalatkplus}{\xval[k+1]}

\newcommand{\gamble}{f}
\newcommand{\gambles}{\mathcal{L}(\posspace)}

\newcommand{\comp}{computable}

\newcommand{\lscomp}{lower semicomputable}

\newcommand{\ind}[1]{\mathbb{I}_{#1}}

\newcommand{\then}{\Rightarrow}

\newcommand{\adddelta}{\Delta}

\usepackage{comment}

\usepackage{verbatim}

\usepackage{nicefrac}

\usepackage{amsfonts}

\usepackage{graphicx}

\usepackage{adjustbox}

\usepackage{amssymb}

\usepackage{xcolor,soul}

\usepackage{tikz}

\usepackage{subfigure}

\usepackage{float}

\usepackage{hyperref}

\usepackage{scalerel}

\usepackage{lipsum}

\colorlet{lightlightlightgray}{lightgray!10}
\colorlet{lightlightgray}{lightgray!50}

\usepackage{enumerate}

\usepackage{enumitem}

\usepackage{aligned-overset}

\usepackage{arydshln}

\newbool{ifnotECSQARU}
\setbool{ifnotECSQARU}{true}

\newenvironment{ECSQARU}{}{}

\ifbool{ifnotECSQARU}{\AtBeginEnvironment{ECSQARU}{\comment}%
\AtEndEnvironment{ECSQARU}{\endcomment}}{}

\newbool{ifnotArxiveExt}
\setbool{ifnotArxiveExt}{false}

\newenvironment{ArxiveExt}{}{}

\ifbool{ifnotArxiveExt}{\AtBeginEnvironment{ArxiveExt}{\comment}%
\AtEndEnvironment{ArxiveExt}{\endcomment}}{}

\begin{document}

\title{The Smallest Probability Interval a \\Sequence Is Random for: A Study \\ for Six Types of Randomness}%\thanks{Supported by organisation x.}}
\titlerunning{The Smallest Probability Interval a Sequence Is Random for}
\author{Floris Persiau %\orcidID{0000-0003-1239-9736} 
\and Jasper De Bock %\orcidID{000-0003-1950-0059} 
\and Gert de Cooman %\orcidID{0000-0002-0469-1422}
}
\authorrunning{F.~Persiau et al.}
\institute{Foundations Lab for imprecise probabilities, Ghent University, Belgium}
\maketitle              % typeset the header of the contribution
\begin{abstract}
\begin{comment}
Classical randomness notions only allow to say whether an infinite binary sequence—--a path—--is random for a probability or not.
Consequently, randomness notions have mainly been compared by saying for which ones a path is random with respect to a probability p, and for which ones it is not.
%For some given path, we can compare these notions by saying for which ones it is random with respect to a probability p, and for which ones it is not.
%This is mainly how randomness notions have been compared: a binary classification.
Allowing for imprecision, however, completely changes our expressiveness to differentiate between such notions; in this paper, we provide a first attempt at defining how much different randomness notions are.
We do so by allowing for closed probability intervals.
It turns out that for many randomness notions, every path is (almost) random for some smallest interval.
We study the relations between these smallest interval forecasts, try to provide equivalent expressions for them, and use them to compare several randomness notions.
When do these smallest interval forecasts coincide?
%We study the relations between the smallest intervals a given path is random, provide alternative expression for them, and use them to compare randomness notions.
% Whilst doing this, we are also able to provide a partial answer to the open question whether there is for every path some smallest interval forecast for which it is Martin-Löf random.
\end{comment}
There are many randomness notions.
On the classical account, many of them are about whether a given infinite binary sequence is random for some given probability. 
If so, this probability turns out to be the same for all these notions, so comparing them amounts to finding out for which of them a given sequence is random.
This changes completely when we consider randomness with respect to probability intervals, because here, a sequence is always random for at least one interval, so the question is not if, but rather for which intervals, a sequence is random. We show that for many randomness notions, every sequence has a smallest interval it is (almost) random for.
We study such smallest intervals and use them to compare a number of randomness notions.
We establish conditions under which such smallest intervals coincide, and provide examples where they do not.
\begin{comment}
Classical randomness notions define whether an infinite binary sequence---a path---is random for some unique probability or not.
For each notion the answer is yes or no; and classical notion of randomness are then compared by looking at the sets of paths that are random for some given probability.
For example, if for every probability the set of paths that are X-random is smaller than the set of Y-random paths, then we call X-randomness a stronger notion of randomness than Y-randomness.
This is mainly how randomness notions have been compared.
By allowing for closed probability intervals---interval forecasts---, a

Moreover, it turns out that sometimes, randomness notions are quite alike, in the sense that the smallest interval forecast coincide for a given path, and sometimes they are quite unalike.

In this paper, we adopt a new approach to comparing randomness notions by looking at the smallest interval forecasts for which it is random.

In contrast with classical randomness notions, which only define what it means to be random for a probability, 

how much different are they?
\end{comment}

\keywords{probability intervals \and Martin-Löf randomness \and computable randomness \and Schnorr randomness \and Church randomness.}
\end{abstract}

\section{Introduction}
The field of algorithmic randomness studies what it means for an infinite binary sequence, such as $\pth=0100110100\dots$, to be random for an uncertainty model.
Classically, this uncertainty model is often a single (precise) probability $p \in \sqgroup{0,1}$.
%Every such notion defines whether an infinite binary sequence~$\pth$ is random for a probability $p$ or not.
Some of the best studied precise randomness notions are Martin-Löf randomness, computable randomness, Schnorr randomness and Church randomness.
They are increasingly weaker; for example, if a sequence~$\pth$ is Martin-Löf random for a probability $p$, then it is also computably random, Schnorr random and Church random for~$p$.
Meanwhile, these notions do not coincide; it is for example possible that a path~$\pth$ is Church random but not computably random for~$\nicefrac{1}{2}$.
From a traditional perspective, this is how we can typically differentiate between various randomness notions \cite{DowneyHirschfeldt2010,Wang1996}.
%\textcolor{gray}{These are the two main ways of looking at the differences and similarities between classical randomness notions.}

As shown by De Cooman and De Bock \cite{CoomanBock2017,CoomanBock2021V2,CoomanBock2021}, these traditional randomness notions can be generalised by allowing for imprecise-probabilistic uncertainty models, such as closed probability intervals $I \subseteq \sqgroup{0,1}$.
These more general randomness notions, and their corresponding properties, allow for more detail to arise in their comparison.
Indeed, every infinite binary sequence~$\pth$ is for example random for at least one closed probability interval.
And for the imprecise generalisations of many of the aforementioned precise randomness notions, we will see that for every (or sometimes many)~$\pth$, there is some smallest probability interval, be it precise or imprecise, that \(\pth\) is (almost) random for---we will explain the modifier `almost' further on.
It is these smallest probability intervals we will use to compare a number of different randomness notions.
%These more general uncertainty models reveal a richer mathematical structure in the field of algorithmic randomness.
%These more general uncertainty models will increase our expressiveness to compare different randomness notions.
%together with their properties
%The uncovering of this mathematical structure will increase our expressiveness to compare different randomness notions.
%opens up a whole new world of uncharted questions.
%Clearly, allowing for imprecision generates a new and larger mathematical playground of uncharted questions. \textcolor{gray}{realm}

%In this paper, we will focus on the following question: how much do the classical precise-probabilistic randomness notions differ when allowing for stationary imprecise-probabilistic uncertainty models, that is, when looking from a stationary imprecise perspective.
We will focus on the following three questions:
(i) when is there a well-defined smallest probability interval for which an infinite binary sequence~$\pth$ is (almost) random; (ii) are there alternative expressions for these smallest intervals; and (iii) for a given sequence~$\pth$, how do these smallest intervals compare for different randomness notions?
Thus, by looking from an imprecise perspective, we are able to do more than merely confirm the known differences between several randomness notions.
Defining randomness for closed probability intervals also lets us explore to what extent existing randomness notions are different, in the sense that we can compare the smallest probability intervals for which an infinite binary sequence is random.
Surprisingly, we will see that there is a large and interesting set of infinite sequences $\pth$ for which the smallest interval that $\pth$ is (almost) random for is the same for several randomness notions.
%Surprisingly, it turns out there is a large and interesting set of non-stationary precise-probabilistic uncertainty models such that if an infinite binary outcome sequence~$\pth$ is random for one of these uncertainty models, then the smallest intervals for which $\pth$ is (almost) random coincide for several of the previously mentioned classical randomness notions, and in that sense, we call them `quite alike'.
%This will lead us to conclude that precise randomness notions are quite alike.

Our contribution is structured as follows.
In Section~2, we introduce (im)pre\-cise uncertainty models for infinite binary sequences, and introduce a generic definition of randomness that allows us to formally define what it means for a sequence to have a smallest interval it is (almost) random for.
%describe what it generally means for an infinite binary sequence~$\pth$ to be random.
In Section~3, we provide the mathematical background on supermartingales that we need in order to introduce a number---six in all---of different randomness notions in Section~4: (weak) Martin-Löf randomness, computable randomness, Schnorr randomness, and (weak) Church randomness.
In the subsequent sections, we tackle our three main questions.
We study the existence of the smallest intervals an infinite binary sequence~$\pth$ is (almost) random for in Section~5. 
%, and are able to provide a partial answer to an open question; is there for every infinite sequence~$\pth$ a smallest interval for which it is Martin-Löf random.
In Sections 6 and 7, we provide alternative expressions for such smallest intervals and compare them; we show that these smallest intervals coincide under certain conditions, and provide examples where they do not.
\begin{ECSQARU}
To adhere to the page limit, the proofs of all novel results---that is, the proofs of all Propositions, Corollaries and Theorems that have no citation---are omitted. 
They are available in Appendix B of an extended on-line version~\cite{floris2021ecsqaru}.
\end{ECSQARU}

\section{Forecasting systems and randomness} \label{sec:frcst}

Consider an infinite sequence of binary variables~$X_1,\dots,X_n,\dots$, where every variable~$X_n$ takes values in the binary \emph{sample space} $\posspace \coloneqq \{0,1\}$, generically denoted by~$x_n$.
We are interested in the corresponding infinite outcome sequences $(x_1,\dots,x_n,\dots)$, and, in particular, in their possible randomness.
We denote such a sequence generically by~$\pth$ and call it a \emph{path}.
All such paths are collected in the set~$\pths \coloneqq \posspace^\naturals$.\footnote{$\naturals$ denotes the natural numbers and~$\naturalswithzero\coloneqq\naturals \cup \{0\}$ denotes the non-negative integers.\protect\footnotemark}
\footnotetext{A real $x \in \reals$ is called negative, positive, non-negative and non-positive, respectively, if $x<0$, $x>0$, $x\geq 0$ and~$x \leq 0$.} 
For any path~$\pth=(\xvaltolong,\dots) \in \pths$, we let $\pthton \coloneqq (\xvaltolong)$ and~$\pthatn \coloneqq x_n$ for all~$n \in \naturals$.
For~$n=0$, the empty sequence~$\pthto{0}\coloneqq\pthat{0}\coloneqq()$ is called the \emph{initial situation} and is denoted by $\init$.
For any~$n \in \naturalswithzero$, a finite outcome sequence~$(x_1,\dots,x_n)~\in~\posspace^n$ is called a \emph{situation}, also generically denoted by~$\sit$, and its length is then denoted by $\abs{\sit}\coloneqq n$.
All situations are collected in the set~$\sits\coloneqq \bigcup_{n \in \naturalswithzero} \posspace^n$.
For any~$\sit=(\xvaltolong) \in \sits$ and~$x \in \posspace$, we use $\sit x$ to denote the concatenation $(\xvaltolong,x)$.

The randomness of a path~$\pth \in \pths$ is always defined with respect to an uncertainty model. 
Classically, this uncertainty model is a real number~$p \in \sqgroup{0,1}$, interpreted as the probability that $X_n$ equals $1$, for any~$n \in \naturals$.
As explained in the Introduction, we can generalise this by considering a closed probability interval $I \subseteq \sqgroup{0,1}$ instead.
%every~$p \in I$ can be interpreted as a possible probability for~$X_n$, with $n \in \naturals$, to be equal to $1$.
These uncertainty models will be called \emph{interval forecasts}, and we collect all such closed intervals in the set~$\intervals$.
Another generalisation of the classical case consists in allowing for non-stationary probabilities that depend on~$\sit$ or $\abs{\sit}$.
Each of these generalisations can themselves be seen as a special case of an even more general approach, which consists in providing every situation~$\sit \in \sits$ with a (possibly different) interval forecast in $\intervals$, denoted by $\frcstsystem(\sit)$.
This interval forecast~$\frcstsystem(\sit) \in \intervals$ then describes the uncertainty about the \emph{a priori} unknown outcome of $X_{\abs{\sit}+1}$, given that the situation~$\sit$ has been observed.
We call such general uncertainty models \emph{forecasting systems}.
%We even allow more general uncertainty models by considering uncertainty models that associate with every situation~$\sit \in \sits$ an interval forecast~$I_\sit \in \intervals$.
\begin{definition}
A \emph{forecasting system} is a map~$\frcstsystem\colon\sits\to\intervals$ that associates with every situation~$\sit \in \sits$ an interval forecast~$\frcstsystem(\sit) \in \intervals$.
We denote the set of all forecasting systems by $\frcstsystems$.
\end{definition}
With any forecasting system~$\frcstsystem \in \frcstsystems$, we associate two real processes~$\lfrcstsystem$ and~$\ufrcstsystem$, defined by $\lfrcstsystem(\sit)\coloneqq\min \frcstsystem(\sit)$ and 
$\ufrcstsystem(\sit)\coloneqq\max \frcstsystem(\sit)$ for all~$\sit \in \sits$.
A forecasting system~$\frcstsystem \in \frcstsystems$ is called \emph{precise} if $\lfrcstsystem=\ufrcstsystem$.
A forecasting system~$\frcstsystem \in \frcstsystems$ is called \emph{stationary} if there is an interval forecast~$I \in \intervals$ such that $\frcstsystem(\sit)=I$ for all~$\sit \in\sits$; for ease of notation, we will then denote this forecasting system simply by $I$.
The case of a single probability $p$ corresponds to a stationary forecasting system with $I=\set{p}$.
A forecasting system~$\frcstsystem \in \frcstsystems$ is called \emph{temporal} if its interval forecasts~$\frcstsystem(\sit)$ only depend on the situations $\sit \in \sits$ through their length $\abs{\sit}$, meaning that $\frcstsystem(\sit)=\frcstsystem(t)$ for any two situations $\sit,t \in \sits$ that have the same length $\abs{\sit}=\abs{t}$.
%the length $\abs{\sit}$, for all~$\sit \in \sits$.

In some of our results, we will consider forecasting systems that are computable.
To follow the argumentation and understand our results, the following intuitive description will suffice: a forecasting system~$\frcstsystem \in \frcstsystems$ is \emph{computable} if there is some finite algorithm that, for every~$\sit \in \sits$ and any $n \in \naturalswithzero$, can compute the real numbers $\lfrcstsystem(\sit)$ and~$\ufrcstsystem(\sit)$ with a precision of $2^{-n}$.
\begin{ECSQARU}
For a formal definition of computability, which we use in our proofs, we refer the reader to Appendix~A of the extended on-line version, which contains these proofs \cite{floris2021ecsqaru}.
\end{ECSQARU}
\begin{ArxiveExt}
For a formal definition of computability, which we use in our proofs, we refer the reader to Appendix~A.
\end{ArxiveExt}

So what does it mean for a path~$\pth \in \pths$ to be random for a forecasting system~$\frcstsystem \in \frcstsystems$?
Since there are many different definitions of randomness, and since we intend to compare them, we now introduce a general abstract definition and a number of potential/desirable properties of such randomness notions that will, as it turns out, allow us to do so.
%Generally speaking, a \emph{notion of randomness} \random \: defines what this means.
\begin{definition} \label{def:random}
A notion of \emph{randomness} $\random$ associates with every forecasting system~$\frcstsystem \in \frcstsystems$ a set of paths~$\pths_\random(\frcstsystem)$.
A path~$\pth \in \pths$ is called $\random$-random for~$\frcstsystem$ if $\pth \in \pths_\random(\frcstsystem)$.
\end{definition}

All of the randomness notions that we will be considering further on, satisfy additional properties.
The first one is a monotonicity property, which we can describe generically as follows.
If a path~$\pth \in \pths$ is \random-random for a forecasting system~$\frcstsystem \in \frcstsystems$, it is also \random-random for any forecasting system~$\frcstsystem' \in \frcstsystems$ that is less precise, meaning that $\frcstsystem(\sit) \subseteq \frcstsystem'(\sit)$ for all~$\sit \in \sits$.
Consequently, this monotonicity property requires that the more precise a forecasting system is, the fewer \random-random paths it ought to have.

\begin{property}\label{prop:monotone}
For any two forecasting systems~$\frcstsystem,\frcstsystem' \in \frcstsystems$ such that $\frcstsystem \subseteq \frcstsystem'$, it holds that $\pths_\random(\frcstsystem) \subseteq \pths_\random(\frcstsystem')$.
%If a path~$\pth \in \pths$ is \random-random for a forecasting system~$\frcstsystem \in \frcstsystems$, then $\pth$ is also \random-random for any forecasting system~$\frcstsystem' \in \frcstsystems$ for which $\frcstsystem\subseteq\frcstsystem'$, meaning that $\frcstsystem(\sit) \subseteq \frcstsystem(\sit)$ for all~$\sit \in \sits$.
\end{property}
\noindent Furthermore, it will also prove useful to consider the property that every path~$\pth\in\pths$ is \random-random for the (maximally imprecise) \emph{vacuous forecasting system}~$\frcstsystem_v \in \frcstsystems$, defined by $\frcstsystem_v(\sit) \coloneqq \sqgroup{0,1}$ for all~$\sit \in \sits$.
\begin{property} \label{proper:non-empty}
$\pths_\random(\sqgroup{0,1}) = \pths$.
%there is at least one path~$\pth \in \pths$ such that $\pth \in \pths_\random(\frcstsystem)$.
\end{property}
\noindent Thus, if Properties~\ref{prop:monotone} and~\ref{proper:non-empty} hold, every path~$\pth \in \pths$ will in particular be \random-random for at least one interval forecast---the forecast $I=\sqgroup{0,1}$---and if a path~$\pth \in \pths$ is \random-random for an interval forecast~$I \in \intervals$, then it will also be \random-random for any interval forecast~$I' \in \intervals$ for which $I \subseteq I'$.
It is therefore natural to wonder whether every path~$\pth \in \pths$ has some \emph{smallest} interval forecast~\(I\) such that~\(\pth\in\pths_\random(I)\).
In order to allow us to formulate an answer to this question, we consider the sets~$\randomintervals$ that for a given path~$\pth \in \pths$ contain all interval forecasts~$I \in \intervals$ that $\pth$ is \random-random for.
If there is such a smallest interval forecast, then it is necessarily given by 
\begin{equation*}
I_\random(\pth) 
\coloneqq \bigcap \randomintervals =
\smashoperator{\bigcap_{I \in \randomintervals}}I.
%=
%\bigcap \mathcal{I}_\random(\pth),
\end{equation*}
As we will see, for some randomness notions~\random, $\randomI$ will indeed be the smallest interval forecast that $\pth$ is random for.
Consequently, for these notions, and for every $\pth \in \pths$, the set~$\randomintervals$ is completely characterised by the interval forecast~$\randomI$, in the sense that $\pth$ will be \random-random for an interval forecast~$I \in \intervals$ if and only if $\randomI \subseteq I$.

In general, however, this need not be the case.
For example, consider the situation depicted in \figureref{fig:intervalsclarify}. 
It could very well be that for some randomness notion \random \: that satisfies Properties~\ref{prop:monotone} and~\ref{proper:non-empty}, there is a path~$\pth^\ast \in \pths$ that is \random-random for all interval forecasts of the form~$\sqgroup{p,1}$ and~$\sqgroup{0,q}$, with $p<\nicefrac{1}{3}$ and~$\nicefrac{2}{3}\leq q$, but for no others.
Then clearly, $I_\random(\pth^\ast) = \sqgroup{\nicefrac{1}{3},\nicefrac{2}{3}}$, but $\pth^\ast$ is not \random-random for~$I_\random(\pth^\ast)$.
% Meanwhile, Properties~\ref{prop:monotone} and~\ref{proper:non-empty} still hold for the path~$\pth^\ast$ when we restrict our attention to interval forecasts.

\begin{figure}[h]
\centering
%\resizebox{\textwidth}{!}{
\begin{tikzpicture}[xscale=5,yscale=2]\footnotesize
\useasboundingbox (0,0) rectangle (1,1.2);

% the unit interval
\draw[->] (-0.05,1) -- (1.05,1) node[above right] {$\reals$};
\draw[thick] (0,1) node[circle,inner sep=1pt,fill] {} node [above=3pt] {\(0\)} -- (1,1) node[circle,inner sep=1pt,fill] {} node[above=2pt] {\(1\)};
%\node[right] at (1,1) {\(\sqgroup{0,1}\)};
\draw[thick] (1/3,1) node[circle,inner sep=1pt,fill] {} node [above=1pt] {\(\nicefrac{1}{3}\)} -- (2/3,1) node[circle,inner sep=1pt,fill] {} node[above=1pt] {\(\nicefrac{2}{3}\)};

% help lines
\draw[densely dotted] (0,1) -- (0,-0.1) (1,1) -- (1,-0.1);

% help lines
\draw[densely dotted] (1/3,1) -- (1/3,-0.1) (2/3,1) -- (2/3,-0.1);

% a good interval
\draw[thick,green!60!black] (0,0.8) node[circle,inner sep=1pt,fill] {} -- (2/3,0.8) node[circle,inner sep=1pt,fill] {};
\node[right,green!60!black] at (1,0.8) {\(\sqgroup{0,\nicefrac{2}{3}}\)};

% another good interval
\draw[thick,green!60!black] (0.2,0.6) node[circle,inner sep=1pt,fill] {} -- (1,0.6) node[circle,inner sep=1pt,fill] {};
\node[right,green!60!black] at (1,0.6) {\(\sqgroup{0.2,1}\)};

% a bad interval
\draw[thick,red!90!black] (0.1,0.4) node[circle,inner sep=1pt,fill] {} node [below=4pt] {} -- (0.8,0.4) node[circle,inner sep=1pt,fill] {} node[below=5pt] {};
\node[right,red!90!black] at (1,0.4) {\(\sqgroup{0.1,0.8}\)};

% another bad interval
\draw[thick,red!90!black] (1/3,0.2) node[circle,inner sep=1pt,fill] {} node [below=4pt] {} -- (1,0.2) node[circle,inner sep=1pt,fill] {} node[below=5pt] {};
\node[right,red!90!black] at (1,0.2) {\(\sqgroup{\nicefrac{1}{3},1}\)};

% another one bad interval
\draw[thick,red!90!black] (1/3,0) node[circle,inner sep=1pt,fill] {} node [below=4pt] {} -- (2/3,0) node[circle,inner sep=1pt,fill] {} node[below=5pt] {};
\node[right,red!90!black] at (1,0) {\(\sqgroup{\nicefrac{1}{3},\nicefrac{2}{3}}\)};

\end{tikzpicture}
%}
\caption{%Consider a path~$\pth \in \pths$ that is almost \random-random for the interval forecast~$I_\random\coloneqq\sqgroup{0.4,0.7}$.
%\protect\footnotemark 
The green intervals correspond to interval forecasts for which $\pth^\ast$ is \random-random, whereas the red intervals correspond to interval forecasts that $\pth^\ast$ is not \random-random for.}
\label{fig:intervalsclarify}
\end{figure}
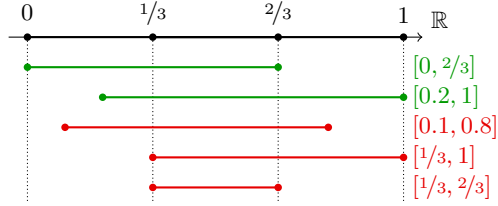

In addition, it need not even be guaranteed that the intersection $\randomI$ is non-empty. 
To guarantee that it will be, and as an imprecise counterpart of the law of large numbers, it suffices to consider the additional property that if a path~$\pth \in \pths$ is \random-random for an interval forecast~$I \in \intervals$, then this $I$ should imply the following bounds on the relative frequency of ones along~$\pth$.
% As a sufficient condition to guarantee that it will be, and as an imprecise counterpart of the law of large numbers, we may require as an additional property that if a path~$\pth \in \pths$ is \random-random for an interval forecast~$I \in \intervals$, then this $I$ should imply the following bounds on the relative frequency of ones along~$\pth$.

\begin{property} \label{prop:law:large:numbers}
For all interval forecasts~$I \in \intervals$ and all paths~$\pth \in \pths_\random(I)$, it holds that $\min I \leq \liminf_{n \to \infty} \frac{1}{n} \sum_{k=1}^{n}\pthatk \leq \limsup_{n \to \infty} \frac{1}{n} \sum_{k=1}^{n}\pthatk \leq \max I$.
\end{property}
Properties~\ref{prop:monotone}--\ref{prop:law:large:numbers} hold for all randomness notions~\random \: that we will consider, whence also, $\randomI \neq \emptyset$.
We repeat that for some of these notions, $\randomI$ will be the smallest interval forecast that~$\pth \in \pths$ is \random-random for.
If not, we will sometimes still be able to show that $\randomI$ is the smallest interval forecast that~$\pth$ is \emph{almost} \random-random for.
\begin{definition} \label{def:almost}
A path~$\pth \in \pths$ is called \emph{almost} \random-random for an interval forecast~$I \in \intervals$ if it is \random-random for any interval forecast~$I' \in \intervals$ of the form 
\begin{equation*}
I'=\sqgroup{\min I - \epsilon_1,\max I + \epsilon_2} \cap \sqgroup{0,1} \textrm{, with }\epsilon_1,\epsilon_2>0.
\end{equation*} 
\end{definition}
If a path~$\pth \in \pths$ is almost \random-random for the interval forecast~$\randomI$, then $\randomI$ almost completely characterises the set~$\randomintervals$: the only case where we cannot immediately decide whether a path~$\pth$ is \random-random for an interval forecast~$I \in \intervals$ or not, occurs when $\min I= \min \randomI$ or $\max I= \max \randomI$.
Moreover, if \propertyref{prop:monotone} holds, then as our terminology suggests, $\pth \in \pths$ is \emph{almost} \random-random for every interval forecast $I \in \intervals$ it is random for.
%all interval forecasts~$I \in \randomintervals$.
\begin{comment}
\begin{property} \label{prop:almost:smallest:interval}
A path~$\pth \in \pths$ is \emph{almost completely} $\random$-random for an interval forecast~$I_\random \in \intervals$ if it is \random-random for any interval forecast~$I \in \intervals$ of the form 
\begin{equation*}
I=\sqgroup{\min I_\random - \epsilon_1,\max I_\random + \epsilon_2} \cap \sqgroup{0,1} \textrm{, with }\epsilon_1,\epsilon_2>0,
\end{equation*} 
and not \random-random for any interval forecast~$I \in \intervals$ if 
\begin{equation*}
\min I_\random < \min I \textrm{ or } \max I < \max I_\random.
\end{equation*}
\end{property}
\end{comment}
\begin{comment}
To clarify the previous property, we provide a graphical representation in \figureref{fig:intervals}.
\begin{figure}[h]
\centering
%\resizebox{\textwidth}{!}{
\input{intervals.tikz}
%}
\caption{The path~$\pth \in \pths$ is almost \random-random for the interval forecast~$I_\random\coloneqq\sqgroup{0.4,0.7}$.
%\protect\footnotemark 
The green interval corresponds to an interval forecast that $\pth$ is \random-random for, whereas the red intervals correspond to interval forecasts that $\pth$ is not \random-random for.}
\label{fig:intervals}
\end{figure}
%\footnotetext{It is clear from \exampleref{ex:pq} that this is always possible.}
\end{comment}

In the remainder of this contribution, we intend to study the smallest interval forecasts a path is (almost) random for, for several notions of randomness.
In the next section, we start by introducing the mathematical machinery needed to introduce some of these notions, and in particular, the martingale-theoretic approach to randomness, which makes extensive use of the concept of betting.
Generally speaking, a path~$\pth \in \pths$ is then considered to be random for a forecasting system~$\frcstsystem \in \frcstsystems$ if a subject can adopt no implementable betting strategy that is allowed by $\frcstsystem$ and makes him arbitrarily rich along~$\pth$.
This approach will enable us to introduce the notions of Martin-Löf randomness, weak Martin-Löf randomness, computable randomness and Schnorr randomness, which differ only in what is meant by `implementable' and in the way a subject should not be able to get arbitrarily rich \cite{DowneyHirschfeldt2010}.

\begin{comment}
\textcolor{red}{
in order to introduce some randomness notions
a betting strategy that is allowed by the forecasting system/a betting strategy that corresponds with
in the martingale-theoretic framework
as has been introduced by De Cooman and De Bock
generally speaking, a path~$\pth \in \pths$ is random for ...
a forecasting system if a subject can adopt no betting strategy that is allowed by the forecasting system and makes her arbitrarily rich along~$\pth$.
martingale-theoretic
}
\end{comment}

\section{A martingale-theoretic approach---betting strategies} \label{sec:betting}
Consider the following betting game involving an infinite sequence of binary variables~$X_1,\dots,X_n,\dots$
There are three players: Forecaster, Sceptic and Reality.

Forecaster starts by specifying a forecasting system~$\frcstsystem \in \frcstsystems$.
For every situation~$\sit \in \sits$, the corresponding interval forecast~$\frcstsystem(\sit)$ expresses for every \emph{gamble}~$\gamble \colon \posspace \to \reals$ whether or not Forecaster allows Sceptic to select $\gamble$; the set of all gambles is denoted by $\gambles$.
A gamble~$g \in \gambles$ is offered by Forecaster to Sceptic if its expectation $\ex_p(g)\coloneqq p g(1)+(1-p)g(0)$ is non-positive for every probability $p \in I$, or equivalently, if $\max_{p \in I}\ex_p(g) \leq 0$.

%with every interval forecast~$I \in \intervals$ there corresponds a set of \emph{allowable} gambles consisting of all gambles $g \in \gambles$ such that for every probability $p \in I$ their expectation satisfies $p g(1)+(1-p)g(0) \leq 0$ \cite{CoomanBock2021,floris2021}.
%for which the expectation with respect to every probability $p \in I$ satisfies $p g(1)+(1-p)g(0) \leq 0$ \cite{CoomanBock2021,floris2021}.
%Namely, with every interval forecast~$I=\interval \in \intervals$ there corresponds a set of \emph{allowable} gambles $g \in \gambles$ defined by $g(1)=\alpha(p-1)+\beta(1-q)$ and~$g(0)=\alpha p-\beta q$, with $\alpha,\beta \geq 0$, $p \leq \intervalmin$ and~$q \geq \intervalmax$ 
%; such a set of gambles is depicted in \figureref{fig:uex(f)<=0}.
After Forecaster has specified a forecasting system~$\frcstsystem \in \frcstsystems$, Sceptic selects a \emph{betting strategy} that specifies for every situation~$\sit \in \sits$ an \emph{allowable} gamble~$\gamble_\sit \in \gambles$ for the corresponding interval forecast~$\frcstsystem(\sit) \in \intervals$, meaning that $\max_{p \in \frcstsystem(\sit)}\ex_p(\gamble_\sit) \leq 0$.

The betting game now unfolds as Reality reveals the successive elements $\pthatn \in \posspace$ of a path~$\pth \in \pths$. % by sharing at every \emph{time instant} $n \in \naturalswithzero$ the successive element $\pthatnplus \in \posspace$.
In particular, at every time instant $n \in \naturalswithzero$, the following actions have been and are completed: Reality has already revealed the situation $\pthton$, Sceptic engages in a gamble~$\gamble_{\pthton} \in \gambles$ that is specified by his betting strategy, Reality reveals the next outcome $\pthatnplus~\in~\posspace$, and Sceptic receives a (possibly negative) reward $\gamble_{\pthton}(\pthatnplus)$.
We furthermore assume that Sceptic starts with initial unit capital, so his running capital at every time instant $n \in \naturalswithzero$ equals $1+\sum_{k=0}^{n-1} \gamble_{\pthtok}(\pthatkplus)$.
We also don't allow Sceptic to borrow.
This means that he is only allowed to adopt betting strategies that, regardless of the path that Reality reveals, will guarantee that his running capital never becomes negative.
\begin{comment}
\begin{figure}[H]
\centering
\resizebox{0.5\textwidth}{!}{
\input{gamblesconeSceptic.tikz}
}
\caption{Let \(I \coloneqq \sqgroup{\nicefrac{1}{4},\nicefrac{3}{4}}\). 
%Then the leftmost blue region represents Sceptic's allowable gambles \(\gamble \in \gambles\), the middle blue region represents Sceptic's allowable additive gambles $\gamble \in \gambles$, and the rightmost blue region represents Sceptic's allowable multiplier gambles \(\gamble \in \gambles\).}
The blue region represents Sceptic's allowable gambles \(g \in \gambles\) for \(\sqgroup{\nicefrac{1}{4},\nicefrac{3}{4}}\).}
%that correspond to an uncertain reward \(\alpha(p-X)+\beta(X-q)\), with \(p~\leq~\nicefrac{1}{4}\), \(q~\geq~\nicefrac{3}{4}\) and \(\alpha,\beta\geq0\) .}
\label{fig:uex(f)<=0}
\end{figure}
\end{comment}

In order to formalise Sceptic's betting strategies, we will introduce the notion of test supermartingales.
We start by considering a \emph{real process} $\process\colon\sits\to\reals$; it is called positive if $\process(\sit) >0$ for all~$\sit \in \sits$ and non-negative if $\process(\sit) \geq 0$ for all~$\sit \in \sits$.
A real process~$\process$ is called \emph{temporal} if $\process(\sit)$ only depends on the situation~$\sit \in \sits$ through its length $\abs{\sit}$, meaning that $\process(\sit)=\process(t)$ for any two $\sit,t\in\sits$ such that $\abs{\sit}=\abs{t}$.
A real process~$\selection$ is called a \emph{selection process} if $\selection(\sit) \in \{0,1\}$ for all~$\sit \in \sits$.

With any real process~$\process$, we can associate a \emph{gamble process} $\adddelta \process\colon \sits \to \gambles$, defined by $\adddelta \process(\sit)(x)\coloneqq \process(\sit \, x)-\process(\sit)$ for all~$\sit \in \sits$ and~$x \in \posspace$, and we call it the \emph{process difference} for~$\process$. 
If $\process$ is positive, then we can also consider another gamble process~$\multprocess_\process \colon \sits \to \gambles$, defined by $\multprocess_\process(\sit)(x)\coloneqq \nicefrac{\process(\sit \, x)}{\process(\sit)}$ for all~$\sit \in \sits$ and~$x \in \posspace$, which we call the \emph{multiplier process} for~$\process$.
And vice versa, with every non-negative real gamble process~$D\colon \sits \to \gambles$, we can associate a non-negative real process~$\mint \colon \sits \to \reals$ defined by $\mint(\sit)\coloneqq \prod_{k=0}^{n-1}\multprocess(\xvaltok)(\xvalatkplus)$ for all~$\sit=(\xvaltolong) \in \sits$, and we then say that $\mint$ is \emph{generated by}~$\multprocess$.

When given a forecasting system~$\frcstsystem \in \frcstsystems$, we call a real process~$\supermartin$ a \emph{supermartingale} for~$\frcstsystem$ if for every~$\sit \in \sits$, $\adddelta \supermartin(\sit)$ is an allowable gamble for the corresponding interval forecast~$\frcstsystem(\sit)$, meaning that $\max_{p \in \frcstsystem(\sit)}\ex_p(\adddelta \supermartin(\sit)) \leq 0$.
Moreover, a supermartingale $\test$ is called a \emph{test} supermartingale if it is non-negative and $\test(\init)\coloneqq1$.
We collect all test supermartingales for~$\frcstsystem$ in the set~$\tests{\frcstsystem}$.
It is easy to see that every test supermartingale~$\test$ corresponds to an allowed betting strategy for Sceptic that starts with unit capital and avoids borrowing.
Indeed, for every situation~$\sit=(\xvaltolong) \in \sits$, $\test$ specifies an allowable gamble~$\adddelta \test(\sit)$ for the interval forecast~$\frcstsystem(\sit)\in\intervals$, and Sceptic's running capital $1+\sum_{k=0}^{n-1}\adddelta \test(\xvaltok)(\xvalatkplus)$ equals $\test(\sit)$ and is therefore non-negative, and equals~$1$ in $\init$.

\begin{comment}
We recall that a path~$\pth$ is random for a forecasting system~$\frcstsystem \in \frcstsystems$ if Sceptic can adopt no computably implementable betting strategy that makes her arbitrarily rich along~$\pth$.
This means that we are going to consider computably implementable supermartingales.
The reasons for doing this are twofold.
First, it gives an intuitively understandable meaning to Sceptic's allowable betting strategies; she can only adopt a betting strategy for which she has a finite description.
After all, of what practical use would a betting strategy be if you could not describe it in a finitary manner.
Second, to end up with a notion of randomness for which \propertyref{proper:non-empty} holds,
% is for example random for the precise interval forecast~$\nicefrac{1}{2}$
it suffices to restrict Sceptic's betting strategies to a countable set; and there is only a countably infinite number of computably implementable betting strategies.
\end{comment}

%there is some finite algorithm that, for every~$\sit \in \sits$ and~$x \in \posspace$, allows to approximate the real number~$\multprocess(\sit)(x)$ from below and to arbitrary precision, with the rate of convergence being unknown.

We recall from Section~2 that martingale-theoretic randomness notions differ in the nature of the implementable betting strategies that are available to Sceptic.
More formally, we will consider three different types of implementable test supermartingales: computable ones, lower semicomputable ones, and test supermartingales generated by lower semicomputable multiplier processes.
A test supermartingale~$\test \in \tests{\frcstsystem}$ is called \emph{computable} if there is some finite algorithm that, for every~$\sit \in \sits$ and any $n \in \naturalswithzero$, can compute the real number~$\test(\sit)$ with a precision of $2^{-n}$.
%This means that we are going to consider computably implementable supermartingales.
% and differ in the way how Sceptic should not become arbitrarily rich.
%Therefore, this seems the right place to us for introducing several %countable 
%sets of computably implementable betting strategies, which will then be used in the next section to introduce martingale-theoretic imprecise-probabilistic notions of Martin-Löf ($\ml$), weak Martin-Löf ($\wml$), computable ($\co$) and Schnorr ($\s$) randomness.
%\footnote{Note that weak Martin-Löf randomness is not one of the classical randomness notions that are mentioned in the introduction.
%In fact, this is a notion of randomness that was only recently introduced by de Cooman and De Bock \cite{CoomanBock2021V2,CoomanBock2021}}
%We will need two types of computably implementable betting strategies: \emph{lower semicomputable} and \emph{computable} ones.
A test supermartingale~$\test \in \tests{\frcstsystem}$ is called \emph{lower semicomputable} if there is some finite algorithm that, for every~$\sit \in \sits$, can compute an increasing sequence $(q_n)_{n\in\naturalswithzero}$ of rational numbers that approaches the real number~$\test(\sit)$ from below---but without knowing, for any given $n$, how good the lower bound $q_n$ is.
Similarly, a real multiplier process~$\multprocess$ is called lower semicomputable if there is some finite algorithm that, for every~$\sit \in \sits$ and~$x \in \posspace$, can compute an increasing sequence $(q_n)_{n\in\naturalswithzero}$ of rational numbers that approaches the real number~$\multprocess(\sit)(x)$ from below.
\begin{ECSQARU}
For more details, we refer the reader 
to Appendix A of the extended on-line version \cite{floris2021ecsqaru}.
\end{ECSQARU}
\begin{ArxiveExt}
For more details, we refer the reader 
to Appendix~A.
\end{ArxiveExt}
% As was proved in \cite{CoomanBock2021V2}, $\schnorrtests{\frcstsystem} = \comptests{\frcstsystem} \subseteq \weakmltests{\frcstsystem} \subseteq \mltests{\frcstsystem}$ for all~$\frcstsystem \in \frcstsystems$.

\section{Several notions of (imprecise) randomness} \label{sec:notionsofrandomness}

At this point, we have introduced the necessary mathematical machinery to define our different randomness notions.
We start by introducing four martingale-theoretic ones: Martin-Löf (\ml) randomness, weak Martin-Löf (\wml) randomness, computable (\co) randomness and Schnorr (\s) randomness.
%At this point, we have introduced the necessary material for introducing the martingale-theoretic imprecise-probabilistic counterparts of the most well-known classical randomness notions: Martin-Löf randomness, computable randomness and Schnorr randomness.
Generally speaking, for these notions, a path~$\pth \in \pths$ is random for a forecasting system~$\frcstsystem \in \frcstsystems$ if Sceptic has no implementable allowed betting strategy that makes him arbitrarily rich along~$\pth$.
We stress again that these randomness notions differ in how Sceptic's betting strategies are implementable, and in how he should not be able to become arbitrarily rich along a path~$\pth \in \pths$.
With these types of restrictions in mind, we introduce the following sets of implementable allowed betting strategies.
\begin{center}
\begin{tabular}{r|l}
\(\mltests{\frcstsystem}\) \: &\: all lower semicomputable test supermartingales for~\(\frcstsystem\)\\
\(\weakmltests{\frcstsystem}\) \: &\: all test supermartingales for~\(\frcstsystem\) generated by lower\\
&\: semicomputable multiplier processes \\
\(\comptests{\frcstsystem},\schnorrtests{\frcstsystem}\) \: &\: all computable test supermartingales for~\(\frcstsystem\)
\end{tabular}
\end{center}
For a path~$\pth$ to be Martin-Löf, weak Martin-Löf or computably random, we require that Sceptic's running capital should never be \emph{unbounded} on~$\pth$ for any implementable allowed betting strategy; that is, no test supermartingale~$\test \in \overline{\mathbb{T}}_\random(\frcstsystem)$ should be \emph{unbounded} on~$\pth$, meaning that $\limsup_{n \to \infty}\test(\pthton)=\infty$.
\begin{definition}[{\cite{CoomanBock2021}}] \label{def:notionsofrandomness}
For any~$\random \in \set{\ml,\wml,\co}$, a path~$\pth \in \pths$ is \random-random for a forecasting system~$\frcstsystem \in \frcstsystems$ if no test supermartingale~$\test \in \overline{\mathbb{T}}_\random(\frcstsystem)$ is unbounded on~$\pth$.
\end{definition}

For Schnorr randomness, we require instead that Sceptic's running capital should not be \emph{computably unbounded} on~$\pth$ for any implementable allowed betting strategy.
More formally, we require that no test supermartingale~$\test \in \schnorrtests{\frcstsystem}$ should be \emph{computably unbounded} on~$\pth$.
That~\(\test\) is computably unbounded on~\(\pth\) means that $\limsup_{n \to \infty}[\test(\pthton)-\realgrowth(n)] \geq 0$ for some real map~$\realgrowth\colon\naturalswithzero\to\nonnegreals$ that is
\begin{enumerate}[label=\upshape(\roman*),leftmargin=*,noitemsep,topsep=3pt]
\item computable;
\item non-decreasing, so $\realgrowth(n+1) \geq \realgrowth(n)$ for all~$n \in \naturalswithzero$;
\item unbounded, so $\lim_{n \to \infty}\realgrowth(n)=\infty$.\footnote{Since $\realgrowth$ is non-decreasing, it being unbounded is equivalent to $\lim_{n \to \infty}\realgrowth(n)=\infty$.}
\end{enumerate}
% Now, we will call $\test$ \emph{computably unbounded}  $\pth$ if there is some real growth function $\realgrowth$ such that $\limsup_{n \to \infty}(\test(\pthton)-\realgrowth(n)) \geq 0$.
Since such a \emph{real growth function}~$\realgrowth$ is unbounded, it expresses a (computable) lower bound for the `rate' at which $\test$ increases to infinity along~$\pth$.
Clearly, if $\test \in \schnorrtests{\frcstsystem}$ is computably unbounded on~$\pth \in \pths$, then it is also unbounded on~$\pth$.

\begin{definition}[{\cite{CoomanBock2021}}] \label{def:schnorr}
A path~$\pth \in \pths$ is \s-random for a forecasting system~$\frcstsystem \in \frcstsystems$ if no test supermartingale~$\test \in \schnorrtests{\frcstsystem}$ is computably unbounded on~$\pth$.
\end{definition}

De Cooman and De Bock have proved that these four martingale-theoretic randomness notions satisfy Properties~\ref{prop:monotone} and~\ref{proper:non-empty} \cite[Propositions 9,10,17,18]{CoomanBock2021}.
To describe the relations between these martingale-theoretic imprecise-probabilistic randomness notions, we consider the sets~$\randompths$, with $\random \in \set{\ml,\wml,\co,\s}$; they satisfy the following inclusions \cite[Section~6]{CoomanBock2021}.
\begin{equation*}
\mlpths \ \subseteq \ \weakmlpths \ \subseteq \ \comppths \ \subseteq \ \schnorrpths.
\end{equation*}
Thus, if a path~$\pth \in \pths$ is Martin-Löf random for a forecasting system~$\frcstsystem \in \frcstsystems$, then it is also weakly Martin-Löf, computably and Schnorr random for~$\frcstsystem$.
Consequently, for every forecasting system~$\frcstsystem \in \frcstsystems$, there are at most as many paths that are Martin-Löf random as there are weakly Martin-Löf, computably or Schnorr random paths.
We therefore call Martin-Löf randomness \emph{stronger} than weak Martin-Löf, computable, or Schnorr randomness.
And so, {\itshape mutatis mutandis}, for the other randomness notions.
%Vice versa, Schnorr randomness is a \emph{weaker} notion of randomness than computable, weak Martin-Löf and Martin-Löf randomness.

We also consider two other imprecise-probabilistic randomness notions, which have a more frequentist flavour: Church randomness (\ch) and weak Church randomness (\wch).
Their definition makes use of yet another (but simpler) type of implementable real processes; a selection process $\selection$ is called \emph{recursive} if there is a finite algorithm that, for every $\sit \in \sits$, outputs the binary digit $\selection(\sit) \in \set{0,1}$.
%In order to do so, it will be convenient to have the following terminology and notation at our disposal.
%We call a recursive selection process~$\selection$ \emph{dense} along a path~$\pth \in \pths$ if $\lim_{n \to \infty}\selectionsum=\infty$.
%For every path~$\pth \in \pths$, we collect the corresponding recursive dense selection processes in the set~$\selectionsdense$.
%Similarly, for every path~$\pth \in \pths$, we collect the corresponding recursive dense temporal selection processes in the set~$\selectionsfdense$.
\begin{definition}[{\cite{CoomanBock2021}}] \label{def:churchrandom}
A path~$\pth \in \pths$ is \ch-random \emph{(}\wch-random\emph{)} for a forecasting system~$\frcstsystem \in \frcstsystems$ if for every recursive \emph{(}temporal\emph{)} selection process~$\selection$ for which $\lim_{n \to \infty}\selectionsum=\infty$, it holds that
% for which $\lim_{n \to \infty}\sum_{k=0}^{n-1}\selection(\pthtok)=\infty$ :
\begin{align*}
\liminf_{n \to \infty} \frac{\sum_{k=0}^{n-1}\selection(\pthtok)[\pthatkplus-\underline{\frcstsystem}(\pthtok)]}{\sum_{k=0}^{n-1}\selection(\pthtok)} \geq 0
\shortintertext{and}
\limsup_{n \to \infty} \frac{\sum_{k=0}^{n-1}\selection(\pthtok)[\pthatkplus-\overline{\frcstsystem}(\pthtok)]}{\sum_{k=0}^{n-1}\selection(\pthtok)} \leq 0.
\end{align*}
\end{definition}
\noindent For a stationary forecasting system~$I \in \intervals$, the conditions in these definitions simplify to the perhaps more intuitive requirement that
\begin{equation*}
\min I \leq 
\liminf_{n \to \infty} \frac{\sum_{k=0}^{n-1}\selection(\pthtok)\pthatkplus}{\sum_{k=0}^{n-1}\selection(\pthtok)} \leq
\limsup_{n \to \infty} \frac{\sum_{k=0}^{n-1}\selection(\pthtok)\pthatkplus}{\sum_{k=0}^{n-1}\selection(\pthtok)} \leq
\max I.
\end{equation*}

%To describe the relation between these two frequentist flavoured randomness notions, we consider the sets~$\churchpths$ and~$\weakchurchpths$. %which contain for every path~$\pth \in \pths$ all forecasting systems for which $\pth$ is Church and weakly Church random, respectively.

It is easy to see that these two randomness notions also satisfy Properties~\ref{prop:monotone} and~\ref{proper:non-empty}.
Since the notion of weak Church randomness considers fewer selection processes than Church randomness does, it is clear that if a path~$\pth~\in~\pths$ is Church random for a forecasting system~$\frcstsystem \in \frcstsystems$, then it is also weakly Church random for~$\frcstsystem$.
Hence, $\churchpths \subseteq \weakchurchpths$.
For computable forecasting systems, we can also relate these two `frequentist flavoured' notions with the martingale-theoretic notions considered before \cite[Sections 6 and 7]{CoomanBock2021}: for every computable forecasting system~$\frcstsystem \in \frcstsystems$,
\begin{equation} \label{eq:inclusions:comp}
\mlpths \ \subseteq \ \weakmlpths \ \subseteq \ \comppths \
\begin{array}{l}
\vspace{3pt} \raisebox{-3 pt}{\rotatebox[origin=c]{20}{$\subseteq$}} \ \ \churchpths \ \ \raisebox{-3 pt}{\rotatebox[origin=c]{-20}{$\subseteq$}} \\
\raisebox{3 pt}{\rotatebox[origin=c]{-20}{$\subseteq$}} \ \ \schnorrpths \phantom{\s} \ \ \raisebox{3 pt}{\rotatebox[origin=c]{20}{$\subseteq$}}
\end{array}
\ \weakchurchpths.
\end{equation}

\section{Smallest interval forecasts and randomness} \label{sec:question1}

From now on, we will focus on stationary forecasting systems and investigate the differences and similarities between the six randomness notions we consider. 
We start by studying if there is a smallest interval forecast for which a path is (almost) random.
To this end, we first compare the sets~$\intervals_\random(\pth)$, with $\random \in \set{\ml,\wml,\co,\s,\ch,\wch}$.
%which contain for every path~$\pth \in \pths$ all interval forecasts~$I \in \intervals$ (thus both the computable and non-computable ones) for which $\pth$ is \random-random, with $\random \in \set{\ml,\wml,\co,\s,\ch,\wch}$.
They satisfy similar relations as the sets~$\randompths$---but without a need for computability assumptions.
\begin{proposition}[{\cite[Section~8]{CoomanBock2021}}] \label{prop:relations:interval}
For every path~$\pth \in \pths$, it holds that
\begin{equation*}
\mlintervals \ \subseteq \ \weakmlintervals \ \subseteq \ \compintervals \
\begin{array}{l}
\vspace{3pt} \raisebox{-3 pt}{\rotatebox[origin=c]{20}{$\subseteq$}} \ \ \churchintervals \ \ \raisebox{-3 pt}{\rotatebox[origin=c]{-20}{$\subseteq$}} \\
\raisebox{3 pt}{\rotatebox[origin=c]{-20}{$\subseteq$}} \ \ \schnorrintervals \phantom{\s} \ \ \raisebox{3 pt}{\rotatebox[origin=c]{20}{$\subseteq$}}
\end{array}
\ \weakchurchintervals.
\end{equation*}
\end{proposition}
\noindent Similarly to before, if a path~$\pth \in \pths$ is Martin-Löf random for an interval forecast~$I \in \intervals$, then it is also weakly Martin-Löf, computably, Schnorr and (weakly) Church random for~$I$.
Observe that for our weakest notion of randomness, \definitionref{def:churchrandom}---with $\selection=1$---guarantees that all interval forecasts~$I \in \weakchurchintervals$ satisfy \propertyref{prop:law:large:numbers}, and therefore, by \propositionref{prop:relations:interval}, all six randomness notions that we are considering here satisfy \propertyref{prop:law:large:numbers}.
%By \definitionref{def:churchrandom}, all interval forecast~$I \in \weakchurchintervals$ for which a path~$\pth \in \pths$ is weakly Church random satisfy \propertyref{prop:law:large:numbers}, and therefore, by \propositionref{prop:relations:interval}, all six randomness notions satisfy \propertyref{prop:law:large:numbers}.
Since the sets~$\randomintervals$ are also non-empty by \propertyref{proper:non-empty}, the interval forecasts~$\randomI$ are well-defined and non-empty for all~$\random \in \set{\ml,\wml,\co,\s,\ch,\wch}$.
Moreover, since the sets~$\mathcal{I}_\random(\pth)$ satisfy the relations in \propositionref{prop:relations:interval}, their intersections $I_\random(\pth)$ satisfy the following inverse relations.
\begin{corollary} \label{cor:relations:interval}
For every path~$\pth \in \pths$, it holds that
\begin{equation*}
\weakchurchI \
\begin{array}{l}
\vspace{3pt} \raisebox{-3 pt}{\rotatebox[origin=c]{20}{$\subseteq$}} \ \ \churchI \ \ \raisebox{-3 pt}{\rotatebox[origin=c]{-20}{$\subseteq$}} \\
\raisebox{3 pt}{\rotatebox[origin=c]{-20}{$\subseteq$}} \ \ \schnorrI \phantom{\s} \ \ \raisebox{3 pt}{\rotatebox[origin=c]{20}{$\subseteq$}}
\end{array}
\ \compI \ \subseteq \ \weakmlI \ \subseteq \ \mlI.
\end{equation*}
\end{corollary}
% For the notions of \wml-, \co- and \s-randomness, De Cooman and de Bock proved that the interval forecast~$\randomI$  almost completely characterise the interval  \cite{CoomanBock2021}
\noindent For Church and weak Church randomness, it holds that every path~$\pth \in \pths$ is in fact Church and weakly Church random, respectively, for the interval forecasts~$\churchI$ and~$\weakchurchI$.
\begin{proposition} \label{prop:churches}
Consider any~$\random \in \set{\ch,\wch}$ and any path~$\pth \in \pths$.
Then $\randomI$ is the smallest interval forecast that $\pth$ is \random-random for.
%$\emptyset \neq \weakchurchI \subseteq \churchI$, $\pth$ is Church random for an interval forecast~$I \in \intervals$ if and only if $\intervals_{\crosssymbol}(\pth) \subseteq I$, and~$\pth$ is weakly Church random for~$I$ if and only if $\weakchurchI \subseteq I$. Then
\begin{comment}
\begin{enumerate}[label=\upshape(\roman*),leftmargin=*,noitemsep,topsep=0pt]
\item $\emptyset \neq \weakchurchI \subseteq \churchI$; \label{prop:churches:i}
\item $\pth$ is completely \ch-random for the interval forecast~$\churchI$.
%$\bigg[\inf_{\selection \in \selectionsdense} \liminf_{n \to \infty} \frac{\sum_{k=0}^{n-1}\selection(\pthtok)\pthatkplus}{\sum_{k=0}^{n-1}\selection(\pthtok)},\sup_{\selection \in \selectionsdense} \limsup_{n \to \infty} \frac{\sum_{k=0}^{n-1}\selection(\pthtok)\pthatkplus}{\sum_{k=0}^{n-1}\selection(\pthtok)}\bigg]$;
\label{prop:churches:ii}
\item $\pth$ is completely \wch-random for the interval forecast~$\weakchurchI$. \label{prop:churches:iii}
%$\bigg[\inf_{\selection \in \selectionsfdense} \liminf_{n \to \infty} \frac{\sum_{k=0}^{n-1}\selection(\pthtok)\pthatkplus}{\sum_{k=0}^{n-1}\selection(\pthtok)},\sup_{\selection \in \selectionsfdense} \limsup_{n \to \infty} \frac{\sum_{k=0}^{n-1}\selection(\pthtok)\pthatkplus}{\sum_{k=0}^{n-1}\selection(\pthtok)}\bigg]$;
\end{enumerate}
\end{comment}
\end{proposition}

A similar result need not hold for the other four types of randomness we are considering here.
As an illustrative example, consider the non-stationary but temporal precise forecasting system~$\frcstsystem_{\sim \nicefrac{1}{2}}$ defined, for all~$\sit \in \sits$, by
\begin{equation*} \label{eq:main:one}
\frcstsystem_{\sim \nicefrac{1}{2}}(\sit) \coloneqq \frac{1}{2}+(-1)^{\abs{\sit}}
\delta(\abs{\sit}),\textrm{ with }
\delta(n)\coloneqq e^{-\frac{1}{n+1}}\sqrt{e^{\frac{1}{n+1}}-1} \textrm{ for all } n \in \naturalswithzero.
\end{equation*}
It has been proved that if a path~$\pth \in \pths$ is computably random for~$\frcstsystem_{\sim \nicefrac{1}{2}}$, then $\pth$ is Church random and almost computably random for the stationary precise model $\nicefrac{1}{2}$, whilst not being computably random for~$\nicefrac{1}{2}$ \cite{CoomanBock2017}.
%Hence, in this case, the smallest intervals for which $\pth$ is Church random and (almost) computably random both equal $\nicefrac{1}{2}$, and therefore coincide.

While in general $\randomI$ may not be the smallest interval forecast that a path~$\pth \in \pths$ is \random-random for, De Cooman and De Bock have effectively proved that for $\random \in \set{\wml,\co,\s}$, every path~$\pth \in \pths$ is almost \random-random for~$\randomI$, essentially because the corresponding sets~\(\intervals_\random(\pth)\) are then closed under finite intersections.

\begin{proposition}[{\cite[Section~8]{CoomanBock2021}}] \label{prop:almost:random}
Consider any~$\random \in \set{\wml,\co,\s}$ and any path~$\pth \in \pths$.
Then $\randomI$ is the smallest interval forecast for which $\pth$ is almost \random-random.
\end{proposition}

It should be noted that there is no mention of Martin-Löf randomness in Propositions~\ref{prop:churches} and~\ref{prop:almost:random}.
%we did not introduce a similar result for Martin-Löf randomness.
Indeed, it is an open problem whether every path~$\pth \in \pths$ is (almost) \ml-random for the interval forecast~$\mlI$.
We can however provide a partial answer by focusing on paths~$\pth \in \pths$ that are \ml-random for a computable precise forecasting system~$\frcstsystem \in \frcstsystems$.
\begin{proposition} \label{prop:ml:almost:completely}
If a path~$\pth \in \pths$ is \ml-random for a computable precise forecasting system~$\frcstsystem \in \frcstsystems$, then $\mlI$ is the smallest interval forecast for which $\pth$ is almost \ml-random.
\end{proposition}

\begin{comment}
\begin{definition}
A path~$\pth \in \pths$ is \emph{(weakly)} \emph{Church random} for an interval forecast~$I \in \intervals$ if for every recursive \emph{(}temporal\emph{)} selection process~$\selection$ for which $\lim_{n \to \infty}\selectionsum=\infty$, it holds that
% for which $\lim_{n \to \infty}\sum_{k=0}^{n-1}\selection(\pthtok)=\infty$ :
\begin{equation*}
\min I \leq 
\liminf_{n \to \infty} \frac{\sum_{k=0}^{n-1}\selection(\pthtok)\pthatkplus}{\sum_{k=0}^{n-1}\selection(\pthtok)} \leq
\limsup_{n \to \infty} \frac{\sum_{k=0}^{n-1}\selection(\pthtok)\pthatkplus}{\sum_{k=0}^{n-1}\selection(\pthtok)} \leq
\max I.
\end{equation*}
\end{definition}
\end{comment}

\section{What do smallest interval forecasts look like?} \label{sec:question2}
\begin{comment}
Classically, the above randomness notions were only introduced with respect to the precise stationary interval forecast~$\nicefrac{1}{2}$, and their similarities and differences were thoroughly studied from a precise-probabilistic perspective.
Regarding the similarities, and in accordance with \propositionref{prop:relations:interval}, if a path~$\pth \in \pths$ is Martin-Löf random for~$\nicefrac{1}{2}$, then it is also computably, Schnorr and Church random for~$\nicefrac{1}{2}$.
Concerning the differences, it was \emph{exempli gratia} proven that there is a path~$\pth \in \pths$ that is Church random but not computably random for~$\nicefrac{1}{2}$ \cite{merkle2008}.

Meanwhile, the example that we discussed in Section~\ref{??} shows that there are paths~$\pth \in \pths$ for which $\churchI=\compI=\nicefrac{1}{2}$, but which are not computably random for~$\nicefrac{1}{2}$.
\end{comment}

Having established conditions under which $\randomI$ is the smallest interval forecast $\pth$ is (almost) random for, we now set out to find an alternative expression for this interval forecast. %a path~$\pth \in \pths$ is \random-random for.
Forecasting systems will play a vital role in this part of the story; for every path~$\pth \in \pths$ and every forecasting system~$\frcstsystem \in \frcstsystems$, we consider the interval forecast~$\frcstI$ defined by
\begin{equation*}
\frcstI \coloneqq \Big[\liminf_{n \to \infty}\underline{\frcstsystem}(\pthton),\limsup_{n \to \infty}\overline{\frcstsystem}(\pthton)\Big].
\end{equation*}

When we restrict our attention to computable forecasting systems~$\frcstsystem\in\frcstsystems$, and if we assume that a path~$\pth \in \pths$ is \random-random for such a forecasting system~$\frcstsystem$, with $\random \in \set{\ml,\wml,\co,\s,\ch,\wch}$, then the forecasting system~$\frcstsystem$ imposes outer bounds on the interval forecast~$\randomI$ in the following sense.
\begin{comment}
\begin{proposition} \label{prop:bounds:ML}
If a path~$\pth \in \pths$ is Martin-Löf random for a computable forecasting system~$\frcstsystem \in \frcstsystems$, then it is almost Martin-Löf random for the interval forecast
%it is almost Martin-Löf random for the interval forecast
\begin{equation*}
\Big[\liminf_{n \to \infty}\underline{\frcstsystem}(\pthton),\limsup_{n \to \infty}\overline{\frcstsystem}(\pthton)\Big].
\end{equation*}
\end{proposition}
\end{comment}

\begin{proposition} \label{prop:bounds:outer}
For any~$\random \in \set{\ml,\wml,\co,\s,\ch,\wch}$ and any path~$\pth \in \pths$ that is \random-random for a computable forecasting system~$\frcstsystem \in \frcstsystems$: $\randomI \subseteq \frcstI$.
%it is almost Martin-Löf random for the interval forecast
%\begin{equation*}
%\randomI \subseteq \Big[\liminf_{n \to \infty}\underline{\frcstsystem}(\pthton),\limsup_{n \to \infty}\overline{\frcstsystem}(\pthton)\Big].
%\end{equation*}
\end{proposition}

\begin{comment}
\begin{proposition} \label{prop:bounds:outer:time}
For any~$\random \in \set{\ml,\wml,\co,\s}$ and any path~$\pth \in \pths$ that is \random-random for a temporal forecasting system~$\frcstsystem \in \frcstsystems$: 
%it is almost Martin-Löf random for the interval forecast
\begin{equation*}
\randomI \subseteq \Big[\liminf_{n \to \infty}\underline{\frcstsystem}(\pthton),\limsup_{n \to \infty}\overline{\frcstsystem}(\pthton)\Big].
\end{equation*}
\end{proposition}
\end{comment}

If we only consider computable \emph{precise} forecasting systems~$\frcstsystem \in \frcstsystems$ and assume that a path~$\pth \in \pths$ is \random-random for~$\frcstsystem$, with $\random \in \set{\ml,\wml,\co,\ch}$, then the forecasting system~$\frcstsystem$ completely characterises the interval forecast~$\randomI$.
\begin{theorem} \label{the:coinciding}
For any~$\random \in \set{\ml,\wml,\co,\ch}$ and any path~$\pth \in \pths$ that is \random-random for a computable precise forecasting system~$\frcstsystem \in \frcstsystems$: $\randomI=\frcstI$.
%\begin{equation*}
%\mlI=\weakmlI=\compI=\churchI=\Big[\liminf_{n \to \infty}\frcstsystem(\pthton),\limsup_{n \to \infty}\frcstsystem(\pthton)\Big].
%\end{equation*}
%\footnote{For \theoremref{the:coinciding} to hold, it actually suffices to consider computable forecasting systems~$\frcstsystem \in \frcstsystems$ that are \emph{eventually} precise in the sense that $\limsup_{n \to \infty}\underline{\frcstsystem}(\pthton)=\liminf_{n \to \infty}\overline{\frcstsystem}(\pthton)$}
\end{theorem}
%By the relations at the end of Section~\ref{sec:notionsofrandomness}, we know that it also suffices for a path~$\pth \in \pths$ to be \ml-,\wml- or \co-random for such a forecasting system.
\begin{comment}
\begin{proposition} \label{prop:bounds:church:inner}
For any~$\random \in \set{\ml,\wml,\co,\ch}$ and any path~$\pth \in \pths$ that is \random-random for a computable precise forecasting system~$\frcstsystem \in \frcstsystems$:
\begin{equation*}
\Big[\liminf_{n \to \infty}\frcstsystem(\pthton),\limsup_{n \to \infty}\frcstsystem(\pthton)\Big] \subseteq \churchI.
\end{equation*}
\end{proposition}
\end{comment}
%By the relations at the end of Section~\ref{sec:notionsofrandomness}, we know that the conclusions of \propositionref{prop:bounds:church:inner} still hold when a path~$\pth \in \pths$ is \ml-,\wml- or \co-random for a computable precise forecasting system~$\frcstsystem \in \frcstsystems$.
%When the computable precise forecasting systems~$\frcstsystem \in \frcstsystems$ are also temporal, then they specify inner bounds for the smallest interval forecast~$\weakchurchI$ for which a path~$\pth \in \pths$ is weakly church random, and they do this for a larger set of randomness notions for which $\pth$ should be random.
%When also imposing time-dependence on the computable precise forecasting systems~$\frcstsystem \in \frcstsystems$, then $\frcstsystem$ also specifies 

When the computable precise forecasting systems~$\frcstsystem \in \frcstsystems$ are also temporal, this result applies to Schnorr and weak Church randomness as well.
%we can even strengthen the conclusions of \theoremref{the:coinciding}.
%In this case, the forecasting system~$\frcstsystem$ completely characterises the smallest interval forecast~$\randomI$ for which a path~$\pth \in \pths$ is \random-random, with $\random \in \set{\ml,\wml,\co,\s,\ch,\wch}$.
%-random for such a forecasting system~$\frcstsystem \in \frcstsystems$, the intervals $\randomI$ coincide for all randomness notions that have been discussed in this paper.
\begin{theorem} \label{the:coinciding:temporal}
For any~$\random \in \set{\ml,\wml,\co,\s,\ch,\wch}$ and any path~\smash{$\pth \in \pths$} that is \smash{\random}-random for a computable precise temporal forecasting system~$\frcstsystem \in \frcstsystems$: $\randomI=\frcstI$.
%\begin{align*}
%\mlI&=\weakmlI=\compI=\schnorrI\\
%&=\churchI=\weakchurchI=\Big[\liminf_{n \to \infty}\frcstsystem(\pthton),\limsup_{n \to \infty}\frcstsystem(\pthton)\Big].
%\end{align*}
\end{theorem}

\section{When do these smallest interval forecasts coincide?} \label{sec:unalike}

%Regarding our third main question, we investigate how these smallest interval forecasts compare.
%In this way, 
Finally, we put an old question into a new perspective: to what extent are randomness notions different?
We take an `imprecise' perspective here, by comparing the smallest interval forecasts for which a path~$\pth \in \pths$ is (almost) \random-random, with $\random \in \set{\ml,\wml,\co,\s,\ch,\wch}$.
As we will see, it follows from our previous exposition that there are quite some paths for which these smallest interval forecasts coincide.
%Regarding our third main question, we investigate how these smallest interval forecasts compare.
%In this way, we put an old question into a new perspective; how different are the considered (im)precise-probabilistic randomness notions, in the sense that we wonder when the interval forecasts~$\randomI$ for which a path~$\pth \in \pths$ is (almost) completely \random-random coincide, with $\random \in \set{\ml,\wml,\co,\s,\ch,\wch}$.

Let us start by considering a path~$\pth \in \pths$ that is \ml-random for some computable precise forecasting system~$\frcstsystem \in \frcstsystems$; similar results hold when focusing on weaker notions of randomness.
We know from \equationref{eq:inclusions:comp} that $\pth$ is then also \wml-, \co- and \ch-random for~$\frcstsystem$.
By invoking Propositions~\ref{prop:churches}, \ref{prop:almost:random} and ~\ref{prop:ml:almost:completely}, we infer that $\randomI$ is the smallest interval forecast that $\pth$ is (almost) \random-random for, for any $\random \in \set{\ml,\wml,\co,\ch}$.
Moreover, by \theoremref{the:coinciding}, these smallest interval forecasts all equal $\frcstI$ and therefore coincide, i.e., $\mlI=\weakmlI=\compI=\churchI=\frcstI$.
%Consequently, by \theoremref{the:coinciding}, the smallest interval forecasts for which $\pth$ is (almost) \ml-, \wml-, \co- and \ch-random coincide: $\mlI=\weakmlI=\compI=\churchI=\frcstI$. 

By only looking at temporal computable precise forecasting systems~$\frcstsystem \in \frcstsystems$, we can even strengthen these conclusions.
For example, using a similar argument as before---but using \theoremref{the:coinciding:temporal} instead of \ref{the:coinciding}---we see that if $\pth$ is \ml-random for such a forecasting system~$\frcstsystem$, then the smallest interval forecasts~$\randomI$ for which $\pth$ is (almost) \random-random coincide for all six randomness notions that we consider.

Looking at these results, the question arises whether there are paths~$\pth \in \pths$ for which the various interval forecasts~$I_\random(\pth)$ do not coincide.
It turns out that such paths do exist.
%it is indeed possible to contradict the conclusions of \theoremref{the:coinciding} and~\ref{the:coinciding:temporal}.
We start by showing that the smallest interval forecasts~$\compI$ and~$\schnorrI$ for which a path~$\pth \in \pths$ is respectively almost \co- and almost \s-random do not always coincide; this result is mainly a reinterpretation of a result in \cite{CoomanBock2021,Wang1996}.
%assumption of randomness for a computable precise temporal forecasting systems is needed in \theoremref{the:coinciding:temporal} since there is at least one path~$\pth \in \pths$ for which $\compI \neq \schnorrI$
%computable randomness and Schnorr randomness can be quite unalike, in the sense that there is a path~$\pth \in \pths$ for which $\compI \neq \schnorrI$, thereby contrasting with the conclusions of \theoremref{the:coinciding:temporal}; 

\begin{proposition} \label{prop:counterexample1}
There is a path~$\pth \in \pths$ such that $\schnorrI = \nicefrac{1}{2} \in \sqgroup{\nicefrac{1}{2},1} \subseteq \compI$.
%that is Schnorr random for~$\nicefrac{1}{2}$, but not computably random for any interval forecast~$I \subset \group{0,1}$.
\end{proposition}
% By \theoremref{the:coinciding:temporal}, we know that the above path is not Martin-Löf random for any computable precise temporal forecasting system~$\frcstsystem \in \frcstsystems$, while it is Martin-Löf random for at least one computable interval forecast~$I \in \intervals$.
%Clearly, the path~$\pth \in \pths$ in \propositionref{prop:counterexample1} is not computably random for any precise computable temporal forecasting system~$\frcstsystem \in \frcstsystems$.
%Otherwise, the interval forecasts~$\compI$ and~$\schnorrI$ would coincide.
%In line with the recent results in \cite{CoomanBock2021}, this path functions as an alternative example to show that there is at least one path that is computably random for some non-singular interval forecast~$I \in \intervals$, but not for any precise computable forecasting system.

We are also able to show that there is a path~$\pth \in \pths$ such that $\compI=\nicefrac{1}{2}$ is the smallest interval forecast it is almost \co-random for, whereas $\pth$ is not almost \ml-random for $\nicefrac{1}{2}$; for this result, we have drawn inspiration from~\cite{Schnorr1971}.
%that is almost computably random for the smallest interval forecast $\compI$, whilst not being almost Martin-Löf random this interval forecast; for this result, we have drawn inspiration from~\cite{Schnorr1971}.
%assumption of randomness for a precise computable forecasting system is needed in \theoremref{the:coinciding} since there is at least one path~$\pth \in \pths$ for which $\mlI\neq\compI$.
%Martin-Löf and computable randomness can be quite unalike, in the sense that there is a path~$\pth \in \pths$ for which $\mlI \neq \compI$, thereby contrasting with the conclusions of \theoremref{the:coinciding}.
\begin{proposition} \label{prop:counterexample2}
For every $\delta \in \group{0,\nicefrac{1}{2}}$, there is a path~$\pth \in \pths$ such that $\compI=\nicefrac{1}{2}$ and $I \notin \mlintervals$ for any $I \in \intervals$ such that $I \subseteq \sqgroup{\nicefrac{1}{2}-\delta,\nicefrac{1}{2}+\delta}$.
%that is \co-random but not almost \ml-random for the precise interval forecast~$\nicefrac{1}{2} \in \intervals$.
%$\compI = \nicefrac{1}{2} \neq \mlI$.
%For every~$\delta \in \group{0,\nicefrac{1}{2}}$, there is a path~$\pth \in \pths$ such that $\compI = \nicefrac{1}{2}$ and~$\abs{\mlI} \geq \nicefrac{1}{2} - \delta$.% \in \sqgroup{\nicefrac{1}{2},1-\delta} \subseteq \mlI$.
%For any rational interval forecast~$I \subset \group{0,1}$, there is a path~$\pth \in \pths$ that is computably random for~$\nicefrac{1}{2}$, but not Martin-Löf random for I.
\end{proposition}
%Similarly, by \theoremref{the:coinciding}, we know that the above path is not Martin-Löf random for any computable precise forecasting system~$\frcstsystem \in \frcstsystems$, while it is Martin-Löf random for at least one computable interval forecast~$I \in \intervals$.
Clearly, the path~$\pth \in \pths$ in \propositionref{prop:counterexample2} cannot be Martin-Löf random for a precise computable forecasting system~$\frcstsystem \in \frcstsystems$, because otherwise, the interval forecasts~$\compI$ and~$\mlI$ would coincide by \equationref{eq:inclusions:comp} and \theoremref{the:coinciding}, and~$\pth$ would therefore be almost Martin-Löf random for~$\nicefrac{1}{2}$ by \propositionref{prop:ml:almost:completely}, contradicting the result.
So the path $\pth$ in this result is an example of a path for which we do not know whether there is a smallest interval forecast that $\pth$ is almost Martin-Löf random for.
However, if there is such a smallest interval forecast, then \propositionref{prop:counterexample2} shows it is definitely not equal to $\nicefrac{1}{2}$; due to \corollaryref{cor:relations:interval}, it must then strictly include $\nicefrac{1}{2}$.
%it will necessarily be~$\mlI$.

% In line with the recent results in \cite{CoomanBock2021}, the path~\(\pth\) in \propositionref{prop:counterexample2} can be used to construct an example that shows that there is at least one path that is Martin-Löf random for some non-singular interval forecast~$I \in \intervals$, but not for any precise computable forecasting system~$\frcstsystem \in \frcstsystems$.

\section{Conclusions and future work}

%We conclude that if a path~$\pth \in \pths$ is \random-random for a precise computable (temporal) forecasting system, with $\random \in \set{\ml,\wml,\co,\s,\ch,\wch}$, then the smallest interval forecasts coincide for (some of) the
%By putting on an imprecise pair of goggles and looking at several precise-probabilistic randomness notions, 
We've come to the conclusion that various (non-stationary) precise-probabilistic randomness notions in the literature are, in some respects, not that different; if a path is random for a computable precise (temporal) forecasting system, then the smallest interval forecast for which it is (almost) random coincides for several randomness notions.
The computability condition on the precise forecasting system is important for this result, but we don't think it is that big a restriction.
After all, computable forecasting systems are those that can be computed by a finite algorithm up to any desired precision, and therefore, they are arguably the only ones that are of practical relevance.
%By putting on an imprecise pair of goggles and looking at several precise-probabilistic randomness notions, we were able to reveal that the smallest interval forecasts for which a path is (almost) random often coincide when focusing on computable precise (temporal) forecasting systems.

%For those readers who are quite familiar with adopting precise probabilities in statistics, imposing computability on the forecasting systems should not feel as a big restriction.
%After all, randomness provides a physical interpretation for probability theory \cite{ZvonkinLevin1970}, and in statistics we only try to learn computable probability models.

An important concept that made several of our results possible was that of almost randomness, a notion that is closely related to randomness but is---slightly---easier to satisfy. In our future work, we would like to take a closer look at the difference between these two notions.
%Meanwhile, we question to what extent the considered randomness notions really differ for a path $\pth \in \pths$ when the smallest interval forecasts for which $\pth$ is random and almost random coincide.
%From a theoretical point of view, the difference is of course clear.
In particular, the present discussion, together with our work in \cite{floris2021}, makes us wonder to what extent the distinction between them is relevant in a more practical context.

We also plan to continue investigating the open question whether there is for every path some smallest interval forecast for which it is (almost) Martin-Löf random.
Finally, there is still quite some work to do in finding out whether the randomness notions we consider here are all different from a stationary imprecise-probabilistic perspective, in the sense that there are paths for which the smallest interval forecasts for which they are (almost) random do not coincide. %for which they are (almost) random do not coincide.

%the two corresponding smallest interval forecasts for which a path is (almost) random do not coincide.

%they are different from a stationary imprecise perspective, in the sense that

%As a general remark, there 

%we hope you are 

%As a general remark, we empathise that there is still a

%s a general remark, we like to empathise, as is clear from our discussion, that we are c

\section*{Acknowledgments}
Floris Persiau’s research was supported by FWO (Research Foundation-Flanders), project number 11H5521N.

\bibliographystyle{splncs04}
\bibliography{biblio.bib}

\begin{ArxiveExt}

\section*{Appendix A}

The field of computability theory studies what it means for a mathematical object to be implementable.
As its basic building blocks, it has natural partial maps $\phi \colon \mathcal{D} \to \naturalswithzero$, which are partial maps from $\mathcal{D}$, which denotes a (countably infinite) set that can be encoded by a finite alphabet,
%\textcolor{red}{effective map onto the natural numbers?}
to the non-negative integers. %p. 8 (62)
Examples of such sets~$\mathcal{D}$ are given by $\naturalswithzero$, $\sits$, $\sits \times \posspace$, $\sits\times\naturalswithzero$ and~$\sits\times\posspace\times\naturalswithzero$.
A natural partial map~$\phi$ is called \emph{partial recursive} if it can be computed by a Turing machine.
If the Turing machine halts for all inputs $d \in \mathcal{D}$, that is, if the Turing machine computes the number $\phi(d)$ in a finite number of steps for every $d \in \mathcal{D}$, then the map $\phi$ computed is defined for all arguments and we call it \emph{total recursive}, or simply \emph{recursive}.
Otherwise, if there is at least one $d \in \mathcal{D}$ for which the Turing machine does not halt, that is, if there is an input $d \in \mathcal{D}$ for which the Turing machine computes forever, then $\phi$ is called \emph{strictly partial recursive}.
By the Church--Turing thesis, the natural map $\phi$ being partial recursive is equivalent to the existence of a finite algorithm that given the input $d \in \mathcal{D}$, outputs the non-negative integer $\phi(d)\in\naturalswithzero$ if $\phi(d)$ is defined, and that never finishes otherwise \cite{DowneyHirschfeldt2010,LiVitanyi2008,Pour-ElRichards2016}.

%For example, a natural function $\selectionf \colon \naturalswithzero \to \{0,1\}$ is called a recursive \emph{selection function} if there is some finite algorithm that upon the input $n \in \naturalswithzero$, outputs the binary digit $\selectionf(n)\in\{0,1\}$.
%Similarly, a natural function $\selectionf \colon \sits \to \{0,1\}$ is called a recursive \emph{selection process} if there is some finite algorithm that upon the input $n \in \naturalswithzero$, outputs the binary digit $\selectionf(n)\in\{0,1\}$.

The notion of (partial) recursive natural maps $\phi$ can be extended to maps that have a range of rational numbers $\rationals$.
We call a rational map~$q \colon \mathcal{D} \to \rationals$ \emph{recursive} if there are three recursive natural maps $a,b,c\colon \mathcal{D}\to\naturalswithzero$ such that 
\begin{equation*}
b(d)\neq 0 \textrm{ and } q(d)=(-1)^{c(d)}\frac{a(d)}{b(d)} \textrm{ for all } d \in \mathcal{D}.
\end{equation*}
Since a finite number of algorithms can always be combined into one finite algorithm, a rational map~$q$ is recursive if and only if there is a finite algorithm that given the input $d \in \mathcal{D}$, outputs the rational number~$q(d)\in \rationals$ \cite{MichaelSipser2006}.
Similarly, a rational partial map $q$ is called \emph{partial recursive} if there is a finite algorithm that given the input $d \in \mathcal{D}$, outputs the rational number $q(d) \in \rationals$ if $q(d)$ is defined, and that doesn't terminate for the input $d$ otherwise.

We can use recursive rational maps to provide several notions of what it means for a real map~$r\colon\mathcal{D}\to\reals$ to be implementable.
We call a real map~$r\colon \mathcal{D} \to \reals$ \emph{lower semicomputable} if there is some recursive rational map~$q \colon \mathcal{D}\times\naturalswithzero \to \rationals$ such that
\begin{equation*}
q(d,n+1) \geq q(d,n) \textrm{ and } r(d)= \lim_{m \to \infty} q(d,m) \textrm{ for all } d \in \mathcal{D} \textrm{ and } n \in \naturalswithzero.
\end{equation*}
This means that there is some finite algorithm that, for every $d \in \mathcal{D}$, allows us to approach the real number~$r(d)$ from below---but without knowing, for any given $n \in \naturalswithzero$, how good the lower bound $q(d,n)$ is.
Correspondingly, we call a real map~$r$ \emph{upper semicomputable} if the real map~$-r$ is lower semicomputable.
A real multiplier process $\multprocess\colon \sits \to \gambles$ is then called \emph{lower \textnormal{(}upper\textnormal{)} semicomputable} if it is lower (upper) semicomputable as a real map on $\sits \times \posspace$ that maps any $(\sit,x)\in\sits\times\posspace$ to $\multprocess(\sit)(x)$.
%Since finitely many algorithms can be combined into one, this is equivalent to the existence of a recursive rational map~$q \colon \sits \times \posspace \times \naturalswithzero \to \rationals$ such that $q(\sit,x,n+1) \geq q(\sit,x,n)$ and $\multprocess(\sit)(x)= \lim_{m \to \infty} q(\sit,x,m)$ for all $\sit \in \sits$, $x \in \posspace$ and $n \in \naturalswithzero$.

Moreover, a real map~$r$ is called \emph{computable} if it is both lower and upper semicomputable.
Equivalently \cite{CoomanBock2017}, a real map~$r \colon \mathcal{D} \to \reals$ is computable if and only if there is some recursive rational map~$q\colon\mathcal{D}\times\naturalswithzero \to \rationals$ such that
\begin{equation*}
\abs{r(d)-q(d,n)}<2^{-n} \textrm{ for all } d \in \mathcal{D} \textrm{ and } n \in \naturalswithzero.
\end{equation*}
This means that there is some finite algorithm that, for every $d \in \mathcal{D}$ and $n \in \naturalswithzero$, allows us to approximate the real number~$r(d)$ with a precision of $2^{-n}$.
A real number $x \in \reals$ is then called computable if there is some recursive rational map~$q\colon \naturalswithzero \to \rationals$ such that $\abs{x-q(n)}<2^{-n}$ for all $n \in \naturalswithzero$.
An interval forecast~$I \in \intervals$ is called computable if the reals $\min I$ and~$\max I$ are both computable.
Similarly, a forecasting system~$\frcstsystem \in \frcstsystems$ is called computable if the real processes~$\lfrcstsystem$ and~$\ufrcstsystem$ are both computable.

We will also consider infinite sequences $(q_i\colon \mathcal{D}\to \rationals)_{i \in \naturalswithzero}$ of (partial) recursive rational maps.
Such a sequence is called a \emph{recursive enumeration} of (partial) recursive rational maps if there is a (partial) recursive rational map~$q\colon \naturalswithzero \times \mathcal{D} \to \rationals$ such that for every~$i \in \naturalswithzero$ and $d \in \mathcal{D}$: $q(i,d)=q_i(d)$ if $q_i(d)$ is defined, and $q(i,d)$ is undefined otherwise. %\textcolor{red}{link tussen undefined and doesn't halt preciseren}

We will use recursive enumerations of recursive rational maps, to define what it means for an infinite sequence $(r_i\colon \mathcal{D}\to \reals )_{i \in \naturalswithzero}$ of lower semicomputable real maps to be \emph{recursively enumerable}.
Such a sequence is called a \emph{recursive enumeration} of lower semicomputable real maps if there is a recursive rational map~$q\colon \naturalswithzero \times \mathcal{D}\times\naturalswithzero \to \rationals$ such that for every~$i \in \naturalswithzero$:
\begin{equation*}
q(i,d,n+1) \geq q(i,d,n) \textrm{ and } r_i(d)= \lim_{m \to \infty} q(i,d,m) \textrm{ for all } d \in \mathcal{D} \textrm{ and } n \in \naturalswithzero.
\end{equation*}
An infinite sequence $(\multprocess_i\colon \sits \to \gambles)_{i \in \naturalswithzero}$ of lower semicomputable real multiplier processes is then called a \emph{recursive enumeration} of lower semicomputable real multiplier processes if it is recursively enumerable as an infinite sequence of real maps on $\sits \times \posspace$ that maps any $(i,\sit,x) \in \naturalswithzero \times \sits \times \posspace$ to $\multprocess_i(\sit ,x)$.
%\textcolor{red}{Heb je dit wel nog nodig?}
\begin{comment}
there is a recursive rational map ~$q\colon \naturalswithzero \times \sits \times \posspace \times \naturalswithzero \to \rationals$ such that for every~$i \in \naturalswithzero$:
\begin{equation*}
q(i,\sit,x,n+1) \geq q(i,\sit,x,n) \textrm{ and } \multprocess_i(\sit)(x)= \lim_{m \to \infty} q(i,\sit,x,m) 
\end{equation*}
for all $\sit \in \sits$, $x \in \posspace$ and $n \in \naturalswithzero$.
\end{comment}

It turns out there is a recursive enumeration of all partial recursive rational maps.
\begin{proposition}[\cite{DowneyHirschfeldt2010,LiVitanyi2008}] \label{prop:enumeration}
There is a recursive enumeration $(q_i)_{i \in \naturalswithzero}$ of partial recursive rational maps $q_i\colon \mathcal{D} \times \naturalswithzero \to \rationals$ that contains every partial recursive map $q\colon \mathcal{D} \times \naturalswithzero \to \rationals$.
\end{proposition}

In the following corollary, we will consider non-negative \emph{extended} real maps $r \colon \mathcal{D} \to \sqgroup{0,+\infty}$, which are maps that define for every~$d \in \mathcal{D}$ a non-negative real number or $+\infty$.
Such a non-negative extended real map is called lower semicomputable if there is a recursive rational map $q\colon \mathcal{D} \times \naturalswithzero \to \rationals$ such that $q(d,n+1) \geq q(d,n)$ and $r(d)= \lim_{m \to \infty} q(d,m)$ for all $d \in \mathcal{D}$ and $n \in \naturalswithzero$.
An infinite sequence $(r_i\colon \mathcal{D} \to \sqgroup{0,+\infty})_{i \in \naturalswithzero}$ of lower semicomputable non-negative extended real maps is called a \emph{recursive enumeration} of lower semicomputable non-negative extended real maps if there is a recursive rational map~$q\colon \naturalswithzero \times \mathcal{D} \times \naturalswithzero \to \rationals$ such that, for every~$i \in \naturalswithzero$, $q(i,d,n+1) \geq q(i,d,n)$ and $r_i(d)= \lim_{m \to \infty} q(i,d,m)$ for all $d \in \mathcal{D}$ and $n \in \naturalswithzero$.
% With every such map~$r$, we associate a non-negative \emph{extended} real gamble process~$\multprocess_r$ as follows: $\multprocess_r(\sit)(x)\coloneqq r(\sit,x)$ for all~$\sit \in \sits$ and~$x \in \posspace$.

\begin{corollary} \label{cor:enumeration}
There is a recursive enumeration $(r_i)_{i \in \naturalswithzero}$ of lower semicomputable non-negative extended real maps $r_i\colon \mathcal{D} \to \sqgroup{0,+\infty}$ that contains every lower semicomputable map $r\colon \mathcal{D} \to \sqgroup{0,+\infty}$.
\end{corollary}

\begin{proof}
By \propositionref{prop:enumeration}, there is a partial recursive rational map $q\colon \naturalswithzero \times \mathcal{D} \times \naturalswithzero \to \rationals$ such that for every partial recursive rational map $q'\colon \mathcal{D} \times \naturalswithzero \to \rationals$, there is some $i \in \naturalswithzero$ such that, for all $d \in \mathcal{D}$ and $n \in \naturalswithzero$, $q'(d,n)=q(i,d,n)$ if $q'(d,n)$ is defined, and $q(i,d,n)$ is undefined otherwise.
Since $q$ is a partial recursive rational map, there is some Turing machine $U$ that, when given the input $(i,d,n) \in \naturalswithzero\times\mathcal{D}\times\naturalswithzero$, outputs the rational $q(i,d,n)$ in a finite number of steps when $q(i,d,n)$ is defined, and computes forever otherwise.
%when $q(i,d,n)$ is undefined, and .

From the partial recursive rational map $q$, we will construct a total recursive non-negative rational map $q^\ast\colon\naturalswithzero\times\mathcal{D}\times\naturalswithzero\to\nonnegrationals$ as follows; for all $i,n \in \naturalswithzero$ and $d \in \mathcal{D}$, we put $q^\ast(i,d,0)\coloneqq0$ and
\begin{multline}
q^\ast(i,d,n+1)\coloneqq \max\big\{0,\{q(i,d,l) \in \rationals\colon 0\leq l \leq n+1 \textrm{ and } \\
U \textrm{ computes } q(i,d,l) \textrm{ in } \leq n+1 \textrm{ steps}\}\big\}. \label{eq:steps}
\end{multline}
It is easy to verify that the map $q^\ast$ is total recursive, non-negative and rational, and satisfies $q^\ast(i,d,n+1) \geq q^\ast(i,d,n)$ for all $i,n \in \naturalswithzero$ and $d \in \mathcal{D}$.

Since $q^\ast(i,d,n+1) \geq q^\ast(i,d,n)$ and $q^\ast(i,d,n)\geq 0$ for all $i,n \in \naturalswithzero$ and $d \in \mathcal{D}$, it follows that for every $i \in \naturalswithzero$, the extended real map $f_i\colon\mathcal{D}\to \sqgroup{0,+\infty}$, defined by $f_i(d)\coloneqq\lim_{n \to \infty}q^\ast(i,d,n)$ for all $d \in \mathcal{D}$, is non-negative and lower semicomputable.
Consequently, the infinite sequence $(f_i)_{i \in \naturalswithzero}$ is a recursive enumeration of lower semicomputable non-negative extended real maps.
Moreover, $(f_i)_{i \in \naturalswithzero}$ is a recursive enumeration of all lower semicomputable non-negative extended real maps.
Indeed, consider any lower semicomputable non-negative extended real map $f\colon\mathcal{D}\to\sqgroup{0,+\infty}$.
By definition, there is some total recursive rational map $q_f\colon \mathcal{D}\times \naturalswithzero \to \rationals$ such that 
\begin{equation*}
q_f(d,n+1) \geq q_f(d,n) \textrm{ and } f(d)= \lim_{m \to \infty} q_f(d,m) \textrm{ for all } d \in \mathcal{D} \textrm{ and } n \in \naturalswithzero.
\end{equation*}
Since every total recursive map is partial recursive, there is some $i \in \naturalswithzero$ such that $q_f(d,n)=q(i,d,n)$ for all $d \in \mathcal{D}$ and $n \in \naturalswithzero$.
We intend to show that, for the same $i \in \naturalswithzero$, it holds that $f(d)=\lim_{n \to \infty}q^\ast(i,d,n)$ for all $d \in \mathcal{D}$.
Since $q(i,d,n+1)\geq q(i,d,n)$ for all $d \in \mathcal{D}$ and $n \in \naturalswithzero$, it follows from \equationref{eq:steps} that $q(i,d,n) \geq q^\ast(i,d,n)$ for all $d \in \mathcal{D}$ and $n \in \naturalswithzero$.
Since $q^\ast(i,d,n+1) \geq q^\ast(i,d,n)$ for all $d \in \mathcal{D}$ and $n \in \naturalswithzero$, it therefore follows that $f(d)=\lim_{n \to \infty} q(i,d,n) \geq \lim_{n \to \infty} q^\ast(i,d,n)$ for all $d \in \mathcal{D}$.
Assume \emph{ex absurdo} that there is some $d \in \mathcal{D}$ such that $f(d)= \lim_{n \to \infty} q(i,d,n) > \lim_{n \to \infty} q^\ast(i,d,n)$.
Consequently, there is some $N \in \naturalswithzero$ such that $q(i,d,N) > q^\ast(i,d,n)$ for all $n \in \naturalswithzero$. 
However, we know that $U$ computes the rational $q(i,d,N)$ in a finite number of steps; assume that this number of steps equals $M \in \naturalswithzero$.
It now follows that
\begin{multline*}
q^\ast(i,d,\max\{M,N\}) = \max\big\{0,\{q(i,d,l) \in \rationals\colon 0\leq l \leq \max\{M,N\} \textrm{ and } \\
U \textrm{ computes } q(i,d,l) \textrm{ in } \leq \max\{M,N\} \textrm{ steps}\}\big\} \\
\begin{aligned}
&\geq 
\max\big\{0,\{q(i,d,N) \in \rationals \colon U \textrm{ computes } q(i,d,N) \textrm{ in } \leq M \textrm{ steps}\}\big\} \\
&= \max\{0,q(i,d,N)\}
\geq q(i,d,N),
\end{aligned}
\end{multline*}
a contradiction.
\qed
\end{proof}

\section*{Appendix B}

In this part of the Appendix, we have gathered all proofs, and all additional material necessary for understanding the argumentation in these proofs.

\section*{Additional material for Section~\ref{sec:betting}}

In our proofs, we will use an operator to characterise whether a gamble $\gamble \in \gambles$ is allowed by an interval forecast~$I \in \intervals$ or not. 
We associate with every interval forecast~$I \in \intervals$ the so-called \emph{upper expectation} $\uex_I \colon \gambles \to \reals$, defined by
\begin{align}
\uex_I(\gamble)
\coloneqq& 
\max_{p \in I} \{ p \gamble(1)+(1-p)\gamble(0) \} \textrm{ for all } \gamble \in \gambles. \label{eq:uex:1} 
\end{align}
Clearly, a gamble~$\gamble \in \gambles$ is allowable for an interval forecast~$I\in \intervals$ if and only if its upper expectation $\uex_I(\gamble)$ is non-positive, i.e., $\iff \uex_I(\gamble)\leq 0$.
%It turns out that a gamble~$\gamble \in \gambles$ is allowable for an interval forecast~$I\in \intervals$ if and only if its upper expectation $\uex_I(\gamble)$ is non-positive.
%\begin{proposition}[{\cite[Proposition 2]{floris2021}}]
%Consider any gamble~$\gamble \in \gambles$ and any interval forecast~$I \in \intervals$.
%Then $\gamble$ is an allowable gamble for~$I$ if and only if $\uex_I(\gamble)\leq0$.
%\end{proposition}

It will be convenient to have the following properties at our disposal.
For all~$\gamble \in \gambles$, it readily follows from \equationref{eq:uex:1} that
\begin{align}
\uex_I(\gamble)
=&
\max\{\min I \gamble(1)+(1-\min I)\gamble(0),\max I \gamble(1)+(1-\max I)\gamble(0)\} \label{eq:uex:2} \\
=&
\begin{cases}
\min I \gamble(1)+(1-\min I)\gamble(0) &\textrm{if } \gamble(1) \leq \gamble(0) \\
\max I \gamble(1)+(1-\max I)\gamble(0) &\textrm{if } \gamble(1) > \gamble(0).
\end{cases} \label{eq:uex:3}
\end{align}
The upper expectation operator $\uex_I$ also satisfies the following properties \cite{ItIP2014,walley1991}.\footnote{We note that \ref{prop:coherence:uniformcontinuity} is usually presented as a property that requires uniform convergence.
However, since $\posspace$ is a finite sample, it is easy to see that uniform convergence is equivalent to pointwise convergence.}

\begin{proposition}
Consider any interval forecast \(I \in \intervals\).
Then for all gambles \(\gamble,g \in \gambles\), all sequences of gambles $(\gamble_n)_{n \in \naturalswithzero} \in \gambles^\naturalswithzero$, and all \(\mu \in \reals\) and \(\lambda\geq 0\):
\begin{enumerate}[label=\upshape{C}\arabic*.,ref=\upshape{C}\arabic*,leftmargin=*,itemsep=0pt]
\item\label{axiom:coherence:bounds} \(\min \gamble \leq \uex_I(\gamble) \leq \max\gamble\)  \hfill{\upshape[boundedness]}
\item\label{axiom:coherence:homogeneity} \(\uex_I(\lambda \gamble) = \lambda \uex_I(\gamble)\) \hfill{\upshape[non-negative homogeneity]}
\item\label{axiom:coherence:subsupadditivity} \(\uex_I(\gamble+g) \leq \uex_I(\gamble) + \uex_I(g)\) \hfill{\upshape[subadditivity]}
\item\label{axiom:coherence:constantadditivity} \(\uex_I( \gamble + \mu) = \uex_I(\gamble) + \mu\) \hfill{\upshape[constant additivity]}
\item \label{prop:coherence:increasingness} if $\gamble \leq g $ then $\uex_I(\gamble) \leq \uex_I(g)$ \upshape\hfill[increasingness]
\item \label{prop:coherence:uniformcontinuity} if $\lim_{n \to \infty} \gamble_n=\gamble$ then $\lim_{n \to \infty}\uex_I(\gamble_n) = \uex_I(\gamble)$. \upshape\hfill[pointwise convergence]
\end{enumerate}
\end{proposition}

\begin{comment}
We will also come across an operator that is closely related to $\uex_I$; we consider the conjugate \emph{lower expectation} $\lex_I\colon\gambles\to\reals$ defined by
\begin{equation*}
\lex_I(\gamble)\coloneqq-\uex_I(-\gamble) \textrm{ for all } \gamble \in \gambles.
\end{equation*}
\end{comment}

For every forecasting system~$\frcstsystem \in \frcstsystems$, there is a close connection between test supermartingales $\test \in \tests{\frcstsystem}$ and real multiplier processes~$\multprocess$.
To reveal this connection, we introduce the following terminology; a real multiplier process~$\multprocess$ is called a real supermartingale multiplier for the forecasting system~$\frcstsystem$ if $\uex_{\frcstsystem(\sit)}(\multprocess(\sit))\leq 1$ for all~$\sit \in \sits$.

\begin{corollary}[{\cite{floris2021}}] \label{cor:from:M:to:D}
Consider a forecasting system~\(\frcstsystem\) and a (recursive) positive rational test supermartingale \(\test \in \tests{\frcstsystem}\).
Then, \(\multprocess_\test\) is a (recursive) positive rational supermartingale multiplier for \(\frcstsystem\).
\end{corollary}

\begin{proof}
This follows immediately from Proposition~5 and 6 in \cite{floris2021}.
\qed
\end{proof}

\begin{corollary}[{\cite{floris2021}}] \label{cor:from:D:to:M}
Consider a forecasting system~\(\frcstsystem\) and a (recursive) non-negative real supermartingale multiplier $\multprocess$ for~$\frcstsystem$.
Then, \(\mint\) is a (recursive) test supermartingale for~$\frcstsystem$.
\end{corollary}

\begin{proof}
This follows immediately from Proposition~4 and 7 in \cite{floris2021}.
\qed
\end{proof}

\section*{Additional material for Section~\ref{sec:notionsofrandomness}}

In what follows, we will use the following properties of Martin-Löf randomness and computable randomness.
%need to know that a path~$\pth \in \pths$ that is Martin-Löf or computably random for some forecasting system~$\frcstsystem \in \frcstsystems$ satisfies the following properties.

\begin{proposition} \label{prop:mlrandomness}
A path~\(\pth~\in~\pths\) is \ml-random for a forecasting system~\(\frcstsystem \in \frcstsystems\) if and only if every lower semicomputable positive test supermartingale~\(\test~\in~\testsml{\frcstsystem}\) is bounded on~$\pth$.
\end{proposition}

\begin{proof}
Consider any lower semicomputable test supermartingale \(\test \in \testsml{\frcstsystem}\), and the process~$\test'$ defined by $\test'(\sit)\coloneqq \nicefrac{(\test(\sit)+1)}{2}$ for all~$\sit \in \sits$.
Clearly, $\test'$ is a lower semicomputable positive test supermartingale for~$\frcstsystem$, and 
\begin{equation*}
\limsup_{n \to \infty}\test'(\pthton)=\infty \quad \iff \quad \limsup_{n \to \infty}\test(\pthton)=\infty
\end{equation*}
for all~$\pth \in \pths$.
\qed
\end{proof}

\begin{proposition}[{\cite[Proposition 6]{sum2020persiau}}] \label{prop:equivalence:compandrec}
A path~\(\pth~\in~\pths\) is \co-random for a forecasting system~\(\frcstsystem \in \frcstsystems\) if and only if every recursive positive rational test supermartingale~\(\test \in \smash{\testscomp{\frcstsystem}}\) is bounded on~$\pth$.
\end{proposition}

\begin{comment}
\begin{proposition}[{\cite[Corollary 20]{CoomanBock2021}}] \label{prop:ml:atleastone}
For any forecasting system~$\frcstsystem \in \frcstsystems$ there is at least one path that is \ml-random for~$\frcstsystem$.
\end{proposition}
\end{comment}

\begin{proposition}[{\cite[Proposition 10]{CoomanBock2021}}] \label{prop:ml:increasing}
Consider any path~$\pth \in \pths$ and any forecasting system~$\frcstsystem \in \frcstsystems$.
If $\pth$ is \ml-random for~$\frcstsystem$, then it is also \ml-random for any forecasting system~$\frcstsystem' \in \frcstsystems$ for which $\frcstsystem(\sit) \subseteq \frcstsystem'(\sit)$ for all~$\sit \in \sits$.
\end{proposition}

\section*{Additional material for Section~\ref{sec:question1}}

{\noindent\bfseries Proof of \corollaryref{cor:relations:interval}\quad}
\begin{comment}
Consider any path $\pth \in \pths$ and the selection process $\selection$ defined by $\selection(\sit)\coloneqq 1$ for every $\sit \in \sits$.
Since $\selection$ is clearly dense along $\pth$, it follows from \definitionref{def:churchrandom} that 
\begin{align*}
\min I \leq \liminf_{n \to \infty} \frac{1}{n}\sum_{k=0}^{n-1}\pthatkplus \leq \limsup_{n \to \infty} \frac{1}{n}\sum_{k=0}^{n-1}\pthatkplus \leq \max I
\end{align*}
for every $I \in \weakchurchintervals$.
Consequently, 
\end{comment}
Consider any path $\pth \in \pths$.
Since all six randomness notions that we are considering satisfy~\propertyref{proper:non-empty} and \ref{prop:law:large:numbers}, the interval forecasts~$\randomI \coloneqq \bigcap \randomintervals$ are well-defined and non-empty for all~$\random \in \set{\ml,\wml,\break\co,\s,\ch,\wch}$.
Moreover, since the sets~$\mathcal{I}_\random(\pth)$ satisfy the relations in \propositionref{prop:relations:interval}, their intersections $I_\random(\pth)$ satisfy the following inverse relations:
\begin{equation*}
\weakchurchI \
\begin{array}{l}
\vspace{3pt} \raisebox{-3 pt}{\rotatebox[origin=c]{20}{$\subseteq$}} \ \ \churchI \ \ \raisebox{-3 pt}{\rotatebox[origin=c]{-20}{$\subseteq$}} \\
\raisebox{3 pt}{\rotatebox[origin=c]{-20}{$\subseteq$}} \ \ \schnorrI \phantom{\s} \ \ \raisebox{3 pt}{\rotatebox[origin=c]{20}{$\subseteq$}}
\end{array}
\ \compI \ \subseteq \ \weakmlI \ \subseteq \ \mlI.
\end{equation*}
\qed \ \\
In what follows, it will be convenient to have the following terminology and notation at our disposal.
We call a recursive selection process~$\selection$ \emph{dense} along a path~$\pth \in \pths$ if $\lim_{n \to \infty}\selectionsum=\infty$.
For every path~$\pth \in \pths$, we collect the corresponding recursive dense selection processes in the set~$\selectionsdense$.
Similarly, for every path~$\pth \in \pths$, we collect the corresponding recursive dense temporal selection processes in the set~$\selectionsfdense$. \\

\begin{comment}
\begin{proposition}[{\cite[Theorem 23]{CoomanBock2021}}]\label{prop:church:nonstat}
Consider a computable forecasting system~$\frcstsystem \in \frcstsystems$, a computable gamble~$\gamble \in \gambles$ and a path~$\pth \in \pths$ that is Martin-L\"of random for~$\frcstsystem$.
For any recursive dense selection process~$\selection \in \selectionsdense$, it holds that
\begin{equation*}
\liminf_{n\to+\infty}
\dfrac{\sum_{k=0}^{n-1}\selection(\pthtok)\sqgroup[\big]{f(\pthatkplus)-\lex_{\frcstsystem(\pthtok)}(f)}}
{\sum_{k=0}^{n-1}\selection(\pthtok)}
\geq0.
\end{equation*}
\end{proposition}
\end{comment}

\begin{comment}
\begin{proposition}[{\cite[Theorem 24]{CoomanBock2021}}]\label{prop:church:time}
Consider a computable forecasting system~$\frcstsystem \in \frcstsystems$, a computable gamble~$\gamble \in \gambles$ and a path~$\pth \in \pths$ that is Martin-Löf random for~$\frcstsystem$.
For any recursive temporal selection process~$\selection$ such that $\lim_{n \to \infty}\selectionsum=\infty$, it holds that
\begin{equation*}
\liminf_{n\to+\infty}
\dfrac{\sum_{k=0}^{n-1}\selection(\pthtok)\sqgroup[\big]{f(\pthatkplus)-\lex_{\frcstsystem(\pthtok)}(f)}}
{\sum_{k=0}^{n-1}\selection(\pthtok)}
\geq0.
\end{equation*}
\end{proposition}
\end{comment}

%We call a recursive selection function $\selectionf \in \selectionsf$ dense along a path~$\pth \in \pths$ if $\lim_{n \to \infty}\selectionfsum=\infty$.
%For every path~$\pth \in \pths$, we collect the corresponding recursive dense selection functions in the set~$\selectionsfdense$.

{\noindent\bfseries Proof of \propositionref{prop:churches}\quad}
We start by proving that every path~$\pth \in \pths$ is \wch-random for the interval forecast~$\weakchurchI$.
To this end, consider the real numbers 
\begin{align*}
p &\coloneqq \inf_{\selection \in \selectionsfdense} \liminf_{n \to \infty} \frac{\sum_{k=0}^{n-1}\selection(\pthtok)\pthatkplus}{\sum_{k=0}^{n-1}\selection(\pthtok)}
\shortintertext{and}
q &\coloneqq \sup_{\selection \in \selectionsfdense} \limsup_{n \to \infty} \frac{\sum_{k=0}^{n-1}\selection(\pthtok)\pthatkplus}{\sum_{k=0}^{n-1}\selection(\pthtok)}.
\end{align*}
%By considering the selection process~$\selection$, defined by $\selection(\sit)=1$ for all~$\sit \in \sits$, we clearly see that the  set~$\selectionsfdense$ is non-empty.
%Consequently, the interval $\sqgroup{p,q}$ is non-empty.
By \definitionref{def:churchrandom}, it holds that $\pth$ is \wch-random for~$\sqgroup{p,q}$, and hence, $\weakchurchI \subseteq \sqgroup{p,q}$.
%again by \definitionref{def:churchrandom}, $\pth$ is weakly Church random for any interval forecast~$I \in \intervals$ if $\sqgroup{p,q} \subseteq I$.
We now show that $\weakchurchI=\sqgroup{p,q}$.
To this end, assume \emph{ex absurdo} that there is some~$I \in \weakchurchintervals$ such that $p < \min I$ or $\max I < q$.
%$\pth$ is not weakly Church random for any interval forecast~$I \in \intervals$ if $\sqgroup{p,q} \nsubseteq I$. 
%Assume \emph{ex absurdo} that $p < \min I$ or $\max I < q$.
Consequently, there is some recursive dense temporal selection process~$\selection \in \selectionsfdense$ such that 
\begin{align*}
\liminf_{n \to \infty}\frac{\selectionsum \pthatkplus}{\selectionsum} < \min I
\shortintertext{or}
\max I < \limsup_{n \to \infty}\frac{\selectionsum \pthatkplus}{\selectionsum},
\end{align*}
which is in contradiction with \definitionref{def:churchrandom}.
We conclude that $\pth$ is \wch-random for the interval forecast~$\weakchurchI = \sqgroup{p,q}$.
% an interval forecast~$I \in \intervals$ if and only if $\sqgroup{p,q} \subseteq I$.

\begin{comment}
Consequently, the set~$\weakchurchintervals$ consists of the interval forecasts~$I \in \intervals$ for which $\sqgroup{p,q} \subseteq I$, and hence, $\weakchurchI = \sqgroup{p,q} \neq \emptyset$.
By \corollaryref{cor:relations:interval}, it then holds that $\emptyset \neq \weakchurchI \subseteq \churchI$.
\end{comment}

To prove that every path~$\pth \in \pths$ is \ch-random for the interval forecast~$\churchI$, we consider the real numbers
\begin{align*}
p &\coloneqq \inf_{\selection \in \selectionsdense} \liminf_{n \to \infty} \frac{\sum_{k=0}^{n-1}\selection(\pthtok)\pthatkplus}{\sum_{k=0}^{n-1}\selection(\pthtok)}
\shortintertext{and}
q &\coloneqq \sup_{\selection \in \selectionsdense} \limsup_{n \to \infty} \frac{\sum_{k=0}^{n-1}\selection(\pthtok)\pthatkplus}{\sum_{k=0}^{n-1}\selection(\pthtok)}.
\end{align*}
An analogue argument shows that $\churchI = \sqgroup{p,q}$ and that $\pth$ is \ch-random for an interval forecast~$I \in \intervals$ if and only if $\sqgroup{p,q}\subseteq I$.
\qed \

{\noindent\bfseries Proof of \propositionref{prop:ml:almost:completely}\quad}
For ease of notation, let $p\coloneqq \liminf_{n \to \infty}\frcstsystem(\pthton)$ and~$q\coloneqq \limsup_{n \to \infty}\frcstsystem(\pthton)$.
By \theoremref{the:coinciding}, we know that $\mlI=\sqgroup{p,q}$.
By \propositionref{prop:aid}, we know that $\pth$ is \ml-random for any interval forecast~$I \in \intervals$ of the form 
\begin{equation*}
\sqgroup{p- \epsilon_1,q + \epsilon_2} \cap \sqgroup{0,1} \textrm{, with }\epsilon_1,\epsilon_2>0.
\end{equation*} 
Hence, $\pth$ is almost \ml-random for the interval forecast~$\mlI=\sqgroup{p,q}$.
%Hence, it remains to proof that $\pth$ is not Martin-Löf random for any interval forecast~$I \in \intervals$ for which $p < \min I$ or $\max I < q$.
%To this end, consider such an interval forecast~$I \in \intervals$ and assume \emph{ex absurdo} that $I \in \mlintervals$.
%By \propositionref{prop:relations:interval}, we then know that $I \in \churchintervals$, and therefore $p<\min \churchI$ or $\max \churchI < q$, a contradiction.
\qed

\section*{Proofs and additional material for Section~\ref{sec:question2}}

\begin{lemma} \label{lem:bounded:above}
Consider any non-negative gamble~$\gamble \in \gambles$ and any interval forecast~$I \subseteq \sqgroup{0,1}$ such that $0 < \max I$ and~$\min I <1$.
If $\uex_I(\gamble) \leq 1$, then $\gamble(1)\leq \nicefrac{1}{\max I}$ and~$\gamble(0)\leq \nicefrac{1}{(1-\min I)}$.
\end{lemma}

\begin{proof}
Since $0 < \max I$ and~$ \min I < 1$, it holds that both $\nicefrac{1}{\max I}$ and~$\nicefrac{1}{1-\min I}$ are real numbers.
By \equationref{eq:uex:2}, $\gamble(1) > \nicefrac{1}{\max I}$ implies that
\begin{equation*}
\uex_I(\gamble) \overset{\eqref{eq:uex:2}}{\geq} \max I \gamble(1)+(1-\max I) \gamble(0) \geq \max I \gamble(1) > 1.
\end{equation*}
Similarly, $\gamble(0) > \nicefrac{1}{(1-\min I)}$ implies that
\begin{equation*}
\uex_I(\gamble) \overset{\eqref{eq:uex:2}}{\geq} \min I \gamble(1)+(1-\min I) \gamble(0) \geq (1-\min I) \gamble(0) > 1.
\end{equation*}
\qed
\end{proof}

Consider any real process $\process$.
In the proof of the following lemma, we will use the following notation to relate $\process$ to its process difference $\adddelta \process$: $\adddelta \process(\sit) = \process(\sit \, \bullet) - \process(\sit)$ for all $\sit \in \sits$.

\begin{lemma} \label{lem:toandback:M:D}
Consider a forecasting system~$\frcstsystem \in \frcstsystems$ and a positive real process~$\supermartin$.
Then $\multprocess_\supermartin$ is positive, and~$\supermartin$ is a supermartingale for~$\frcstsystem$ if and only if $\multprocess_\supermartin$ is a real supermartingale multiplier for $\frcstsystem$.
%$\multprocess_\supermartin$ is a supermartingale multiplier for~$\frcstsystem$.
\end{lemma}

\begin{proof}
\begin{comment}
Assume that $\supermartin$ is positive. 
Clearly, $\multprocess_\supermartin(\sit)(x)=\frac{\supermartin(\sit x)}{\supermartin(\sit)}$ is positive as well for all~$\sit \in \sits$ and~$x \in \posspace$.
Vice versa, assume that $\multprocess_\supermartin$ is positive.
Clearly, $\supermartin(\sit)=\prod_{k=0}^{\abs{\sit}-1}\multprocess_\supermartin(\sittok)(\sitatkplus)$ is positive as well for all~$\sit \in \sits$.

We are left to prove the supermartingale property.
To this end, we 
\end{comment}
%We are left with proving the positivity of $\multprocess_\supermartin$.
Since $\supermartin$ is positive by assumption, $\multprocess_\supermartin(\sit)(x)=\frac{\supermartin(\sit x)}{\supermartin(\sit)}$ is clearly positive as well for all~$\sit \in \sits$ and~$x \in \posspace$.
Fix any~$\sit \in \sits$.
Note that
\begin{align*}
\uex_{\frcstsystem(\sit)}(\adddelta \supermartin(\sit))
&=
\uex_{\frcstsystem(\sit)}(\supermartin(\sit \, \bullet)-\supermartin(\sit))\\
&=
\uex_{\frcstsystem(\sit)}(\supermartin(\sit)\multprocess_\supermartin(\sit)-\supermartin(\sit))
\overset{\textrm{\ref{axiom:coherence:homogeneity}}}{=}
\supermartin(\sit) \uex_{\frcstsystem(\sit)}(\multprocess_\supermartin(\sit)-1),
\end{align*}
and therefore,
\begin{align*}
\uex_{\frcstsystem(\sit)}(\adddelta \supermartin(\sit))\leq 0
&\iff
\supermartin(\sit) \uex_{\frcstsystem(\sit)}(\multprocess_\supermartin(\sit)-1)\leq 0 \\
\overset{\supermartin(\sit)>0}&{\iff}
\uex_{\frcstsystem(\sit)}(\multprocess_\supermartin(\sit)-1)\leq 0 \\
\overset{\textrm{\ref{axiom:coherence:constantadditivity}}}&{\iff}
\uex_{\frcstsystem(\sit)}(\multprocess_\supermartin(\sit))\leq 1.
\end{align*} 
\qed
\end{proof}

In the proof of the following proposition, we will use the following terminology and notation.
For any situations $\sit,t \in \sits$, we write $\sit \precedes t$ when every path that goes through $t$ also goes through $\sit$, and we say that the situation~$\sit$ \emph{precedes} the situation $t$; so $\sit$ is a precursor of $t$.
We say that $\sit$ \emph{strictly precedes} $t$, and write $\sit \sprecedes t$, when $\sit \precedes t$ and~$\sit \neq t$.

\begin{proposition} \label{prop:aid}
For any~$\random \in \set{\ml,\wml,\co,\s,\ch,\wch}$, any path~$\pth \in \pths$ that is \random-random for a computable forecasting system~$\frcstsystem \in \frcstsystems$ and any~$\epsilon_1,\epsilon_2>0$:
\begin{equation*}
\Big[\liminf_{n \to \infty}\underline{\frcstsystem}(\pthton)-\epsilon_1,\limsup_{n \to \infty}\overline{\frcstsystem}(\pthton)+\epsilon_2\Big] \cap \sqgroup{0,1} \in \intervals_\random(\pth)
\end{equation*}
\end{proposition}

\begin{proof}
%For ease of notation, 
We start with the case $\random=\ml$.
Let $p\coloneqq \liminf_{n \to \infty}\underline{\frcstsystem}(\pthton)$ and~$q\coloneqq \limsup_{n \to \infty}\overline{\frcstsystem}(\pthton)$.
Consider a path~$\pth \in \pths$ that is \ml-random for the computable forecasting system~$\frcstsystem$; the proof is very similar for~$\random \in \set{\wml,\co,\s}$.
We will show that for any \(\epsilon_1,\epsilon_2 >0\), the path~$\pth$ is \ml-random for the interval forecast~$I\coloneqq \sqgroup{p-\epsilon_1,q+\epsilon_2} \cap \sqgroup{0,1}$.
%$\sqgroup{\max\{0,p-\epsilon\},\min\{1,q+\epsilon\}}$.
%$\sqgroup{p-\epsilon,q+\epsilon} \cap \sqgroup{0,1}$.
\begin{comment}
that \(\min \churchI \leq p+\epsilon\) for all~$\epsilon>0$.
\end{comment}
To this end, fix any rational numbers \(\underline{r}\) and \(\overline{r}\) such that $p-\frac{3}{4}\epsilon_1 < \underline{r} < p-\frac{1}{2}\epsilon_1$ and  $q+\frac{1}{2}\epsilon_2 < \overline{r} < q+\frac{3}{4}\epsilon_2$, and a natural number \(N\) such that \(2^{-N}<\frac{1}{4} \min\{\epsilon_1,\epsilon_2\}\).
Since \(\frcstsystem\) is a computable forecasting system, there are two recursive rational maps \(\underline{q},\overline{q}\colon \sits \times \naturalswithzero \to \rationals\) such that \(\abs{\underline{\frcstsystem}(\sit)-\underline{q}(\sit,n)}< 2^{-n}\) and \(\abs{\overline{\frcstsystem}(\sit)-\overline{q}(\sit,n)} < 2^{-n}\) for all~$\sit \in \sits$ and~$n \in \naturalswithzero$.
Consequently, for all~$\sit \in \sits$,
\begin{align}
\abs{\underline{\frcstsystem}(\sit)-\underline{q}(\sit,N)} <  2^{-N} < \frac{1}{4}\epsilon_1 \label{eq:aid:11}
\shortintertext{and}
\abs{\overline{\frcstsystem}(\sit)-\overline{q}(\sit,N)} <  2^{-N} < \frac{1}{4}\epsilon_2. \label{eq:aid:22}
\end{align}
To show that $\pth$ is \ml-random for~$I$, and in line with \propositionref{prop:mlrandomness}, we fix any lower semicomputable positive test supermartingale~$\test \in \testsml{I}$ and show that it remains bounded on $\pth$.
Since $\test$ is lower semicomputable, there is some recursive rational map~$q\colon \sits \times \naturalswithzero \to \rationals$ such that
\begin{enumerate}[label=\upshape(\roman*),leftmargin=*,noitemsep,topsep=0pt]
\item $q(\sit,n+1) \geq q(\sit,n)$ for all~$\sit \in \sits$ and~$n \in \naturalswithzero$;
\item $\test(\sit)=\lim_{n \to \infty}q
(\sit,n)$ for all~$\sit \in \sits$.
\end{enumerate}
We will suitably adapt this recursive rational map to end up with a lower semicomputable positive test supermartingale for~$\frcstsystem$.
To this end, consider the selection process~$\selection'$ defined by
\begin{equation*}
\selection'(\sit) \coloneqq \begin{cases}
1 &\textrm{if } \underline{q}(\sit,N) < \underline{r} \textrm{ or } \overline{r} < \overline{q}(\sit,N)\\
0 &\textrm{otherwise}
\end{cases}
\textrm{ for all } \sit \in \sits.
\end{equation*}
This selection process expresses for every situation~$\sit \in \sits$ whether $\underline{r} \leq \underline{q}(\sit,N)$ and $\overline{q}(\sit,N) \leq \overline{r}$, or not.
Since \(\underline{r}\) and~$\overline{r}$ are rational numbers, \(N\) is a natural number, and \(\underline{q}\) and \(\overline{q}\) are recursive rational maps, it follows that the inequalities in the above expression are decidable for every~$\sit \in \sits$, and hence, the selection process \(\selection'\) is recursive.
We use the selection process~$\selection'$ to introduce the process~$\selection^\ast$ defined by
\begin{equation*}
\selection^\ast(\sit)\coloneqq\sum_{k=0}^{n-1} \selection'(\xvaltok) \textrm{ for all } \sit=(\xvaltolong) \in \sits,
\end{equation*}
which expresses for every situation~$\sit \in \sits$ how many times $\underline{q}(t,N)<\underline{r}$ or $\overline{r}<\overline{q}(t,N)$ for all strictly preceding situations $t \sprecedes s$.
Since $\selection'$ is recursive and natural, and since for every~$\sit \in \sits$ the finite number of situations $t \in \sits$ for which $t \sprecedes \sit$ can be recursively enumerated, the process~$\selection^\ast$ is recursive and natural.

Since $0 \leq \min I < \max I \leq 1$, it follows that $ 0 < \max I$ and $\min I <1$, and hence, by \lemmaref{lem:bounded:above}, we can fix a rational $K>1$ such that for any positive gamble~$\gamble \in \gambles$ for which $\uex_I(\gamble)\leq 1$, it holds that $\gamble \leq K$.
Consequently, $\nicefrac{\gamble}{K}\leq 1$, and hence, by \ref{axiom:coherence:bounds}, $\uex_{\frcstsystem(\sit)}(\nicefrac{\gamble}{K}) \leq 1$ for all~$\sit \in \sits$.

We introduce a new process~$\process^\ast$ defined by
\begin{equation*}
\process^\ast(\sit) \coloneqq \bigg(\frac{1}{K}\bigg)^{\selection^\ast(\sit)}
\textrm{ for all } \sit \in \sits.
\end{equation*}
Since $K$ is rational and since the process~$\selection^{\ast}$ is recursive and natural, it follows that the process~$\process^\ast$ is recursive and rational.
Furthermore, since $K$ is positive and since the process~$\selection^\ast$ is natural, the process~$\process^\ast$ is positive as well.
Moreover, note that $\process^\ast(\init)=(\nicefrac{1}{K})^{\selection^\ast(\init)}=(\nicefrac{1}{K})^0=1$.

Consider the map~$q^\ast\colon \sits \times \naturalswithzero \to \rationals$ defined by
\begin{equation}
q^\ast(\sit,n)=q(\sit,n)\process^\ast(\sit) \textrm{ for all } \sit \in \sits \textrm{ and } n \in \naturalswithzero.
\end{equation}
Since $q(\sit,n)$ is a recursive rational map and since the process~$\process^\ast$ is recursive and rational, the map~$q^\ast(\sit,n)$ is recursive and rational as well.
Due to the positivity of $\process^\ast$, $q^\ast(\sit,n+1) \geq q^\ast(\sit,n)$ for all~$\sit \in \sits$ and~$n \in \naturalswithzero$, because $q(\sit,n+1) \geq q(\sit,n)$.
Consequently, the real process~$\test^\ast$ defined by
\begin{equation*}
\test^\ast(\sit)
\coloneqq
\lim_{n\to\infty} q^\ast(\sit,n) \textrm{ for all } \sit \in \sits,
\end{equation*}
is lower semicomputable.
Note that, for all~$\sit \in \sits$,
\begin{equation}
\test^\ast(\sit)
=
\lim_{n\to\infty} q^\ast(\sit,n)
=
\lim_{n\to\infty} q(\sit,n) \process^\ast(\sit)
=
\test(\sit) \process^\ast (\sit). \label{eq:mult}
\end{equation}
Since $\test$ and $\process^\ast$ are both real processes, $\test^\ast$ is indeed a real process.
Let us now show that $\test^\ast$ is a lower semicomputable positive test supermartingale for~$\frcstsystem$.
First, since $\test$ and~$\process^\ast$ are positive, it follows from \equationref{eq:mult} that $\test^\ast$ is positive as well.
Second, since $\test(\init)=1$ and~$\process^\ast (\init)=1$, it follows from \equationref{eq:mult} that $\test^\ast(\init)=1$.
Last, to show that it is a supermartingale for~$\frcstsystem$, we fix any~$\sit \in \sits$ and, in accordance with \lemmaref{lem:toandback:M:D}, prove that $\uex_{\frcstsystem(\sit)}(\multprocess_{\test^\ast}(\sit)) \leq 1$.
To this end, note that
\begin{align}
\multprocess_{\test^\ast}(\sit)(x)
=
\frac{\test^\ast(\sit x)}{\test^\ast(\sit)}
\overset{\eqref{eq:mult}}{=}
\frac{\test(\sit x)\process^\ast(\sit x)}{\test(\sit) \process^\ast(\sit)} 
&=
\frac{\test(\sit x)}{\test(\sit)} \bigg(\frac{1}{K}\bigg)^{\selection^\ast(\sit x)-\selection^\ast(\sit)} \nonumber \\
&=
\frac{\test(\sit x)}{\test(\sit)} \bigg(\frac{1}{K}\bigg)^{\selection'(\sit)} \textrm{ for all } x \in \posspace. \label{eq:decrease:Dm}
\end{align}
If $\underline{q}(\sit,N)<\underline{r}$ or $\overline{r}<\overline{q}(\sit,N)$, then $\selection'(\sit)=1$, and hence,
\begin{equation*}
\multprocess_{\test^\ast}(\sit)(x)
\overset{\eqref{eq:decrease:Dm}}{=}
\frac{\test(\sit x)}{\test(\sit)} \bigg(\frac{1}{K}\bigg)^{\selection'(\sit)}
=
\frac{\test(\sit x)}{\test(\sit)} \frac{1}{K}
=
\frac{\multprocess_\test(\sit)(x)}{K} \textrm{ for all } x \in \posspace.
\end{equation*}
Since $\test$ is a positive supermartingale for~$I$, it follows from \lemmaref{lem:toandback:M:D} that $\multprocess_\test$ is a positive supermartingale multiplier for~$I$, i.e., $\multprocess_\test(t)>0$ and~$\uex_I(\multprocess_\test(t))\leq 1$ for all~$t \in \sits$.
By recalling that $\uex_{\frcstsystem(\sit)}(\nicefrac{\gamble}{K}) \leq 1$ if $\gamble>0$ and~$\uex_I(\gamble)\leq 1$, it immediately follows that
\begin{equation*}
\uex_{\frcstsystem(\sit)}(\multprocess_{\test^\ast}(\sit))=\uex_{\frcstsystem(\sit)}(\nicefrac{\multprocess_{\test}(\sit)}{K}) \leq 1.
\end{equation*}
\begin{comment}
\begin{align*}
\adddelta \test^\ast(\sit)
&=
\test^\ast(\sit)\bigg(\frac{\test^\ast(\sit \, \bullet)}{\test^\ast(\sit)}-1\bigg)
=
\test^\ast(\sit)\bigg(\frac{\test(\sit \, \bullet)\process^\ast(\sit \, \bullet)}{\test(\sit)\process^\ast(\sit)}-1\bigg)\\
&=
\test^\ast(\sit)\bigg(\frac{\test(\sit \, \bullet)\process^\ast(\sit)}{\test(\sit)\process^\ast(\sit)}\frac{1}{K}-1\bigg)
=
\test^\ast(\sit)\bigg(\frac{\test(\sit \, \bullet)}{\test(\sit)}\frac{1}{K}-1\bigg)
\leq
\test(\sit)\bigg(\frac{\test(\sit \, \bullet)}{\test(\sit)}\frac{1}{K}-1\bigg)
\end{align*}
\end{comment}
Otherwise, if $\underline{r} \leq \underline{q}(\sit,N)$ and~$\overline{q}(\sit,N) \leq \overline{r}$, then $\selection'(\sit)=0$, and hence,
\begin{equation}
\multprocess_{\test^\ast}(\sit)(x)
\overset{\eqref{eq:decrease:Dm}}{=}
\frac{\test(\sit x)}{\test(\sit)} \bigg(\frac{1}{K}\bigg)^{\selection'(\sit)}
=
\frac{\test(\sit x)}{\test(\sit)}
=
\multprocess_\test(\sit)(x) \textrm{ for all } x \in \posspace. \label{eq:equal:Dm}
\end{equation}
Moreover, since $\underline{r} \leq \underline{q}(\sit,N)$ and~$\overline{q}(\sit,N) \leq \overline{r}$, it then holds that
\begin{align*}
p-\frac{3}{4}\epsilon_1< \underline{r} \leq  \underline{q}(\sit,N) \overset{\eqref{eq:aid:11}}{<} \underline{\frcstsystem}(\sit)+\frac{1}{4}\epsilon_1
\shortintertext{and}
q+\frac{3}{4}\epsilon_2 > \overline{r} \geq  \overline{q}(\sit,N) \overset{\eqref{eq:aid:22}}{>} \overline{\frcstsystem}(\sit)-\frac{1}{4}\epsilon_2,
\end{align*}
and therefore, $p-\epsilon_1 < \underline{\frcstsystem}(\sit)$ and~$\overline{\frcstsystem}(\sit)<q+\epsilon_2$.
Consequently, $\frcstsystem(\sit) \subseteq I$, and therefore it follows from Equation~\eqref{eq:uex:1} that 
\begin{equation*}
\uex_{\frcstsystem(\sit)}(\multprocess_{\test^\ast}(\sit))
\overset{\eqref{eq:equal:Dm}}{=}
\uex_{\frcstsystem(\sit)}(\multprocess_\test(\sit))
\overset{\eqref{eq:uex:1}}{\leq}
\uex_{I}(\multprocess_\test(\sit))\leq 1,
\end{equation*}
where the last inequality holds because $\multprocess_\test$ is a supermartingale multiplier for~$I$.

We conclude that $\test^\ast$ is a lower semicomputable positive test supermartingale for~$\frcstsystem$.
Since $\pth$ is Martin-Löf random for~$\frcstsystem$ by assumption, this implies that $\limsup_{n \to \infty}\test^\ast(\pthton)<\infty$.

Now, for any $n \in \naturalswithzero$, if $\underline{q}(\pthton,N)<\underline{r}$, then 
\begin{equation*}
\underline{\frcstsystem}(\pthton)\overset{\eqref{eq:aid:11}}{<}\underline{q}(\pthton,N)+\frac{1}{4}\epsilon_1<\underline{r}+\frac{1}{4}\epsilon_1<p-\frac{1}{2}\epsilon_1+\frac{1}{4}\epsilon_1 = p-\frac{1}{4}\epsilon_1.
\end{equation*}
Similarly, if $\overline{r}<\overline{q}(\pthton,N)$, then 
\begin{equation*}
\overline{\frcstsystem}(\pthton)\overset{\eqref{eq:aid:22}}{>}\overline{q}(\pthton,N)-\frac{1}{4}\epsilon_2>\overline{r}-\frac{1}{4}\epsilon_2>q+\frac{1}{2}\epsilon_2-\frac{1}{4}\epsilon_2 = q+\frac{1}{4}\epsilon_2.
\end{equation*}
By recalling that $p\coloneqq \liminf_{n \to \infty}\underline{\frcstsystem}(\pthton)$ and~$q\coloneqq \limsup_{n \to \infty}\overline{\frcstsystem}(\pthton)$, it is clear that there are only a finite number of natural numbers $n\in \naturalswithzero$ for which $\underline{\frcstsystem}(\pthton)<p-\frac{1}{4}\epsilon_1$ or $q+\frac{1}{4}\epsilon_2<\overline{\frcstsystem}(\pthton)$, and hence, there is only a finite number of natural numbers $n \in \naturalswithzero$ for which $\underline{q}(\pthton,N)<\underline{r}$ or $\overline{r}<\overline{q}(\pthton,N)$, or, equivalently, for which $\selection'(\pthton)=1$.
Consequently, there is some~$B \in \naturalswithzero$ such that $\lim_{n \to \infty}\selection^\ast(\pthton)=\lim_{n \to \infty} \smash{\sum_{k=0}^{n-1}\selection'(\pthtok)}=B<\infty$, and hence,
\begin{equation*}
\limsup_{n \to \infty}\test^\ast(\pthton)
=
\limsup_{n \to \infty}\test(\pthton)\process^\ast(\pthton)
=
\frac{1}{K^B} \limsup_{n \to \infty}\test(\pthton).
\end{equation*}
Since $\limsup_{n \to \infty}\test^\ast(\pthton)<\infty$, this implies that $\limsup_{n \to \infty}\test(\pthton)<\infty$.

We will proceed by proving that if a path~$\pth \in \pths$ is \random-random for the computable forecasting system $\frcstsystem$, for any $\random \in \set{\wml,\co,\s,\ch,\wch}$, then it is also \random-random for the interval forecast $I$.
We thus consider, without loss of generality, the same forecasting system $\frcstsystem$ and the same interval forecast $I$ as before.
Moreover, we also consider the same numbers $\epsilon_1,\epsilon_2, N$ and $K$.
This allows us to reuse the recursive selection process $\selection'$, the recursive natural process $\selection^\ast$ and the recursive rational process $\process^\ast$, since they only depend on the mathematical objects that we mentioned in this paragraph.
%Before we tackle the proof of \propositionref{prop:aid} for $\random \in \set{\wml,\co,\s,\ch,\wch}$, we remark that  dict

We continue by sketching the proof for~$\random = \wml$.
Consider a path~$\pth \in \pths$ that is \wml-random for the computable forecasting system~$\frcstsystem$.
To show that $\pth$ is \wml-random for~$I$, and in line with \definitionref{def:notionsofrandomness}, we fix any test supermartingale $\mint \in \testswml{I}$ that is generated by a lower semicomputable multiplier process~$\multprocess$ and show that it remains bounded on~$\pth$.
Since $\multprocess$ is lower semicomputable, there is some recursive rational map~$q\colon \sits \times \posspace \times \naturalswithzero \to \rationals$ such that
\begin{enumerate}[label=\upshape(\roman*),leftmargin=*,noitemsep,topsep=0pt]
\item $q(\sit,x,n+1) \geq q(\sit,x,n)$ for all~$\sit \in \sits$, $x \in \posspace$ and~$n \in \naturalswithzero$;
\item $\multprocess(\sit)(x)=\lim_{n \to \infty}q(\sit,x,n)$ for all~$\sit \in \sits$ and~$x \in \posspace$.
\end{enumerate}
We will suitably adapt this recursive rational map to end up with a test supermartingale for~$\frcstsystem$ that is generated by a lower semicomputable multiplier process.
To this end, we consider the recursive rational map~$q^\ast\colon \sits\times \posspace \times \naturalswithzero \to \rationals$, defined by
\begin{equation*}
q^\ast(\sit,x,n) \coloneqq q(\sit,x,n)\bigg(\frac{1}{K}\bigg)^{\selection'(\sit)} \textrm{ for all } \sit \in \sits\textrm{, } x \in \posspace \textrm{ and } n \in \naturalswithzero,
\end{equation*}
and the related multiplier process~$\multprocess^\ast$ defined by $\multprocess^\ast(\sit)(x) \coloneqq \lim_{n \to \infty} q^\ast(\sit,x,n)$ for all~$\sit \in \sits$ and~$x \in \posspace$.
By using a similar line of reasoning as above, it is easy to verify that $\multprocess^\ast$ is a lower semicomputable multiplier process and that ${\multprocess^\ast}^\circledcirc$ is a test supermartingale for~$\frcstsystem$.
%that is generated by the lower semicomputable multiplier process~$\multprocess^\ast$.
Since $\pth$ is \wml-random for~$\frcstsystem$ by assumption, it holds that $\limsup_{n \to \infty}{\multprocess^\ast}^\circledcirc(\pthton)<\infty$.
By using a similar argument as before, there is some $B \in \naturalswithzero$ such that $\limsup_{n \to \infty}\multprocess^\circledcirc(\pthton)=K^B\limsup_{n \to \infty}{\multprocess^\ast}^\circledcirc(\pthton)$, and hence, $\limsup_{n \to \infty}\multprocess^\circledcirc(\pthton)<\infty$.

If $\random=\co$, then we consider a path~$\pth \in \pths$ that is \co-random for the computable forecasting system~$\frcstsystem$.
To show that $\pth$ is \co-random for~$I$, and in line with \definitionref{def:notionsofrandomness}, we fix any computable test supermartingale~$\test \in \testscomp{I}$ and show that it remains bounded on~$\pth$.
%\cite[Proposition 6]{sum2020persiau}.
Since $\test$ is computable, there is some recursive rational map~$q\colon \sits \times \naturalswithzero \to \rationals$ such that $\abs{\test(\sit)-q(\sit,n)}<2^{-n}$ for all~$\sit \in \sits$ and~$n \in \naturalswithzero$.
We will suitably adapt this recursive rational map to end up with a computable test supermartingale for~$\frcstsystem$.
To this end, we consider the recursive rational map~$q^\ast \colon \sits \times \naturalswithzero \to \rationals$, defined by $q^\ast(\sit,n)\coloneqq q(\sit,n) \process^\ast(\sit)$ for all~$\sit \in \sits$ and~$n \in \naturalswithzero$, and the real process~$\test^\ast\colon \sits \to \reals$, defined by $\test^\ast(\sit)\coloneqq \lim_{n \to \infty} q^\ast(\sit,n)= \lim_{n \to \infty} q(\sit,n) \process^\ast(\sit) = \test(\sit) \process^\ast (\sit)$ for all~$\sit \in \sits$.
By recalling that $K>1$ and that $\selection^\ast$ is a recursive natural process, it follows that 
\begin{equation*}
\abs{\test^\ast(\sit)-q^\ast(\sit,n)}=\bigg(\frac{1}{K}\bigg)^{\selection^\ast(\sit)}\abs{\test(\sit)-q(\sit,n)} \leq \abs{\test(\sit)-q(\sit,n)} < 2^{-n}
\end{equation*} 
for all $\sit \in \sits$ and $n \in \naturalswithzero$, and hence, $\test^\ast$ is a computable real process.
By using a similar line of reasoning as above, it is easy to verify that $\test^\ast$ is in fact a computable test supermartingale for~$\frcstsystem$.
Since $\pth$ is \co-random for~$\frcstsystem$ by assumption, it holds that $\limsup_{n \to \infty}\test^\ast(\pthton)<\infty$.
By using a similar argument as before, there is some $B \in \naturalswithzero$ such that $\limsup_{n \to \infty}\test(\pthton)=K^B\limsup_{n \to \infty}\test^\ast(\pthton)$, and hence, $\limsup_{n \to \infty}\test(\pthton)<\infty$.

If $\random=\s$, then we consider a path~$\pth \in \pths$ that is \s-random for the computable forecasting system~$\frcstsystem$.
To show that $\pth$ is \s-random for~$I$, and in line with \definitionref{def:schnorr}, we intend to prove that every computable test supermartingale for $I$ remains computably bounded on~$\pth$.
Without loss of generality, we consider the computable test supermartingales~$\test \in \testsschnorr{I}$ and $\test^\ast \in \testsschnorr{\frcstsystem}$ that were introduced in the previous paragraph.
Assume \emph{ex absurdo} that $\test$ is computably unbounded on $\pth$, meaning that there is some real growth function $\realgrowth$ such that $\limsup_{n \to \infty}(\test(\pthton)-\realgrowth(n)) \geq 0$.
%Consider the computable test supermartingale $\test^\ast$ that was introduced in the previous paragraph.
%Since $\test$ is computable, there is some recursive rational map~$q\colon \sits \times \naturalswithzero \to \rationals$ such that $\abs{\test(\sit)-q(\sit,n)}<2^{-n}$ for all~$\sit \in \sits$ and~$n \in \naturalswithzero$.
%We will suitably adapt this recursive rational map to end up with a computable test supermartingale for~$\frcstsystem$.
%To this end, we consider the recursive rational map~$q^\ast \colon \sits \times \naturalswithzero \to \rationals$ defined by $q^\ast(\sit,n)\coloneqq q(\sit,n) \process^\ast(\sit)$ for all~$\sit \in \sits$ and~$n \in \naturalswithzero$.
%It is easy to verify that the real process~$\test^\ast\colon \sits \to \reals$, defined by $\test^\ast(\sit)\coloneqq \lim_{n \to \infty} q(\sit,n) \process^\ast(\sit) = \test(\sit) \process^\ast (\sit)$ for all~$\sit \in \sits$, is a computable test supermartingale for~$\frcstsystem$.
By using a similar argument as before, there is some $B \in \naturalswithzero$ such that $\limsup_{n \to \infty}\test(\pthton)=K^B\limsup_{n \to \infty}\test^\ast(\pthton)$.
Moreover, we consider the real growth function $\tau^\ast$ defined by $\tau^\ast(n) \coloneqq \nicefrac{\tau(n)}{K^B}$ for all~$n \in \naturalswithzero$.
Since $\pth$ is Schnorr random for~$\frcstsystem$ by assumption, it holds that $\limsup_{n \to \infty}(\test^\ast(\pthton)-\tau^\ast(n)) < 0$, and hence, $\limsup_{n \to \infty}(\test(\pthton)-\tau(n))=K^B\limsup_{n \to \infty}(\test^\ast(\pthton)-\tau^\ast(n)) < 0$, a contradiction.

If $R\in\set{\ch,\wch}$, then we consider a path~$\pth \in \pths$ that is \ch-random (\wch-random) for the computable forecasting system~$\frcstsystem$.
In this case, to show that $\pth$ is \ch-random (\wch-random) for~$I$, the line of reasoning differs. 
We consider any recursive dense (temporal) selection process~$\selection$, and, according to \definitionref{def:churchrandom}, intend to show that 
\begin{equation*}
\min I
\leq
\liminf_{n \to \infty} \frac{\sum_{k=0}^{n-1}\selection(\pthtok)\pthatkplus}{\sum_{k=0}^{n-1}\selection(\pthtok)} \leq
\limsup_{n \to \infty} \frac{\sum_{k=0}^{n-1}\selection(\pthtok)\pthatkplus}{\sum_{k=0}^{n-1}\selection(\pthtok)}
\leq
\max I.
\end{equation*}
By using a similar argument as before, we know there is some~$B \in \naturalswithzero$ such that $\lim_{n \to \infty}\selection^\ast(\pthton)=\lim_{n \to \infty} \smash{\sum_{k=0}^{n-1}\selection'(\pthtok)=B<\infty}$.
Consequently, there is some~$M \in \naturalswithzero$ such that $\selection'(\pthton)=0$ for all~$n \geq M$, and hence, $\frcstsystem(\pthton) \subseteq I$ for all~$n\geq M$.
Since $\pth$ is \ch-random (\wch-random) for~$\frcstsystem$ by assumption, it holds by \definitionref{def:churchrandom} that
\begin{align*}
\liminf_{n \to \infty} \frac{\sum_{k=0}^{n-1}\selection(\pthtok)[\pthatkplus-\underline{\frcstsystem}(\pthtok)]}{\sum_{k=0}^{n-1}\selection(\pthtok)} \geq 0
\shortintertext{and}
\limsup_{n \to \infty} \frac{\sum_{k=0}^{n-1}\selection(\pthtok)[\pthatkplus-\overline{\frcstsystem}(\pthtok)]}{\sum_{k=0}^{n-1}\selection(\pthtok)} \leq 0.
\end{align*}
Since $\lim_{n \to \infty} \selectionsum = \infty$, it follows that 
\begin{align*}
\liminf_{n \to \infty} \frac{\sum_{k=M}^{n-1}\selection(\pthtok)[\pthatkplus-\underline{\frcstsystem}(\pthtok)]}{\sum_{k=M}^{n-1}\selection(\pthtok)} \geq 0
\shortintertext{and}
\limsup_{n \to \infty} \frac{\sum_{k=M}^{n-1}\selection(\pthtok)[\pthatkplus-\overline{\frcstsystem}(\pthtok)]}{\sum_{k=M}^{n-1}\selection(\pthtok)} \leq 0,
\end{align*}
and hence,
\begin{multline*}
\liminf_{n \to \infty} \frac{\sum_{k=M}^{n-1}\selection(\pthtok)[\pthatkplus-\min I]}{\sum_{k=M}^{n-1}\selection(\pthtok)} \\
\geq
\liminf_{n \to \infty} \frac{\sum_{k=M}^{n-1}\selection(\pthtok)[\pthatkplus-\underline{\frcstsystem}(\pthtok)]}{\sum_{k=M}^{n-1}\selection(\pthtok)}
\geq 0 \\
\textrm{and} \hfill  \\
\shoveleft \limsup_{n \to \infty} \frac{\sum_{k=M}^{n-1}\selection(\pthtok)[\pthatkplus-\max I)]}{\sum_{k=M}^{n-1}\selection(\pthtok)} \\
\leq
\limsup_{n \to \infty} \frac{\sum_{k=M}^{n-1}\selection(\pthtok)[\pthatkplus-\overline{\frcstsystem}(\pthtok)]}{\sum_{k=M}^{n-1}\selection(\pthtok)}
\leq 0.
\end{multline*}
\begin{comment}
\begin{multline*}
\liminf_{n \to \infty} \frac{\sum_{k=M}^{n-1}\selection(\pthtok)[\pthatkplus-(p-\frac{1}{4}\epsilon_1)]}{\sum_{k=M}^{n-1}\selection(\pthtok)} \\
\geq
\liminf_{n \to \infty} \frac{\sum_{k=M}^{n-1}\selection(\pthtok)[\pthatkplus-\underline{\frcstsystem}(\pthtok)]}{\sum_{k=M}^{n-1}\selection(\pthtok)}
\geq 0
\end{multline*}
and
\begin{multline*}
\limsup_{n \to \infty} \frac{\sum_{k=M}^{n-1}\selection(\pthtok)[\pthatkplus-(q+\frac{1}{4}\epsilon_2))]}{\sum_{k=M}^{n-1}\selection(\pthtok)} \\
\leq
\limsup_{n \to \infty} \frac{\sum_{k=M}^{n-1}\selection(\pthtok)[\pthatkplus-\overline{\frcstsystem}(\pthtok)]}{\sum_{k=M}^{n-1}\selection(\pthtok)}
\leq 0.
\end{comment}
Consequently,
\begin{equation*}
\min I
\leq
\liminf_{n \to \infty} \frac{\sum_{k=M}^{n-1}\selection(\pthtok)\pthatkplus}{\sum_{k=M}^{n-1}\selection(\pthtok)} \leq
\limsup_{n \to \infty} \frac{\sum_{k=M}^{n-1}\selection(\pthtok)\pthatkplus}{\sum_{k=M}^{n-1}\selection(\pthtok)}
\leq
\max I,
\end{equation*}
and therefore also, since $\lim_{n \to \infty} \selectionsum = \infty$,
\begin{equation*}
\min I
\leq
\liminf_{n \to \infty} \frac{\sum_{k=0}^{n-1}\selection(\pthtok)\pthatkplus}{\sum_{k=0}^{n-1}\selection(\pthtok)} \leq
\limsup_{n \to \infty} \frac{\sum_{k=0}^{n-1}\selection(\pthtok)\pthatkplus}{\sum_{k=0}^{n-1}\selection(\pthtok)}
\leq
\max I.
\end{equation*}
%If $\random=\wch$, then a completely analogous argument shows that if a path $\pth \in \pths$ is \wch-random for the computable forecasting system~$\frcstsystem$, then it is also \wch-random for~$I$.
\qed
\end{proof}

{\noindent\bfseries Proof of \propositionref{prop:bounds:outer} \quad}
For ease of notation, let $p\coloneqq \liminf_{n \to \infty}\underline{\frcstsystem}(\pthton)$ and~$q\coloneqq \limsup_{n \to \infty}\overline{\frcstsystem}(\pthton)$.
By \propositionref{prop:aid}, it holds for every~$\epsilon_1,\epsilon_2~>0$ that $\sqgroup{p-\epsilon_1,q+\epsilon_2}\cap\sqgroup{0,1} \in \randomintervals$, and therefore
\begin{align*}
\max\{p-\epsilon_1,0\} \leq \min \randomI
\textrm{ and }
\max \randomI \leq \min\{q+\epsilon_2,1\}.
\end{align*}
Since this is true for all~$\epsilon_1,\epsilon_2 >0$, and since $0 \leq p$ and $q \leq 1$, we conclude that
\begin{align*}
p \leq \min \randomI \leq
\max \randomI \leq q.
\end{align*}
\qed

\begin{comment}
{\noindent\bfseries Proof of \corollaryref{cor:bounds:ML}\quad}
For ease of notation, let $p\coloneqq \liminf_{n \to \infty}\overline{\frcstsystem}(\pthton)$ and~$q\coloneqq \limsup_{n \to \infty}\underline{\frcstsystem}(\pthton)$.
By invoking \theoremref{the:coinciding}, we know that $\pth$ is almost Martin-Löf random for the interval forecast~$\sqgroup{p,q}$.
Consequently, for every~$\epsilon_1,\epsilon_2~>0$, $\sqgroup{p-\epsilon_1,q+\epsilon_2}\cap\sqgroup{0,1} \in \mlintervals$, and therefore
\begin{align*}
\max\{p-\epsilon_1,0\} \leq \min \mlI
\textrm{ and }
\max \mlI \leq \min\{q+\epsilon_2,1\}.
\end{align*}
Since this is true for all~$\epsilon_1,\epsilon_2 >0$, we conclude that
\begin{align*}
p \leq \min \mlI \leq
\max \mlI \leq q.
\end{align*}
\qed \
\end{comment}

\begin{proposition} \label{prop:bounds:church}
If a path~$\pth \in \pths$ is \ch-random for a computable forecasting system~$\frcstsystem \in \frcstsystems$, then
\begin{equation*}
\min \churchI
\leq
\liminf_{n \to \infty}\overline{\frcstsystem}(\pthton)
\textrm{ and }
\limsup_{n \to \infty}\underline{\frcstsystem}(\pthton)
\leq
\max \churchI.
\end{equation*}
\end{proposition}

\begin{proof}
For ease of notation, consider the reals $p\coloneqq \liminf_{n \to \infty}\overline{\frcstsystem}(\pthton)$ and~$q\coloneqq \limsup_{n \to \infty}\underline{\frcstsystem}(\pthton)$.
Fix any path~$\pth \in \churchpths$ and any~$\epsilon>0$.
Then we will show that \(\min \churchI \leq p + \epsilon\) and~$q-\epsilon \leq \max \churchI$.
We prove the first inequality; the proof of the second one is very similar.
Fix any rational number \(r\) such that \(p+\frac{1}{2} \epsilon < r < p+\frac{3}{4}\epsilon\), and any natural number \(N\) such that \(2^{-N}<\frac{1}{4} \epsilon\).

Since \(\frcstsystem\) is a computable forecasting system, there is some recursive rational map \(\overline{q}\colon \sits \times \naturalswithzero \to \rationals\) such that \(\abs{\overline{\frcstsystem}(\sit)-\overline{q}(\sit,n)}\leq 2^{-n}\) for all~$\sit \in \sits$ and~$n \in \naturalswithzero$.
Consequently, for all~$\sit \in \sits$,
\begin{equation}
\abs{\overline{\frcstsystem}(\sit)- \overline{q}(\sit,N) } <  2^{-N}< \frac{1}{4}\epsilon \label{eq:aid:1}.
\end{equation} 
Consider the selection process \(\selection\) defined by
\begin{equation*}
\selection(\sit)\coloneqq \begin{cases}
1 &\textrm{if } \overline{q}(\sit,N) < r \\
0 &\textrm{otherwise}
\end{cases}
\textrm{ for all } \sit \in \sits.
\end{equation*}
Since \(r\) is a rational number, \(N\) is a natural number and \(\overline{q}\) is a recursive rational map, the inequality in the above expression is decidable for every~$\sit \in \sits$, so the selection process \(\selection\) is recursive.
Recall that \(p= \liminf_{n \to \infty}\overline{\frcstsystem}(\pthton)\), so there is some infinite subset $A \subseteq \naturalswithzero$ such that $\overline{\frcstsystem}(\pthton)<p+\frac{1}{4}\epsilon$ for all~$n \in A$, and therefore also
\begin{equation*}
\overline{q}(\pthton,N) \overset{\eqref{eq:aid:1}}{<} \overline{\frcstsystem}(\pthton) + \frac{1}{4}\epsilon < p + \frac{1}{4}\epsilon + \frac{1}{4}\epsilon = p + \frac{1}{2}\epsilon < r,
\end{equation*}
and hence, $\selection(\pthton)=1$ for all~$n \in A$.
Since $A$ is an infinite subset of $\naturalswithzero$, it follows that \smash{$\lim_{n \to \infty}\sum_{k=0}^{n-1} \selection(\pthtok)=\infty$}.
Consequently, since $\selection$ is dense along~$\pth$ and recursive, it holds by \definitionref{def:churchrandom} that
%\textcolor{red}{By invoking \propositionref{prop:church:nonstat} for the computable forecasting system~$\frcstsystem$, the computable gamble~$-\ind{1}$, the path~$\pth$ (which is computably random for~$\frcstsystem$) and the recursive selection process~$\selection$ for which \smash{$\lim_{n \to \infty}\selectionsum=\infty$}}, we find that
\begin{align*}
\limsup_{n \to \infty}\frac{\sum_{k=0}^{n-1}\selection(\pthtok)[\pthatkplus-\overline{\frcstsystem}(\pthtok)]}{\sum_{k=0}^{n-1}\selection(\pthtok)} \leq 0.
\end{align*}
\begin{comment}
\begin{align*}
0
&\leq
\liminf_{n\to+\infty}
\dfrac{\sum_{k=0}^{n-1}\selection(\pthtok)\sqgroup[\big]{-\ind{1}(\pthatkplus)-\lex_{\frcstsystem(\pthtok)}(-\ind{1})}}
{\sum_{k=0}^{n-1}\selection(\pthtok)} \\
&=
-\limsup_{n\to+\infty}
\dfrac{\sum_{k=0}^{n-1}\selection(\pthtok)\sqgroup[\big]{\ind{1}(\pthatkplus)-\uex_{\frcstsystem(\pthtok)}(\ind{1})}}
{\sum_{k=0}^{n-1}\selection(\pthtok)} \\
&=
-\limsup_{n \to \infty}\frac{\sum_{k=0}^{n-1}\selection(\pthtok)[\pthatkplus-\overline{\frcstsystem}(\pthtok)]}{\sum_{k=0}^{n-1}\selection(\pthtok)}.
\end{align*}
\end{comment}
Moreover, for any \(\sit \in \sits\), if $\selection(\sit)=1$ or, equivalently, if \(\overline{q}(\sit,N) < r\), then 
\begin{equation*}
\overline{\frcstsystem}(\sit) \overset{\eqref{eq:aid:1}}{<} \overline{q}(\sit,N) + \frac{1}{4}\epsilon < r + \frac{1}{4}\epsilon <p + \frac{3}{4}\epsilon + \frac{1}{4}\epsilon= p+\epsilon.
\end{equation*}
Hence,
\begin{multline*}
\limsup_{n \to \infty}\frac{\sum_{k=0}^{n-1}\selection(\pthtok)[\pthatkplus-(p+\epsilon))]}{\sum_{k=0}^{n-1}\selection(\pthtok)} \\
\leq
\limsup_{n \to \infty}\frac{\sum_{k=0}^{n-1}\selection(\pthtok)[\pthatkplus-\overline{\frcstsystem}(\pthtok)]}{\sum_{k=0}^{n-1}\selection(\pthtok)} 
\leq 
0,
\end{multline*}
so
\begin{equation*}
\limsup_{n \to \infty}\frac{\sum_{k=0}^{n-1}\selection(\pthtok)\pthatkplus}{\sum_{k=0}^{n-1}\selection(\pthtok)} 
\leq
p+\epsilon.
\end{equation*}
Since $\pth$ is \ch-random for~$\churchI$ by \propositionref{prop:churches}, it now follows from \definitionref{def:churchrandom} that, indeed, \(\min \churchI \leq p+\epsilon\).
%In a similar way, it can be proven that $q-\epsilon \leq \max \churchI$ for all~$\epsilon>0$.
\qed 
\end{proof}

\begin{proposition} \label{prop:bounds:weakchurch}
If a path~$\pth \in \pths$ is \wch-random for a computable temporal forecasting system~$\frcstsystem \in \frcstsystems$, then
\begin{equation*}
\min \weakchurchI
\leq
\liminf_{n \to \infty}\overline{\frcstsystem}(\pthton)
\textrm{ and }
\limsup_{n \to \infty}\underline{\frcstsystem}(\pthton)
\leq
\max \weakchurchI.
\end{equation*}
\end{proposition}

\begin{proof}
For ease of notation, consider the reals $p\coloneqq \liminf_{n \to \infty}\overline{\frcstsystem}(\pthton)$ and~$q\coloneqq \limsup_{n \to \infty}\underline{\frcstsystem}(\pthton)$.
Fix any path~$\pth \in \weakchurchpths$ and any~$\epsilon>0$.
Then we will show that \(\min \weakchurchI \leq p + \epsilon\) and~$q-\epsilon \leq \max \weakchurchI$.
We prove the first inequality; the proof of the second one is very similar.
Fix any rational number \(r\) such that \(p+\frac{1}{2} \epsilon < r < p+\frac{3}{4}\epsilon\), and any natural number \(N\) such that \(2^{-N}<\frac{1}{4} \epsilon\).
Since \(\frcstsystem\) is a computable temporal forecasting system, there is some recursive rational map \(\overline{q}\colon \naturalswithzero \times \naturalswithzero \to \rationals\) such that \(\abs{\overline{\frcstsystem}(\sit)-\overline{q}(\abs{\sit},n)}\leq 2^{-n}\) for all~$\sit \in \sits$ and~$n \in \naturalswithzero$.
Consequently, for all~$\sit \in \sits$,
\begin{equation}
\abs{\overline{\frcstsystem}(\sit)- \overline{q}(\abs{\sit},N) } <  2^{-N}< \frac{1}{4}\epsilon \label{eq:aid:2}.
\end{equation} 
Consider the selection process \(\selection\) defined by
\begin{equation*}
\selection(\sit)\coloneqq \begin{cases}
1 &\textrm{if } \overline{q}(\abs{\sit},N) < r \\
0 &\textrm{otherwise}
\end{cases}
\textrm{ for all } \sit \in \sits.
\end{equation*}
Since \(r\) is a rational number, \(N\) is a natural number and \(\overline{q}\) is a recursive rational map, the inequality in the above expression is decidable for every~$\sit \in \sits$, so the selection process \(\selection\) is recursive.
Moreover, $\selection$ is clearly temporal since it only depends on the situations $\sit \in \sits$ through their length $\abs{\sit}$.
Recall that \(p= \liminf_{n \to \infty}\overline{\frcstsystem}(\pthton)\), so there is some infinite subset $A \subseteq \naturalswithzero$ such that $\overline{\frcstsystem}(\pthton)<p+\frac{1}{4}\epsilon$ for all~$n \in A$, and therefore also
\begin{equation*}
\overline{q}(\pthton,N) \overset{\eqref{eq:aid:2}}{<} \overline{\frcstsystem}(\pthton) + \frac{1}{4}\epsilon < p + \frac{1}{4}\epsilon + \frac{1}{4}\epsilon = p + \frac{1}{2}\epsilon < r,
\end{equation*}
and hence, $\selection(\pthton)=1$ for all~$n \in A$.
Since $A$ is an infinite subset of $\naturalswithzero$, it follows that \smash{$\lim_{n \to \infty}\sum_{k=0}^{n-1} \selection(\pthtok)=\infty$}. 
Consequently, since $\selection$ is dense along~$\pth$, recursive and temporal, it holds by \definitionref{def:churchrandom} that 
\begin{align*}
\limsup_{n \to \infty}\frac{\sum_{k=0}^{n-1}\selection(\pthtok)[\pthatkplus-\overline{\frcstsystem}(\pthtok)]}{\sum_{k=0}^{n-1}\selection(\pthtok)} \leq 0.
\end{align*}
Moreover, for any \(\sit \in \sits\), if $\selection(\sit)=1$ or, equivalently, if \(\overline{q}(\abs{\sit},N) < r\), then 
\begin{equation*}
\overline{\frcstsystem}(\sit) \overset{\eqref{eq:aid:2}}{<} \overline{q}(\abs{\sit},N) + \frac{1}{4}\epsilon < r + \frac{1}{4}\epsilon <p + \frac{3}{4}\epsilon + \frac{1}{4}\epsilon= p+\epsilon.
\end{equation*}
Hence,
\begin{multline*}
\limsup_{n \to \infty}\frac{\sum_{k=0}^{n-1}\selection(\pthtok)[\pthatkplus-(p+\epsilon))]}{\sum_{k=0}^{n-1}\selection(\pthtok)} \\
\leq
\limsup_{n \to \infty}\frac{\sum_{k=0}^{n-1}\selection(\pthtok)[\pthatkplus-\overline{\frcstsystem}(\pthtok)]}{\sum_{k=0}^{n-1}\selection(\pthtok)} 
\leq 
0,
\end{multline*}
so
\begin{equation*}
\limsup_{n \to \infty}\frac{\sum_{k=0}^{n-1}\selection(\pthtok)\pthatkplus}{\sum_{k=0}^{n-1}\selection(\pthtok)} 
\leq
p+\epsilon.
\end{equation*}
Since $\pth$ is \wch-random for~$\weakchurchI$ by \propositionref{prop:churches}, it now follows from \definitionref{def:churchrandom} that, indeed, \(\min \weakchurchI \leq p+\epsilon\).
%In a similar way, it can be proven that $q-\epsilon \leq \max \weakchurchI$ for all~$\epsilon>0$.
\qed
\end{proof}

\begin{proposition} \label{prop:bounds:church:inner}
For any path~$\pth \in \pths$ that is \ch-random for a computable precise forecasting system~$\frcstsystem \in \frcstsystems$:
\begin{equation*}
\Big[\liminf_{n \to \infty}\frcstsystem(\pthton),\limsup_{n \to \infty}\frcstsystem(\pthton)\Big] \subseteq \churchI.
\end{equation*}
\end{proposition}

\begin{proof}
%{\noindent\bfseries Proof of \propositionref{prop:bounds:church:inner}\quad}
Since the forecasting system~$\frcstsystem \in \frcstsystems$ is precise by assumption, it holds that $\underline{\frcstsystem}=\overline{\frcstsystem}$, and therefore, by \propositionref{prop:bounds:church},
\begin{equation*}
\min \churchI
\leq
\liminf_{n \to \infty}\frcstsystem(\pthton)
\leq
\limsup_{n \to \infty}\frcstsystem(\pthton)
\leq
\max \churchI.
\end{equation*}
\qed
\end{proof}

\begin{proposition} \label{prop:bounds:weakchurch:inner}
For any path~$\pth \in \pths$ that is \wch-random for a computable precise temporal forecasting system~$\frcstsystem \in \frcstsystems$:
\begin{equation*}
\Big[\liminf_{n \to \infty}\frcstsystem(\pthton),\limsup_{n \to \infty}\frcstsystem(\pthton)\Big] \subseteq \weakchurchI.
\end{equation*}
\end{proposition}

\begin{proof}
%{\noindent\bfseries Proof of \propositionref{prop:bounds:weakchurch:inner}\quad}
Since the forecasting system~$\frcstsystem \in \frcstsystems$ is precise by assumption, it holds that $\underline{\frcstsystem}=\overline{\frcstsystem}$, and therefore, by \propositionref{prop:bounds:weakchurch},
\begin{equation*}
\min \weakchurchI
\leq
\liminf_{n \to \infty}\frcstsystem(\pthton)
\leq
\limsup_{n \to \infty}\frcstsystem(\pthton)
\leq
\max \weakchurchI.
\end{equation*}
\qed
\end{proof}

{\noindent\bfseries Proof of \theoremref{the:coinciding}\quad}
For ease of notation, let $p\coloneqq \liminf_{n \to \infty}\frcstsystem(\pthton)$ and~$q\coloneqq \limsup_{n \to \infty}\frcstsystem(\pthton)$.
Since $\pth$ is \random-random for the computable precise forecasting system~$\frcstsystem$, with $\random \in \set{\ml,\wml,\co,\ch}$, $\pth$ is also \ch-random for~$\frcstsystem$ by \equationref{eq:inclusions:comp}.
Consequently, it follows from Propositions~\ref{prop:bounds:outer} and~\ref{prop:bounds:church:inner} that 
\begin{equation*}
\randomI \subseteq \sqgroup{p,q} \textrm{ and } \sqgroup{p,q} \subseteq \churchI.
%\min \churchI \leq p \textrm{ and } q \leq \max \churchI.
\end{equation*}
%By \propositionref{prop:relations:interval}, it holds that if a path~$\pth$ is Martin-Löf random for an interval forecast~$I$, then it is also Church random for~$I$, and therefore $\sqgroup{p,q} \subseteq \churchI \subseteq I$.
%Moreover, we know from \propositionref{prop:bounds:outer} that 
%$\pth$ is (almost) completely Martin-Löf random for the interval forecast~$\sqgroup{p,q}$.
%Hence, indeed, \(\churchI=\mlI= \intervalpq\), and therefore, by \corollaryref{cor:relations:interval}, \(\churchI=\compI=\weakmlI=\mlI= \intervalpq\).
By invoking \corollaryref{cor:relations:interval}, we infer that \(\churchI=\randomI= \intervalpq\).
\qed \

{\noindent\bfseries Proof of \theoremref{the:coinciding:temporal}\quad}
For ease of notation, let $p\coloneqq \liminf_{n \to \infty}\frcstsystem(\pthton)$ and~$q\coloneqq \limsup_{n \to \infty}\frcstsystem(\pthton)$.
Since $\pth$ is \random-random for the computable precise temporal forecasting system~$\frcstsystem$, with $\random \in \set{\ml,\wml,\co,\s,\ch,\wch}$, $\pth$ is also \wch-random for~$\frcstsystem$ by \equationref{eq:inclusions:comp}.
Consequently, it follows from Propositions~\ref{prop:bounds:outer} and~\ref{prop:bounds:weakchurch:inner} that 
\begin{equation*}
\randomI \subseteq \sqgroup{p,q} \textrm{ and } \sqgroup{p,q} \subseteq \weakchurchI.
%\min \churchI \leq p \textrm{ and } q \leq \max \churchI.
\end{equation*}
%\begin{equation*}
%\min \weakchurchI \leq p \textrm{ and } q \leq \max \weakchurchI.
%\end{equation*}
%By \propositionref{prop:relations:interval}, it holds that if a path~$\pth$ is Martin-Löf random for an interval forecast~$I$, then it is also weakly Church random for~$I$, and therefore $\sqgroup{p,q} \subseteq \weakchurchI \subseteq I$.
%Moreover, we know from \propositionref{prop:bounds:ML} that $\pth$ is (almost) Martin-Löf random for the interval forecast~$\sqgroup{p,q}$.
%Hence, indeed, \(\weakchurchI=\mlI= \intervalpq\), and therefore, by \corollaryref{cor:relations:interval}, \(\weakchurchI=\churchI=\compI=\schnorrI=\weakmlI=\mlI= \intervalpq\).
By invoking \corollaryref{cor:relations:interval}, we infer that \(\weakchurchI=\randomI=\intervalpq\).
\qed \

\section*{Proofs and additional material for Section~\ref{sec:unalike}}

{\noindent\bfseries Proof of \propositionref{prop:counterexample1}\quad}
Yongge Wang proved the existence of a path~\(\pth \in \pths\) and a recursive rational test supermartingale~\(\test \in \testscomp{\nicefrac{1}{2}}\) such that \cite{Wang1996}
\begin{enumerate}[label=\upshape(\roman*),leftmargin=*,noitemsep]
\item $\nicefrac{1}{2} \in \schnorrintervals$;
\item\label{it:wang:unbounded} \(\test\) is unbounded---but not computably so---on~\(\pth\), also implying that \(\pth\) is not computably random for~$\nicefrac{1}{2}$;
\item\label{it:wang:martingale} for all~\(\sit\in\sits\), either \(\group{\forall x\in\outcomes}\test(\sit x)=2x\test(\sit)\) or \(\group{\forall x\in\outcomes}\test(\sit x)=\test(\sit)\);
\end{enumerate}
and as a consequence also
\begin{enumerate}[label=\upshape(\roman*),leftmargin=*,noitemsep,resume]
\item\label{it:wang:martingale:on:path} if \(\test(\pthtonplus)=2\pthatnplus\test(\pthton)\) then \(\pthatnplus=1\), for all~\(n\in\naturalswithzero\).
\end{enumerate}
We will prove that $\pth$ is exactly the path we are after.
One immediate conclusion we can draw from the above conditions, is that \(\test\) remains positive on~\(\pth\), so \(\test(\pthto{n})>0\) for all \(n\in\naturalswithzero\), simply because \ref{it:wang:martingale} implies that if \(\test\) ever becomes zero, it remains zero.
It can therefore then never become unbounded on~\(\pth\), contradicting~\ref{it:wang:unbounded}.
From \ref{it:wang:martingale} and~\ref{it:wang:martingale:on:path}, it then follows that, for every~$n \in \naturalswithzero$, $\test(\pthtonplus)=\test(\pthton)>0$ or $\test(\pthtonplus)=\test(\pthton 1)=2 \test(\pthton)>0$.
%From \ref{it:wang:martingale} and~\ref{it:wang:martingale:on:path}, it then follows that for every~$n \in \naturalswithzero$, there is some~$m_n \in \naturalswithzero$ such that $\test(\pthton)=2^{m_n}$, with $m_{n+1} = m_n$ or $m_{n+1} = m_n + 1$.
%Clearly, by \ref{it:wang:martingale}, if $m_{n+1}=m_n+1$, then $\pthatnplus=1$.
%Moreover, since $\limsup_{n \to \infty}\test(\pthton)=\infty$, it holds that $\lim_{n \to \infty} m_n=\infty$.

To show that $\schnorrI=\nicefrac{1}{2}$, we just observe that since $\nicefrac{1}{2} \in \schnorrintervals$ and since $\schnorrI$ is non-empty by \propertyref{prop:law:large:numbers}, it immediately follows that $\schnorrI=\bigcap \schnorrintervals = \nicefrac{1}{2}$.

We continue by showing that $\sqgroup{\nicefrac{1}{2},1} \subseteq\compI$.
By \corollaryref{cor:relations:interval}, it suffices to show that $\sqgroup{\nicefrac{1}{2},1} \subseteq\churchI$.
Since $\schnorrI=\nicefrac{1}{2}$, it follows from \corollaryref{cor:relations:interval} that $\weakchurchI=\nicefrac{1}{2}\neq \emptyset$, and therefore, again by \corollaryref{cor:relations:interval}, $\min \churchI \leq \nicefrac{1}{2}$.

We end by showing that $\max \churchI=1$.
To this end, consider the selection process~$\selection$ defined by 
\begin{equation*}
\selection(\sit)
\coloneqq
\begin{cases}
1 &\textrm{if \(\test(\sit1)=2\test(\sit)\)} \\
0 &\textrm{otherwise}
\end{cases}
\quad\text{for all~\(\sit\in\sits\)}.
\end{equation*}
Since $\test$ is recursive, the above equality can be checked in a recursive way, and therefore, $\selection$ is recursive.
From the above discussion, it follows that if $\selection(\pthton)~=~1$, then $\pthatnplus=1$ for all~$n \in \naturalswithzero$.
Since $\limsup_{n \to \infty}\test(\pthton)=\infty$, it follows from~\ref{it:wang:unbounded} that the first multiplication rule in~\ref{it:wang:martingale} must apply an infinite number of times on~\(\pth\). 
Consequently, \(\selection\) selects an infinite subsequence of ones from~\(\pth\), so the corresponding sequence of relative frequencies on this recursively selected infinite subsequence converges to~\(1\). 
By \definitionref{def:churchrandom} and \propositionref{prop:churches}, it now holds that $\max \churchI =1$.
\qed \
\begin{comment}
\begin{equation*}
\multprocess_\delta(\sit)(x)\coloneqq \begin{cases}
\frac{x}{1-\delta} &\textrm{if } \multprocess_\test(\sit)(x)=2x\\
1 &\textrm{otherwise}
\end{cases} \textrm{ for all } \sit \in \sits \textrm{ and } x \in \posspace.
\end{equation*}
\end{comment}

\begin{lemma} \label{lem:rec:enum:weakml:D}
For every rational interval forecast~$I \subseteq \group{0,1}$, there is a recursive enumeration of all lower semicomputable non-negative real supermartingale multipliers for~$I$.
%$\test \in \smash{\weakmltests{I}}$ with lower semicomputable multiplier processes.
\end{lemma}

\begin{proof}
By \corollaryref{cor:enumeration}, 
%The existence of a universal Turing machine implies that
there is a recursive enumeration of all lower semicomputable non-negative extended real maps $r \colon \sits \times \posspace \to \sqgroup{0,\infty}$. % \cite[Lemma 13]{Vovk2010}.
This means that there is some recursive rational map~$q\colon \naturalswithzero \times \sits \times \posspace \times \naturalswithzero \to \rationals$ such that 
\begin{enumerate}[label=\upshape(\roman*),leftmargin=*,noitemsep]
\item $q(i,\sit,x,n+1) \geq q(i,\sit,x,n)$ for all~$\sit \in \sits$, $x \in \posspace$ and~$ i,n \in \naturalswithzero$;
\item for every~$i \in \naturalswithzero$, the extended real map $r\colon \sits \times \posspace \to \sqgroup{0,+\infty}$, defined by $r(\sit,x) \coloneqq \lim_{n \to \infty}q(i,\sit,x,n)$ for all $\sit \in \sits$ and $x \in \posspace$, is lower semicomputable and non-negative;
\item for every lower semicomputable non-negative extended real map~$r\colon \sits \times \posspace \to \sqgroup{0,+\infty}$, there is an index $i \in \naturalswithzero$ such that $\lim_{n \to \infty}q(i,\sit,x,n)=r(\sit,x)$ for all~$\sit \in \sits$ and~$x \in \posspace$.
\end{enumerate}

\begin{comment}
\begin{equation*}
q(i,\sit,x,n) \leq r_{i,\sit,x,n+1} \textrm{ for all } \sit \in \sits \textrm{, } x \in \posspace \textrm{ and } i,n \in \naturalswithzero,
\end{equation*}
\end{comment}
\begin{comment}
\begin{equation*}
\lim_{n \to \infty}r_{i',\sit,x,n}=\multprocess(\sit)(x) \textrm{ for all } \sit \in \sits \textrm{, } x \in \posspace \textrm{ and } n \in \naturalswithzero.
\end{equation*}
\end{comment}
%Fix any~$K \in \posrationals$ such that $K \geq \max\{\nicefrac{1}{\max I},\nicefrac{1}{1-\min I}\}$.
%We can assume that $\leq q(i,\sit,x,n) \leq K$ for all~$\sit \in \sits$, $x \in \posspace$ and~$i,n \in \naturalswithzero$.
%Otherwise, we just consider the recursive net $\min\{K,\max\{0,q(i,\sit,x,n)\}\}$.
From the recursive rational map~$q$, we will construct a recursive rational map~$q'\colon\naturalswithzero \times \sits\times\posspace\times \naturalswithzero \to \rationals$ such that 
\begin{enumerate}[label=\upshape(\roman*'),leftmargin=*,noitemsep,ref=\upshape(\roman*')]
\item $q'(i,\sit,x,n+1) \geq q'(i,\sit,x,n)$ for all~$\sit \in \sits$, $x \in \posspace$ and~$i,n \in \naturalswithzero$; \label{step:universal:1a}
\item for every~$i \in \naturalswithzero$, the gamble process $\multprocess_i \colon \sits \to \gambles$, defined by $\multprocess_i(\sit)(x)\coloneqq \lim_{n \to \infty}q'(i,\sit,x,n)$ for all $\sit \in \sits$ and $x \in \posspace$, is a lower semicomputable non-negative real supermartingale multiplier for~$I$; \label{step:universal:2a}
\item for every lower semicomputable non-negative real supermartingale multiplier $\multprocess$ for~$I$, there is an index $i \in \naturalswithzero$ such that $\lim_{n \to \infty}q'(i,\sit,x,n)=\multprocess(\sit)(x)$ for all~$\sit \in \sits$ and~$x \in \posspace$. \label{step:universal:3a}
\end{enumerate}
To this end, consider the rational map~$q'$ defined by $q'(i,\sit,x,0)\coloneqq 0$ for all~$i \in \naturalswithzero$, $\sit \in \sits$ and~$x \in \posspace$, and by the recursion equation 
\begin{equation*}
q'(i,\sit,x,n+1)\coloneqq \begin{cases}
\max\{0,q(i,\sit,x,n+1)\} &\textrm{if } \uex_I(\max\{0,q(i,\sit,\bullet,n+1)\}) \leq 1 \\
q'(i,\sit,x,n) &\textrm{otherwise},
\end{cases}
\end{equation*}
for all~$\sit \in \sits$, $x \in \posspace$ and~$i,n \in \naturalswithzero$, with $\max\{0,q(i,\sit,\bullet,n+1)\}$ the pointwise maximum of the gambles $0$ and~$q(i,\sit,\bullet,n+1)$.
Since $q$ is a recursive rational map and since $I$ is a rational interval forecast, it follows from \equationref{eq:uex:2} that the inequality in the above expression is decidable for every~$\sit \in \sits$ and~$i,n \in \naturalswithzero$.
Combined with the fact that taking the maximum with respect to zero preserves recursiveness, it immediately follows from the above definition that $q'$ is a recursive map as well.
Since we take the maximum with respect to zero and since $q(i,\sit,x,0)=0$ for all~$i \in \naturalswithzero$, $\sit \in \sits$ and~$x \in \posspace$, it also follows that $q'$ is non-negative.

To prove \ref{step:universal:1a}, we fix any $\sit \in \sits$, $x \in \posspace$ and $i,n \in \naturalswithzero$.
Note that since $q(i,\sit,x,n)\leq q(i,\sit,x,n+1)$, it immediately follows that $\max\{0,q(i,\sit,x,n)\} \leq \max\{0,q(i,\sit,x,n+1)\}$.
Consequently, if $\uex_I(\max\{0,q(i,\sit,\bullet,n+1)\}) \leq 1$, then by \ref{prop:coherence:increasingness} also $\uex_I(\max\{0,q(i,\sit,\bullet,n)\}) \leq 1$.
We infer that if $\uex_I(\max\{0,q(i,\sit,\bullet,n+1)\}) \leq 1$, then 
\begin{equation*}
q'(i,\sit,x,n)=\max\{0,q(i,\sit,x,n)\} \leq \max\{0,q(i,\sit,x,n+1)\} = q'(i,\sit,x,n+1),
\end{equation*}
and if $\uex_I(\max\{0,q(i,\sit,\bullet,n+1)\}) > 1$, then $q'(i,\sit,x,n) = q'(i,\sit,x,n+1)$.
Therefore, \ref{step:universal:1a} indeed holds.
%$q'(i,\sit,x,n)\leq r'_{i,\sit,x,n+1}$ for all~$\sit \in \sits$, $x \in \posspace$ and~$i,n \in \naturalswithzero$.

To prove \ref{step:universal:2a}, we fix any~$i \in \naturalswithzero$, and consider the multiplier process~$\multprocess_i$ defined by $\multprocess_i(\sit)(x) \coloneqq \lim_{n \to \infty}q'(i,\sit,x,n)$ for all~$\sit \in \sits$ and~$x \in \posspace$.
Let's show that $\multprocess_i$ is a lower-semicomputable non-negative real supermartingale multiplier for~$I$.
First, since $q'$ is a recursive map and since $q'(i,\sit,x,n)\leq q'(i,\sit,x,n+1)$ for all~$\sit \in \sits$, $x \in \posspace$ and~$n \in \naturalswithzero$, it immediately follows that $\multprocess_i$ is well-defined (possibly infinite) and lower semicomputable.
Second, since $q'(i,\sit,x,n) \geq 0$ for all~$\sit \in \sits$, $x \in \posspace$ and~$n \in \naturalswithzero$, it follows that $\multprocess_i$ is non-negative.
Last, it holds by construction that $\uex_I(q'(i,\sit,\bullet,0))=\uex_I(0)=0$ and that 
\begin{align*}
\uex_I(q'(i,\sit,\bullet,n+1))
&=
\begin{cases}
\uex_I(\max\{0,q(i,\sit,\bullet,n+1)\}) \\
\quad \quad \quad \textrm{if } \uex_I(\max\{0,q(i,\sit,\bullet,n+1)\}) \leq 1, \\
\uex_I(q'(i,\sit,\bullet,n)) \\
\quad \quad \quad \textrm{otherwise},
\end{cases} \\
&\leq
\max\{1, \uex_I(q'(i,\sit,\bullet,n))\} \\%\hphantom{aaaaaaaaaaaa} (q'(i,s,\bullet,n) \geq 0)\\
&\leq \max\{1, \max\{1,\uex_I(q'(i,\sit,\bullet,n-1))\}\} \\
&=\max\{1,\uex_I(q'(i,\sit,\bullet,n-1))\} \\
&\leq ...
\leq \max\{1, \uex_I(q'(i,\sit,\bullet,0))\}
=
\max\{1,0\}
=
1,
\end{align*}
for all~$\sit \in \sits$ and~$n \in \naturalswithzero$.
Hence, for all~$\sit \in \sits$ and~$n \in \naturalswithzero$, $\uex_I(q'(i,\sit,\bullet,n)) \leq 1$.
Therefore, it follows from \lemmaref{lem:bounded:above} that for the rational number~$K \in \posrationals$, defined by $K \coloneqq\max\{\nicefrac{1}{\max I},\nicefrac{1}{(1-\min I)}\}$, it holds that $q'(i,\sit,x,n)\leq K$ for all~$\sit \in \sits$, $x \in \posspace$ and~$n \in \naturalswithzero$, and therefore, $\multprocess_i(\sit)(x)=\lim_{n \to \infty} q'(i,\sit,x,n) \leq K$ for all~$\sit \in \sits$ and~$x \in \posspace$.
By recalling that $\multprocess_i$ is non-negative, we conclude that $\multprocess_i$ is real.
Consequently, since $\posspace$ is a finite set, it follows from \ref{prop:coherence:uniformcontinuity} that $\uex_I(\multprocess_i(\sit))=\lim_{n \to \infty} \uex_I(q'(i,\sit,\bullet,n)) \leq 1$ for all~$\sit \in \sits$, and hence, $\multprocess_i$ is a supermartingale multiplier for~$I$.
We conclude that $\multprocess_i$ is a lower semicomputable non-negative real supermartingale multiplier for~$I$.
%Last, since $\uex_I(\multprocess_i(\sit))\leq 1$ for all~$\sit \in \sits$, it follows from $\lemmaref{lem:bounded:above}$ that $\multprocess_i(\sit)(1)\leq \nicefrac{1}{\max I}$ and~$\multprocess_i(\sit)(0)\leq \nicefrac{1}{1-\min I}$ for all~$\sit \in \sits$, and hence, the lower semicomputable non-negative supermartingale muliplier $\multprocess_i$ for~$I$ is real.

To prove \ref{step:universal:3a}, consider any lower semicomputable non-negative real supermartingale multiplier $\multprocess$ for~$I$.
Since $\multprocess$ is lower semicomputable, there is some index $i \in \naturalswithzero$ such that $\multprocess(\sit)(x)=\lim_{n \to \infty}q(i,\sit,x,n)$ for all~$\sit \in \sits$ and~$x \in \posspace$.
We intend to show that for the same index $i \in \naturalswithzero$, it holds that $\multprocess(\sit)(x)=\lim_{n \to \infty}q'(i,\sit,x,n)$ for all~$\sit \in \sits$ and~$x \in \posspace$.
To this end, fix any~$\sit \in \sits$.
Since $\multprocess$ is a non-negative real supermartingale multiplier for~$I$ and since $q(i,\sit,x,n) \leq \multprocess(\sit)(x)$ for all~$x \in \posspace$ and~$n \in \naturalswithzero$, it follows from \ref{prop:coherence:increasingness} that 
\begin{equation*}
\uex_I(\max\{0,q(i,\sit,\bullet,n)\}) \leq \uex_I(\max\{0,\multprocess(\sit)\}) = \uex_I(\multprocess(\sit)) \leq 1 \textrm{ for all } n \in \naturalswithzero.
\end{equation*} 
Consequently, $q'(i,\sit,x,n)=\max\{0,q(i,\sit,x,n)\}$ for all~$x \in \posspace$ and~$n \in \naturals$.
Since $\multprocess(\sit)(x)~\geq~0$ for all $x \in \posspace$, it immediately follows that
\begin{multline*}
\multprocess(\sit)(x)
=
\lim_{n\to\infty}q(i,\sit,x,n) \\
=
\lim_{n\to\infty}\max\{0,q(i,\sit,x,n)\}
=
\lim_{n\to\infty}q'(i,\sit,x,n) \textrm{ for all } x \in \posspace.
\end{multline*}

From \ref{step:universal:1a}--\ref{step:universal:3a} we conclude that $q'$ provides us with a recursive enumeration of all lower semicomputable non-negative real supermartingale multipliers for~$I$.
\qed
\end{proof}

\begin{comment}
\begin{lemma}[{\cite[Proposition 7]{CoomanBock2021}}] \label{lem:lowersemi:D:Md}
Consider any lower semicomputable non-negative real multiplier process~$\multprocess$.
Then $\mint$ is also lower semicomputable.
\end{lemma}
\end{comment}

\begin{corollary} \label{cor:rec:enum:weakml:T}
For every rational interval forecast~$I \subseteq \group{0,1}$, there is a recursive enumeration of all lower semicomputable test supermartingales $\mint \in \smash{\weakmltests{I}}$ generated by lower semicomputable supermartingale multipliers for $I$.
\end{corollary}

\begin{proof}
By \lemmaref{lem:rec:enum:weakml:D}, we know there is a recursive enumeration $\multprocess_i$, with $i \in \naturalswithzero$, of all lower semicomputable non-negative real supermartingale multipliers for I.
This means that there is a recursive map~$q \colon \naturalswithzero \times \sits \times \posspace \times \naturalswithzero \to \rationals$ such that 
\begin{enumerate}[label=\upshape(\roman*),leftmargin=*,noitemsep]
\item $q(i,\sit,x,n+1) \geq q(i,\sit,x,n)$ for all~$\sit \in \sits$, $x \in \posspace$ and~$i,n \in \naturalswithzero$; \label{step:universal:1}
\item for every~$i \in \naturalswithzero$, the gamble process $\multprocess_i\colon\sits \to \gambles$, defined by $\multprocess_i(\sit)(x) \coloneqq \lim_{n \to \infty}q(i,\sit,x,n)$ for all $\sit \in \sits$ and $x \in \posspace$, is a lower semicomputable non-negative real supermartingale multiplier for~$I$; \label{step:universal:2}
\item for every lower semicomputable non-negative real supermartingale multiplier $\multprocess$ for~$I$, there is an index $i \in \naturalswithzero$ such that $\lim_{n \to \infty}q(i,\sit,x,n)=\multprocess(\sit)(x)$ for all~$\sit \in \sits$ and~$x \in \posspace$. \label{step:universal:3}
\end{enumerate}

Consider the rational map~$q'\colon \naturalswithzero \times \sits \times \naturalswithzero \to \rationals$, defined by
\begin{equation*}
q'(i,\sit,n) \coloneqq \prod_{k=0}^{l-1}q(i,\xvalto{k},\xvalatkplus,n) \textrm{ for all } \sit=(x_1,\dots,x_l) \in \sits \textrm{ and } i,n \in \naturalswithzero.
\end{equation*}
Since $q$ is a recursive map, $q'$ is clearly a recursive map as well.
Moreover, since $q(i,\sit,x,n+1)\geq q(i,\sit,x,n)$ for all~$\sit \in \sits$, $x \in \posspace$ and~$i,n \in \naturalswithzero$, clearly also $q'(i,\sit,n+1)\geq q'(i,\sit,n+1)$ for all~$\sit \in \sits$ and~$i,n \in \naturalswithzero$.
For every~$i \in \naturalswithzero$, we consider the real process~$\mint_i$ defined by
\begin{equation}
\mint_i(\sit)= \prod_{k=0}^{l-1}\multprocess_i(\xvaltok)(\xvalatkplus) \textrm{ for all } \sit=(x_1,\dots,x_l) \in \sits.
\end{equation}
By \corollaryref{cor:from:D:to:M}, $\mint_i$ is a test supermartingale for~$I$.
It is clear that $\mint_i(\sit)=\lim_{n \to \infty} q'(i,\sit,n)$ for all~$i \in \naturalswithzero$ and~$\sit \in \sits$, so $\mint_i$ is also lower semicomputable.

We conclude that $(\mint_i)_{i \in \naturalswithzero}$ is a recursive enumeration of all lower semicomputable test supermartingales for~$I$ with lower semicomputable supermartingale multipliers for~$I$.
\qed
\end{proof}

\begin{lemma}[{\cite[Lemma 14]{floris2021}}] \label{lemma:14}
Consider any path \(\pth \in \pths\) and any interval forecast \(I \subset \group{0,1}\).
If \(\pth\) is recursive, then \(\pth\) is not \co-random for \(I\).
\end{lemma}

\begin{lemma} \label{lem:rec:enum:weakml:T}
Consider any interval forecast~$I \subseteq \group{0,1}$ and any recursive enumeration $\mint_i$, with $i \in \naturalswithzero$, of lower semicomputable test supermartingales for~$I$ generated by supermartingale multipliers $\multprocess$.
Then the process~$\test$, defined by
\begin{equation}
\test(\sit) \coloneqq \sum_{i=0}^\infty 2^{-i-1} \mint_i(\sit)  \textrm{ for all } \sit \in \sits,
\end{equation}
is a lower semicomputable test supermartingale for~$I$.
\end{lemma}

\begin{proof}
Since $(\mint_i)_{i \in \naturalswithzero}$ is a recursive enumeration of lower semicomputable test supermartingales for~$I$, there is some recursive rational map~$q \colon \naturalswithzero \times \sits \times \naturalswithzero \to \rationals$ such that $q(i,\sit,n+1) \geq q(i,\sit,n)$ and~$\lim_{m \to \infty}q(i,\sit,m)=\mint_i(\sit)$ for all~$\sit \in \sits$ and~$i,n \in \naturalswithzero$.
We can assume without loss of generality that $q$ is non-negative. 
Otherwise, we just consider $\max\{0,q\}$ instead.

First, since $\mint_i(\init)=1$ for all~$i \in \naturalswithzero$, it follows that $\test(\init)=\sum_{i=0}^\infty 2^{-i-1}=1$.
Second, since $\multprocess_i(\sit)\geq0$ and~$\uex_I(\multprocess_i(\sit))\leq 1$ for all~$\sit \in \sits$, and since $I \subseteq \group{0,1}$, it follows from \lemmaref{lem:bounded:above} that $\multprocess_i(\sit) \leq \max\{\nicefrac{1}{\max I},\nicefrac{1}{(1-\min I)}\}$ for all~$\sit \in \sits$.
Let $K \coloneqq \max\{\nicefrac{1}{\max I},\nicefrac{1}{(1-\min I)}\}$.
Consequently, 
\begin{align*}
\mint_i(\sit)
=
\prod_{k=0}^{l-1} \multprocess_i(\xvaltok)(\xvalatkplus)
\leq
K^{l} \textrm{ for all } \sit=(x_1,\dots,x_l) \in \sits \textrm{ and } i \in \naturalswithzero  ,
\shortintertext{and hence,}
\test(\sit)
=
\sum_{i=0}^\infty 2^{-i-1} \mint_i(\sit)
\leq
\sum_{i=0}^\infty 2^{-i-1} K^{\abs{\sit}}
=
K^{\abs{\sit}}
\textrm{ for all } \sit \in \sits.
\end{align*}
Since, for every~$\sit \in \sits$, $\test(\sit)$ is an infinite sum of non-negative terms that is bounded above by $K^{\abs{\sit}}$, it follows that $\test$ is well-defined, real and non-negative.
Third, since 
\begin{equation*}
\adddelta \test(\sit) = \lim_{n \to \infty} \sum_{i=0}^{n}2^{-i-1}\adddelta \mint_i(\sit) \textrm{ for all } \sit \in \sits,
\end{equation*}
and since $\posspace$ is a finite set, it follows from \ref{prop:coherence:uniformcontinuity} that
\begin{align*}
\uex_I(\adddelta \test(\sit)) 
&=
\lim_{n \to \infty} \uex_I(\sum_{i=0}^{n}2^{-i-1}\adddelta \mint_i(\sit)) \\
\overset{\textrm{\ref{axiom:coherence:homogeneity}}-\textrm{\ref{axiom:coherence:subsupadditivity}}}&{\leq}
\limsup_{n \to \infty} \sum_{i=0}^{n}2^{-i-1} \uex_I(\adddelta \mint_i(\sit)) \\
&\leq
\limsup_{n \to \infty} \sum_{i=0}^{n} 2^{-i-1} 0
=
0
\textrm{ for all } \sit \in \sits,
\end{align*}
and therefore $\test$ is a supermartingale for~$I$.
Last, in order to prove that $\test$ is lower semicomputable, consider the rational map~$q^\ast\colon \sits \times \naturalswithzero \to \rationals$ defined by
\begin{equation*}
q^\ast(\sit,n) \coloneqq \sum_{i=0}^{n} 2^{-i-1} q(i,\sit,n) \textrm{ for all } \sit \in \sits \textrm{ and } n \in \naturalswithzero.
\end{equation*}
First, since $q$ is a recursive map and since taking the above finite weighted sum is a recursive operation, the map~$q^\ast$ is recursive as well.
Second, since $q(i,\sit,n)\geq 0$ and~$q(i,\sit,n+1) \geq q(i,\sit,n)$ for all~$\sit \in \sits$ and~$i,n \in \naturalswithzero$, it follows that
\begin{multline*}
q^\ast(\sit,n+1)
=
\sum_{i=0}^{n+1} 2^{-i-1} q(i,\sit,n+1) \\
\geq
\sum_{i=0}^{n} 2^{-i-1} q(i,\sit,n+1)
\geq
\sum_{i=0}^{n} 2^{-i-1} q(i,\sit,n)
=
q^\ast(\sit,n)
\end{multline*}
for all~$\sit \in \sits$ and~$n \in \naturalswithzero$.
Last, we will show that $\lim_{n \to \infty}q^\ast(\sit,n)=\test(\sit)$ for all~$\sit \in \sits$.
To this end, note that since $q(i,\sit,n)\leq \mint_i(\sit)$ for all~$\sit \in \sits$ and~$i,n \in \naturalswithzero$, it holds that
\begin{multline*}
q^\ast(\sit,n)
=
\sum_{i=0}^{n} 2^{-i-1} q(i,\sit,n)
\leq
\sum_{i=0}^{n} 2^{-i-1} \mint_i(\sit)
\leq
\sum_{i=0}^{\infty} 2^{-i-1} \mint_i(\sit)
=
\test(\sit),
\end{multline*}
for all~$\sit \in \sits$ and~$n \in \naturalswithzero$.
Thus, since $q^\ast(\sit ,n)$ is an increasing sequence in $n$ and since $q^\ast(\sit ,n)$ is bounded above by $\test(\sit)$ for all~$\sit \in \sits$ and~$n \in \naturalswithzero$, $\lim_{n \to \infty}q^\ast(\sit ,n)$ is well-defined and bounded above by $\test(\sit)$ for all~$\sit \in \sits$.
Fix any~$\sit \in \sits$, any~$\epsilon>0$ and any~$N \in \naturalswithzero$ such that $K^{\abs{\sit}}2^{-N}<\nicefrac{\epsilon}{4}$.
By recalling that $\mint_i(\sit) \leq K^{\abs{\sit}}$ for all~$i \in \naturalswithzero$, we infer that
\begin{equation*}
\frac{\mint_i(\sit)}{2^{N}}
\leq
\frac{K^{\abs{\sit}}}{2^{N}}
<
\frac{\epsilon}{4} \textrm{ for all } i \in \naturalswithzero.
\end{equation*}
Consequently,
\begin{align}
\sum_{i=N+1}^\infty 2^{-i-1} \mint_i(\sit)
&=
\sum_{i=N+1}^\infty 2^{-i-1+N} \frac{\mint_i(\sit)}{2^{N}} \nonumber \\
&\leq
\sum_{i=N+1}^\infty 2^{-i-1+N} \frac{\epsilon}{4}
=
\sum_{i=0}^\infty 2^{-i} \frac{\epsilon}{4}
=
\frac{\epsilon}{2}. \label{eq:part:1}
\end{align}
Since $q(i,\sit,n+1) \geq q(i,\sit,n) \geq 0$ and~$\mint_i(\sit)=\lim_{m \to \infty} q(i,\sit,m)$ for all~$i,n \in \naturalswithzero$, there is for every~$i \leq N$ some~$N_i \in \naturalswithzero$ such that
\begin{equation*}
0 \leq \mint_i(\sit)- q(i,\sit,N_i) < \frac{\epsilon}{4} \textrm{ for every } i \leq N.
\end{equation*}
Consequently,
\begin{equation}
\sum_{i=0}^{N} 2^{-i-1}(\mint_i(\sit)-q(i,\sit,N_i))
\leq
\sum_{i=0}^{N} 2^{-i-1} \frac{\epsilon}{4}
\leq
\sum_{i=0}^{\infty} 2^{-i-1} \frac{\epsilon}{4}
=
\frac{\epsilon}{2}. \label{eq:part:2}
\end{equation}
Let $N_\textrm{max}\coloneqq\max\{N,N_1,N_2,\dots,N_N\}$.
Then by \equationref{eq:part:1} and \eqref{eq:part:2} for all~$n \geq N_\textrm{max}$
\begin{align*}
\test(\sit)-q^\ast_{\sit,n}
&=
\sum_{i=0}^\infty 2^{-i-1} \mint_i(\sit)-\sum_{i=0}^{n} 2^{-i-1} q(i,\sit,n) \\
&=
\sum_{i=0}^{N} 2^{-i-1} \mint_i(\sit) + \sum_{i=N+1}^\infty 2^{-i-1} \mint_i(\sit) - \sum_{i=0}^{n} 2^{-i-1} q(i,\sit,n) \\
&\leq
\sum_{i=0}^{N} 2^{-i-1} \mint_i(\sit) + \sum_{i=N+1}^\infty 2^{-i-1} \mint_i(\sit) - \sum_{i=0}^{N_\textrm{max}} 2^{-i-1} q(i,\sit,N_\textrm{max}) \\
&\leq
\sum_{i=0}^{N} 2^{-i-1} \mint_i(\sit) + \sum_{i=N+1}^\infty 2^{-i-1} \mint_i(\sit) - \sum_{i=0}^{N_\textrm{max}} 2^{-i-1} q(i,\sit,N_i) \\
&\leq
\sum_{i=0}^{N} 2^{-i-1} \mint_i(\sit) + \sum_{i=N+1}^\infty 2^{-i-1} \mint_i(\sit) - \sum_{i=0}^{N} 2^{-i-1} q(i,\sit,N_i) \\
&=
\sum_{i=0}^{N} 2^{-i-1} (\mint_i(\sit)-q(i,\sit,N_i)) + \sum_{i=N+1}^\infty 2^{-i-1} \mint_i(\sit) \\
\overset{\eqref{eq:part:1}-\eqref{eq:part:2}}&{\leq}
\frac{\epsilon}{2}+\frac{\epsilon}{2}
=
\epsilon.
\end{align*}
Since this is true for every~$\epsilon>0$, it indeed holds that $\lim_{n \to \infty}q^\ast(\sit,n)=\test(\sit)$.
We conclude that $\test$ is a lower semicomputable test supermartingale for~$I$.
\qed
\end{proof}

\begin{lemma} \label{lem:ml:rec}
For every rational interval forecast~$I \subseteq \group{0,1}$, there is a lower semicomputable test supermartingale~$\test \in \testsml{I}$ such that $\limsup_{n \to \infty} \test(\pthton)=\infty$ for every path~$\pth \in \pths$ that is not \wml-random for~$I$.
\end{lemma}

\begin{proof}
By invoking \corollaryref{cor:rec:enum:weakml:T}, let $\mint_i$, with $i \in \naturalswithzero$, be a recursive enumeration of all lower semicomputable test supermartingales for~$I$ generated by lower semicomputable supermartingale multipliers for $I$.
By \lemmaref{lem:rec:enum:weakml:T}, the process~$\test$ defined by
\begin{equation}
\test(\sit) \coloneqq \sum_{i=0}^\infty 2^{-i-1} \mint_i(\sit)  \textrm{ for all } \sit \in \sits,
\end{equation}
is a lower semicomputable test supermartingale for~$I$.

\begin{comment}
Consequently, there is some recursive net $q(i,\sit,n)$ such that $r_{i,\sit,n+1} \geq q(i,\sit,n)$ and~$\lim_{n \to \infty}q(i,\sit,n)=\mint_i(\sit)$ for all~$\sit \in \sits$ and~$i,n \in \naturalswithzero$.
We can assume without loss of generality that $q(i,\sit,n)$ is non-negative. 
Otherwise, we just consider $\max\{0,q(i,\sit,n)\}$ instead.

Consider the process~$\test$ defined by
\begin{equation*}
\test(\sit)\coloneqq \sum_{i=0}^\infty 2^{-i-1} \mint_i(\sit) \textrm{ for all } \sit \in \sits.
\end{equation*}
We intend to show that $\test$ is the lower semicomputable non-negative real test supermartingale for~$I$ we are looking for.
\end{comment}

Consider a path~$\pth \in \pths$ that is not \wml-random for~$I$.
By \definitionref{def:notionsofrandomness}, there is some~$i' \in \naturalswithzero$ such that $\limsup_{n \to \infty} \mint_{i'}(\pthton)=\infty$.
%Since $I \subseteq \group{0,1}$, it holds by \lemmaref{lemma:2} that $\pth$ is not computably random for~$I$, and hence, $\pth$ is not weak Martin-Löf random for~$I$.
%By \definitionref{def:weakML}, there is some lower semicomputable non-negative real supermartingale multiplier $\multprocess$ for~$I$ for which $\limsup_{n \to \infty}\mint(\pthton)=\infty$.
%Consequently, there is some~$i' \in \naturals$ such that $\mint(\sit)=\mint_{i'}(\sit)$ for all~$\sit \in \sits$.
Consequently, it also holds that $\limsup_{n \to \infty}2^{-i'-1}\mint_{i'}(\pthton)=\infty$, and hence, since all~$\mint_i$ are non-negative,
\begin{align*}
\limsup_{n \to \infty} \test(\pthton)
=
\limsup_{n \to \infty} \sum_{i=0}^{\infty} 2^{-i-1}\mint_i(\pthton)
\geq
\limsup_{n \to \infty} 2^{-i'-1}\mint_{i'}(\pthton)
=
\infty.
\end{align*}
\qed
\end{proof}

For the following lemma and theorem, we drew inspiration from \cite{Schnorr1971}, in which Schnorr shows there is a difference between Martin-Löf randomness and its conjugate notion in terms of upper semicomputable test supermartingales.

\begin{lemma} \label{lem:recursive:path}
For every recursive positive rational supermartingale $\supermartin$ %\in \supermartinscomp{\nicefrac{1}{2}}$
for~$\nicefrac{1}{2}$, every positive real $y \in \posreals$ and every situation~$\sit \in \sits$ for which $\supermartin(t)\leq y$ for all~$t \precedes \sit$, there is a recursive path~$\pth \in \pths$ such that $\pthto{\abs{\sit}}=\sit$ and~$\supermartin(\pthton) \leq y$ for all~$n \in \naturalswithzero$.
\end{lemma}

\begin{proof}
The path~$\pth \in \pths$ will be constructed recursively by an induction argument.
First, we put $\pthto{\abs{\sit}} \coloneqq \sit$.
Now, assume that $\pthton$ has been constructed recursively such that $\supermartin(\pthto{i}) \leq y$ for all~$i \leq n$; note that this assumption holds trivially for~$n\leq\abs{\sit}$.
By \corollaryref{cor:from:M:to:D}, we know that $\multprocess_\supermartin$ is a recursive positive rational supermartingale multiplier for~$\nicefrac{1}{2}$.
Since $\multprocess_\supermartin$ is a supermartingale multiplier for~$\nicefrac{1}{2}$, there is an $x \in \posspace$ for which $\multprocess_\supermartin(\pthton)(x) \leq 1$.
Otherwise, $\multprocess_\supermartin(\pthton)(x)>1$ for all~$x \in \posspace$, and therefore $\ex_{\nicefrac{1}{2}}(\multprocess_\supermartin(\pthton))=\frac{1}{2}\multprocess_\supermartin(\pthton)(1)+\frac{1}{2}\multprocess_\supermartin(\pthton)(0) >1$, a contradiction.
Since $\multprocess_\supermartin$ is a recursive rational multiplier process, we can recursively determine an $x \in \posspace$ for which $\multprocess_\supermartin(\pthton)(x) \leq 1$.
We fix such an $x \in \posspace$ and put $\pthtonplus\coloneqq \pthton x$.
Clearly, $\supermartin(\pthtonplus) = \supermartin(\pthton)\multprocess_\supermartin(\pthton)(\pthatnplus) \leq \supermartin(\pthton) \leq y$, and hence, $\supermartin(\pthto{i}) \leq y$ for all~$i \leq n+1$.
\qed
\end{proof}

\begin{proposition} \label{prop:diff:ml:comp}
For any rational interval forecast~$I \subseteq \group{0,1}$, there is a path~$\pth \in \pths$ that is \co-random for~$\nicefrac{1}{2}$, but not \ml-random for~$I$.
\end{proposition}

\begin{proof}
Fix any rational interval forecast~$I \subseteq \group{0,1}$.
Due to \lemmaref{lem:ml:rec}, there is a lower semicomputable test supermartingale~$\test^\ast \in \overline{\mathbb{T}}_\ml(I)$ such that $\test^\ast$ is unbounded on every path~$\pth \in \pths$ that is not \wml-random for~$I$.
% be a lower semicomputable test supermartingale for~$\sqgroup{\delta,1-\delta}$ as in \lemmaref{lem:ml:rec}, and 
Let $(\test_i)_{i \in \naturalswithzero}$ be an enumeration (not necessarly recursive) of all recursive positive rational test supermartingales for~$\nicefrac{1}{2}$.
By \definitionref{def:notionsofrandomness} and \propositionref{prop:equivalence:compandrec}, it suffices to define a path~$\pth \in \pths$ for which $\limsup_{n \to \infty}\test^\ast(\pthton)=\infty$ and~$\limsup_{n \to \infty}\test_i(\pthton)<\infty$ for all~$i \in \naturalswithzero$.
%but computably random for~$\nicefrac{1}{2}$.

We prove by induction that such a path exists.
for~$i=0$, we let $n_i\coloneqq 0$ and~$\pthto{n_i}\coloneqq \init$.
%The induction step is given by $i \in \naturalswithzero$.
Consider any~$i \in \naturalswithzero$ and assume that $\pthto{n_i}$, with $0=n_0 < n_1 < \dots < n_i \in \naturalswithzero$, is defined such that $\test^\ast(\pthto{n_i})\geq i$ and
\begin{equation}
\sum_{k=0}^i 2^{-n_k-k}\test_k (\pthton)
\leq
\sum_{k=0}^i 2^{-k} \textrm{ for all } n \leq n_i. \label{eq:inductionstep}
\end{equation}
Note that the induction hypothesis is trivially true for~$i=0$, since $\test^\ast(\init)=1 \geq 0$ and~$\test_0(\init)=1 \leq 1$.
We start by showing that the process~$\test'_i$ defined by
\begin{equation*}
\test'_i(\sit) \coloneqq \sum_{k=0}^i 2^{-n_k-k}\test_k (\sit) \textrm{ for all } \sit \in \sits
\end{equation*}
is a recursive positive rational supermartingale for~$\nicefrac{1}{2}$.
Since this is a finite weighted sum of recursive positive rational test supermartingales for~$\nicefrac{1}{2}$ with positive rational coefficients, it immediately follows that it is a recursive positive rational process.
Hence, we are left with proving the supermartingale property.
From \ref{axiom:coherence:homogeneity} and~\ref{axiom:coherence:subsupadditivity}, it follows, for all~$\sit \in \sits$, that 
\begin{align*}
\ex_{\nicefrac{1}{2}}(\adddelta \test'_i(\sit))
=
\ex_{\nicefrac{1}{2}}\Bigg(\sum_{k=0}^i 2^{-n_k-k}\adddelta \test_k (\sit)\Bigg)
\leq
\sum_{k=0}^i 2^{-n_k-k} \ex_{\nicefrac{1}{2}}(\adddelta \test_k (\sit))
\leq
0
,
\end{align*}
where the last inequality holds because all~$\test_i$, with $i \in \naturalswithzero$, are supermartingales for~$\nicefrac{1}{2}$.
We conclude that $\test'_i$ is a recursive positive rational supermartingale for~$\nicefrac{1}{2}$.
From \equationref{eq:inductionstep}, we know that $\test'_i(\pthton) \leq \sum_{k=0}^{i} 2^{-k}$ for all~$n \leq n_i$.
If we now invoke \lemmaref{lem:recursive:path}, we find that there is a recursive path~$\pth'$ such that $\pthto{n_i}'=\pthto{n_i}$ and~$\test'_i(\pthton')\leq \sum_{k=0}^i 2^{-k}$ for all~$n \in \naturalswithzero$.
Since $\pth'$ is recursive, it follows from \lemmaref{lemma:14} that it is not \co-random for~$I$, and hence, by \propositionref{prop:relations:interval}, it is also not \wml-random for~$I$.
Consequently, $\limsup_{ n \to \infty} \test^\ast(\pthton')=\infty$, and therefore, there is a natural number~$n_{i+1} > n_i$ such that $\smash{\test^\ast}(\pthto{n_{i+1}}') \geq i+1$.
We put $\pthto{n_{i+1}} \coloneqq \pthto{n_{i+1}}'$.
Note that $\pthto{n_i}$ is unchanged.
By construction, it now holds that
\begin{multline}
\sum_{k=0}^i 2^{-n_k-k}\test_k (\pthton)
=
\sum_{k=0}^i 2^{-n_k-k}\test_k (\pthton') \\
=
\test_i'(\pthton') 
\leq
\sum_{k=0}^i 2^{-k}
\textrm{ for all } n \leq n_{i+1}. \label{eq:induction:part}
\end{multline}
Consider the recursive positive rational test supermartingale $\test_{i+1} \in \testscomp{\nicefrac{1}{2}}$.
From \lemmaref{lem:toandback:M:D}, it follows that $\multprocess_{\test_{i+1}}(\sit)>0$ and~$\ex_{\nicefrac{1}{2}}(\multprocess_{\test_{i+1}}(\sit))\leq1$ for all~$\sit \in \sits$, and hence, by \lemmaref{lem:bounded:above}, $0<\multprocess_{\test_{i+1}}(\sit) \leq 2$.
Consequently, $\test_{i+1}(\pthton)=\prod_{k=0}^{n-1}\multprocess_{\test_{i+1}}(\pthtok)(\pthatkplus)\leq 2^n$ for all~$n \in \naturalswithzero$, and hence,
\begin{equation*}
2^{-i-1}\frac{\test_{i+1}(\pthton)}{2^{n_{i+1}}} \leq 2^{-i-1} \textrm{ for all } n \leq n_{i+1}.
\end{equation*}
Therefore, and by \equationref{eq:induction:part}, it holds that
\begin{equation*}
\sum_{k=0}^{i+1} 2^{-n_k-k}\test_k (\pthton) \leq \sum_{k=0}^{i+1} 2^{-k} \textrm{ for all } n \leq n_{i+1}.
\end{equation*}
Since also $\test^\ast(\pthto{n_{i+1}}) \geq i+1$, this ends the induction step.

By observing that $0=n_0<n_1 < \dots < n_i < n_{i+1} < \dots$, the definition of $\pth$ implies that 
\begin{equation*}
\limsup_{n \to \infty}\test^\ast(\pthton)
=
\limsup_{i \to \infty}\test^\ast(\pthto{n_i})
\geq
\limsup_{i \to \infty} i
=
\infty.
\end{equation*}
Since $\test^\ast$ is a lower semicomputable test supermartingale for~$I$, it follows from \definitionref{def:notionsofrandomness} that $\pth$ is not \ml-random for~$I$.
%Since $y$ is recursive, it follows from \ref{lem:ml:rec} that there is an $n_{i+1}>n_i$ such that $\test^\ast(y_{1:n_{i+1}})>i+1$.
%Define $\pthto{n_i+1} \coloneqq y_{1:n_{i+1}}$.
%Note that $\lim_{i \to \infty}n_i=\infty$.
%The definition of $\pth$ implies that $\limsup_{k \to \infty} \test^\ast(\pthtok)=\infty$, and hence, by \definitionref{def:ML}, $\pth$ is not Martin-Löf random for~$\sqgroup{\delta,1-\delta}$.

On the other hand, by \equationref{eq:inductionstep} and since $0=n_0<n_1 < \dots < n_i < n_{i+1} < \dots$, there is for every~$n \in \naturalswithzero$ some~$i \in \naturalswithzero$ such that $n \leq n_i$ and
\begin{align*}
%\sum_{k=0}^{\infty} 2^{-n_k-k}\test_k(\pthton)
%=
\sum_{k=0}^{j} 2^{-n_k-k}\test_k(\pthton)
\leq
\sum_{k=0}^{j} 2^{-k}
%\leq
%\sum_{k=0}^{\infty} 2^{-k}
%\leq 2
\textrm{ for all } j \geq i.
\end{align*}
Consequently,
\begin{align*}
\sum_{k=0}^{\infty} 2^{-n_k-k}\test_k(\pthton)
\leq
\sum_{k=0}^{\infty} 2^{-k}
\leq 2.
\end{align*}
Since this is true for every~$n \in \naturalswithzero$, it follows that
\begin{align*}
\limsup_{n \to \infty} \sum_{k=0}^{\infty} 2^{-n_k-k}\test_k(\pthton)
\leq
2
.
\end{align*}
\begin{comment}
\begin{align*}
\limsup_{n \to \infty} \sum_{k=0}^{\infty} 2^{-n_k-k}\test_k(\pthton)
&\leq
\limsup_{n \to \infty} \limsup_{i \to \infty} \sum_{k=0}^{i} 2^{-n_k-k}\test_k(\pthton) \\
&=
\limsup_{i \to \infty} \sum_{k=0}^{i} 2^{-n_k-k}\test_k(\pthto{n_i}) \\
&\leq
\limsup_{i \to \infty} \sum_{k=0}^{i}2^{-k} 
=
2
\end{align*}
\end{comment}
This shows that every recursive positive rational test supermartingale $\test_i$ for~$\nicefrac{1}{2}$ remains bounded on~$\pth$.
Indeed, fix any~$i \in \naturalswithzero$ and assume \emph{ex absurdo} that $\limsup_{n \to \infty} \test_i(\pthton)=\infty$.
Consequently, 
\begin{equation*}
\limsup_{n \to \infty} 2^{-n_i-i}\test_i(\pthton)=\infty,
\end{equation*}
and hence, since all~$\test_k$ are positive,
\begin{align*}
\limsup_{n \to \infty} \sum_{k=0}^{\infty} 2^{-n_k-k}\test_k(\pthton)
\geq
\limsup_{n \to \infty} 2^{-n_i-i}\test_i(\pthton)
=
\infty
,
\end{align*}
a contradiction.
Hence, $\limsup_{n \to \infty}\test_i(\pthton)<\infty$ for all~$i \in \naturalswithzero$.
%It follows from \propositionref{prop:equivalence:compandrec} that $\pth$ is computably random for~$\nicefrac{1}{2}$.
\qed
\end{proof}

\begin{comment}

\end{comment}

{\noindent\bfseries Proof of \propositionref{prop:counterexample2}\quad}
Fix any $\delta \in \group{0,\nicefrac{1}{2}}$.
Consider any rational number~$\delta'$ such that $\delta<\delta'<\nicefrac{1}{2}$, and the interval forecasts~$I_{\delta}=\sqgroup{\nicefrac{1}{2}-\delta,\nicefrac{1}{2}+\delta}$ and~$I_{\delta'}=\sqgroup{\nicefrac{1}{2}-\delta',\nicefrac{1}{2}+\delta'}$.
By \propositionref{prop:diff:ml:comp}, there is a path~$\pth \in \pths$ that is \co-random for~$\nicefrac{1}{2}$ but not \ml-random for~$I_{\delta'}$.

To show that $\compI=\nicefrac{1}{2}$, we just observe that since $\nicefrac{1}{2} \in \compintervals$ and since $\compI$ is non-empty by \propertyref{prop:law:large:numbers}, it immediately follows that $\compI=\bigcap \compintervals = \nicefrac{1}{2}$.

We continue by showing that $I \notin \mlintervals$ for any $I \in \intervals$ such that $I \subseteq I_\delta$.
Assume \emph{ex absurdo} that $\pth$ is \ml-random for such an interval forecast $I$.
Since $I \subseteq I_\delta \subseteq I_{\delta'}$, it then holds by \propositionref{prop:ml:increasing} that $I_{\delta'} \in \mlintervals$, a contradiction.
%almost \ml-random for~$\nicefrac{1}{2}$.
%Consequently, $\pth$ is \ml-random for the interval forecast~$\sqgroup{\nicefrac{1}{2}-\nicefrac{\delta}{2},\nicefrac{1}{2}+\nicefrac{\delta}{2}}$.
%Since $\sqgroup{\nicefrac{1}{2}-\nicefrac{\delta}{2},\nicefrac{1}{2}+\nicefrac{\delta}{2}} \subseteq I$, it follows from \propositionref{prop:ml:increasing} that $I \in \mlintervals$, a contradiction.
\qed

\end{ArxiveExt}
\end{document}